\documentclass[aop,preprint]{imsart}

\usepackage{amsmath,amsthm,amssymb,mathtools}
\usepackage{mathrsfs,bm}
\numberwithin{equation}{section}
\allowdisplaybreaks[1]

\usepackage{graphicx,xcolor}
\usepackage{float}
\usepackage{placeins}

\usepackage[numbers,sort&compress]{natbib}
\usepackage{enumitem}
\setlist[enumerate,1]{label=\upshape(\arabic*)}
\usepackage[colorlinks=true,linkcolor=blue,citecolor=blue,urlcolor=blue,linktocpage=true]{hyperref}

\startlocaldefs

\theoremstyle{plain}
\newtheorem{thm}{Theorem}[section]
\newtheorem{lem}[thm]{Lemma}
\newtheorem{cor}[thm]{Corollary}
\newtheorem{prop}[thm]{Proposition}

\theoremstyle{definition}
\newtheorem{df}[thm]{Definition}
\newtheorem{ex}[thm]{Example}

\theoremstyle{definition}
\newtheorem{rmk}[thm]{Remark}

\newcommand{\C}{\mathbb{C}}
\newcommand{\E}{\mathbb{E}}
\newcommand{\F}{\mathbb{F}}

\newcommand{\I}{\mathbb{I}}

\newcommand{\N}{\mathbb{N}}
\renewcommand{\P}{\mathbb{P}}
\newcommand{\Q}{\mathbb{Q}}
\newcommand{\R}{\mathbb{R}}

\newcommand{\cA}{\mathcal{A}}

\newcommand{\cD}{\mathcal{D}}

\newcommand{\cF}{\mathcal{F}}

\newcommand{\cM}{\mathcal{M}}

\newcommand{\cP}{\mathcal{P}}

\newcommand{\cS}{\mathcal{S}}
\newcommand{\cT}{\mathcal{T}}

\newcommand{\cX}{\mathcal{X}}

\newcommand{\sfD}{\mathsf{D}}

\renewcommand{\a}{\alpha}
\renewcommand{\b}{\beta}
\newcommand{\gm}{\gamma}
\newcommand{\Gm}{\Gamma}
\newcommand{\dl}{\delta}
\newcommand{\Dl}{\Delta}
\newcommand{\eps}{\varepsilon}
\newcommand{\lm}{\lambda}

\renewcommand{\phi}{\varphi}
\newcommand{\sg}{\sigma}
\newcommand{\Sg}{\Sigma}

\newcommand{\Om}{\Omega}

\newcommand{\relmiddle}[1]{\mathrel{}\middle#1\mathrel{}}
\newcommand{\1}{\mbox{\rm1}\hspace{-0.25em}\mbox{\rm l}}
\newcommand{\la}{\langle}
\newcommand{\ra}{\rangle}
\newcommand{\del}{\partial}
\renewcommand{\tilde}{\widetilde}

\newcommand{\esa}{\mathrm{esa}}
\newcommand{\fin}{\mathrm{fin}}
\newcommand{\loc}{\mathrm{loc}}
\newcommand{\LW}{\mathrm{LW}}
\newcommand{\up}{\mathrm{up}}

\DeclareMathOperator{\Bernoulli}{Bernoulli}

\DeclareMathOperator{\dist}{dist}
\DeclareMathOperator{\Dgm}{Dgm}
\DeclareMathOperator{\Dom}{Dom}
\DeclareMathOperator{\Hom}{Hom}
\DeclareMathOperator{\im}{Im}
\DeclareMathOperator{\Ker}{Ker}
\DeclareMathOperator{\Law}{Law}

\DeclareMathOperator{\nullity}{nullity}
\DeclareMathOperator{\OT}{OT}
\DeclareMathOperator{\Pers}{Pers}
\DeclareMathOperator{\Poisson}{Poisson}
\DeclareMathOperator{\PGW}{PGW}

\DeclareMathOperator{\rank}{rank}
\DeclareMathOperator{\RSM}{RSM}
\DeclareMathOperator{\sgn}{sgn}

\DeclareMathOperator{\st}{st}
\DeclareMathOperator{\Supp}{Supp}
\DeclareMathOperator{\tr}{tr}

\makeatletter
\renewcommand\nocontentsline{%
  \let\@@addcontentsline\addcontentsline
  \def\addcontentsline##1##2##3{\let\addcontentsline\@@addcontentsline}%
}
\makeatother

\endlocaldefs

\begin{document}
\begin{frontmatter}
\title{Limit Theorems for Persistent Homology of Multiparameter Complex Processes}
\runtitle{Persistent Homology of Multiparameter Complex Processes}

\begin{aug}
\author[A]{\fnms{Masanori}~\snm{Hino}\ead[label=e1]{hino@math.kyoto-u.ac.jp}\orcid{0009-0004-8835-7115}}
\author[B]{\fnms{Shu}~\snm{Kanazawa}\ead[label=e2]{s.kanazawa@qmul.ac.uk}\orcid{0000-0002-3368-1629}}
\address[A]{Department of Mathematics, Graduate School of Science, Kyoto University\printead[presep={,\ }]{e1}}
\address[B]{School of Mathematical Sciences, Queen Mary University of London\printead[presep={,\ }]{e2}}
\end{aug}

\begin{abstract}
  Multiparameter complex processes, based on the Costa--Farber model, provide a common framework for growing random simplicial complexes, including the Linial--Meshulam and random flag complex processes.
  We establish laws of large numbers for persistent Betti numbers of these processes and prove that, after a deterministic time change and normalization, their persistence diagrams converge weakly in probability to a deterministic finite measure.
  Its quadrant masses are determined by universal limits of normalized persistent Betti numbers, characterized by an explicit finite-dimensional optimization problem.
  We extend the diagram convergence to integrals of continuous functions with polynomial growth in the death coordinate and controlled singularities at the diagonal.
  These results yield convergence of additive and multiplicative lifetime measures and total persistence.
  We further deduce convergence in optimal transport distance.
  The proofs combine local weak convergence of simplicial pairs, spectral analysis of compressed Laplacians accounting for dependence through shared faces, and quantitative estimates for large death coordinates and short lifetimes.
\end{abstract}


\end{frontmatter}

\tableofcontents

\section{Introduction}
\label{sec:intro}
Random simplicial complexes provide higher-dimensional analogues of random graphs and have been studied through the asymptotic behavior of their geometric and topological invariants.
Among the basic examples are the Linial--Meshulam complex~\cite{LM2006,MW2009} and the random flag complex~\cite{Kah2009,Kah2014}, in which randomness is introduced in a single dimension.
The multiparameter random complex model of Costa and Farber~\cite{CF2016a,CF2016b} places these and many other models in a common framework by allowing a separate inclusion probability in every dimension.

Betti numbers are among the most basic invariants describing the topology of an individual simplicial complex.
For an increasing family of complexes, however, one may also ask when holes appear and disappear as the complexes grow.
Informally, persistent Betti numbers count the holes present at one time that persist until a later time, while a persistence diagram records the birth and death times of holes~\cite{ELZ2002,ZC2005}.
By representing each birth--death pair as a point, a persistence diagram can be viewed as a multiset of points in the birth--death plane, or equivalently as a counting measure on this plane.
Persistent Betti numbers can then be recovered as the masses assigned by this measure to certain quadrants.

Earlier work on persistent homology of random simplicial complex processes includes asymptotic results for lifetime sums~\cite{HK2019,HS2017}.
For the Linial--Meshulam complex process with complete lower skeleton, all birth times in the relevant homological degree are zero, and the persistence diagram is therefore determined by the death times, for which limit theorems were established in~\cite{FMT2023,STY2020}.
More recently, M\'esz\'aros~\cite{Me2025} obtained limits for persistent Betti numbers and persistence diagrams in a modified Linial--Meshulam complex process in which top-dimensional simplices are added together with their faces.

In this paper, we establish limit theorems for persistent Betti numbers and persistence diagrams of the multiparameter complex process as the number of vertices tends to infinity.
Here, every dimension carries its own randomness: each simplex receives an independent weight whose distribution depends on its dimension alone, and a simplex appears once the weights of all its nonempty faces have been exceeded.
The Linial--Meshulam and random flag complex processes, in which randomness is confined to a single dimension, are special cases.
Our first result is a law of large numbers for persistent Betti numbers, whose limit we identify by maximizing an explicit function over the fixed points of an explicit map on a finite-dimensional cube.
Since these limits determine the quadrant masses of the diagrams, the suitably time-changed and normalized diagrams converge weakly in probability to a deterministic finite measure, and their mean measures converge weakly to the same limit.
We also prove convergence for unbounded test functions, allowing polynomial growth in the death coordinate and a controlled singularity at the diagonal.
Consequences include convergence of additive and multiplicative lifetime measures, total persistence, and convergence in the optimal transport distance.
Throughout, homology is reduced and is taken with coefficients in a fixed field of characteristic zero.

\subsection{Multiparameter Complex Processes}
Let $[n]:=\{1,2,\ldots,n\}$, and let $K(n):=2^{[n]}$ denote the complete simplicial complex on the vertex set $[n]$.
Given a multiparameter $\bm p_n=(p_j)_{j=0}^{n-1}\in[0,1]^n$, a random subcomplex of $K(n)$ is constructed inductively over dimension as follows.
Each vertex is first retained independently with probability $p_0$.
Having determined the $(j-1)$-skeleton, every $j$-simplex whose boundary is present is included independently with probability $p_j$.
The resulting random simplicial complex is called the \emph{multiparameter random complex} on $[n]$ with multiparameter $\bm p_n$, and its law is denoted by $X(n,\bm p_n)$.

For this model, define
\begin{equation}\label{eq:qk-rk}
\begin{gathered}
  q_{-1}:=1,
  \quad q_k:=\prod_{j=0}^k p_j^{\binom{k+1}{j+1}}
  \qquad(0\le k\le n-1),\\
  r_k:=\frac{q_{k+1}}{q_k}
  =\prod_{j=0}^{k+1}p_j^{\binom{k+1}j}
  \qquad(-1\le k\le n-2).
\end{gathered}
\end{equation}
Here and throughout, we adopt the conventions $0/0:=0$ and $\binom{0}{0}:=1$.
If $X_n\sim X(n,\bm p_n)$, then
\[
q_k=\P(\sg\in X_n)
\quad\text{and}\quad
r_k=\P(\tau\in X_n\mid\sg\in X_n)
\quad\text{whenever $q_k>0$}
\]
for every $k$-simplex $\sg\in K(n)$ and every $(k+1)$-simplex $\tau\in K(n)$ containing $\sg$.
Thus, the quantity
\[
\binom n{k+1}q_k=\E[f_k(X_n)]
\]
is the expected number of $k$-simplices in $X_n$, where $f_k(X_n)$ denotes the number of $k$-simplices in $X_n$.
Moreover, whenever $q_k>0$, the quantity $(n-k-1)r_k$ is the expected degree of a given $k$-simplex conditional on its presence in $X_n$.

Linial and Peled~\cite{LP2016a,LP2016b} obtained a law of large numbers for Betti numbers of the $d$-Linial--Meshulam complex $Y_d(n,p)$, which corresponds to the multiparameter random complex with $p_j=1$ for $j<d$, $p_d=p\in[0,1]$, and $p_j=0$ for $j>d$.
In~\cite{Kan2022}, the second author established the corresponding limit theorem for multiparameter random complexes.
For $k\in\N_0$ and $c\ge0$, define a continuous map $\Phi_{k,c}\colon[0,1]\to[0,1]$ by
\[
\Phi_{k,c}(u)
:=\exp\bigl(-c(1-u)^{k+1}\bigr),
\qquad u\in[0,1],
\]
and a continuous function $\Psi_{k,c}\colon[0,1]\to\R$ by
\[
\Psi_{k,c}(u)
:=u+c\,u(1-u)^{k+1}-\frac c{k+2}\bigl(1-(1-u)^{k+2}\bigr),
\qquad u\in[0,1].
\]
Set
\begin{equation}\label{eq:hk_c}
h_k(c)
:=\max\bigl\{\Psi_{k,c}(u)\bigm|u\in[0,1],\,u=\Phi_{k,c}(u)\bigr\}.
\end{equation}
The maximization in~\eqref{eq:hk_c} can be restricted to the least and greatest fixed points of $\Phi_{k,c}$; see~\cite[Lemma~4.5]{LP2016a}.
It was shown in~\cite{Kan2022} that, if $X_n\sim X(n,\bm p_n)$,
\[
\binom n{k+1}q_k\longrightarrow\infty
\quad\text{and}\quad
nr_k\longrightarrow c\in(0,\infty)
\quad\text{as $n\to\infty$,}
\]
then, for every $p\in[1,\infty)$,
\begin{equation}\label{eq:LLN_Betti}
\frac{\b_k(X_n)}{\binom n{k+1}q_k}
\xrightarrow{L^p}h_k(c).
\end{equation}
Here and throughout, $\xrightarrow{L^p}$ denotes convergence in $L^p$.

\smallskip
We next introduce the multiparameter complex process, whose one-time marginals are multiparameter random complexes.
Let
\[
\bm p
=\bigl(p_j\colon[0,\infty]\to[0,1]\bigr)_{j=0}^\infty
\]
be a sequence of distribution functions on $[0,\infty]$; that is, each $p_j$ is a nondecreasing and right-continuous function on $[0,\infty]$ satisfying $p_j(\infty)=1$.
We do not require $p_j(0)=0$ or $\lim_{t\to\infty}p_j(t)=1$, so the corresponding distributions may have atoms at $0$ and at $\infty$.
Assign to every nonempty simplex $\rho\in K(n)$ an independent random weight $w_\rho$ with distribution function $p_{\dim(\rho)}$, and set
\begin{equation}\label{eq:def_MCP}
X_n(t)
:=\{\sg\in K(n)\mid w_\rho\le t\text{ for every $\emptyset\ne\rho\subset\sg$}\},
\qquad t\ge0.
\end{equation}
The increasing family of random simplicial complexes
\[
\cX_n=\{X_n(t)\}_{t\ge0}
\]
is called the \emph{multiparameter complex process} on $[n]$ associated with $\bm p$~\cite{HK2019}.
Here, \emph{multiparameter} refers to the probability parameters in different dimensions; the process $\cX_n$ itself is indexed by the single time parameter $t$.
For $t\ge0$, let
\[
\bm p_n(t):=\bigl(p_j(t)\bigr)_{j=0}^{n-1}\in[0,1]^n.
\]
Then
\[
X_n(t)\sim X\bigl(n,\bm p_n(t)\bigr),
\qquad t\ge0.
\]
Thus, $t\mapsto\bm p_n(t)$ may be viewed as a right-continuous path in the parameter space $[0,1]^n$ of multiparameter random complex models on $[n]$.

Replacing $p_j$ by $p_j(t)$ in~\eqref{eq:qk-rk}, define
\begin{equation}\label{eq:qkt-rkt}
\begin{gathered}
  q_{-1}(t):=1,
  \quad q_k(t):=\prod_{j=0}^k p_j(t)^{\binom{k+1}{j+1}},\\
  r_k(t):=\frac{q_{k+1}(t)}{q_k(t)}
  =\prod_{j=0}^{k+1}p_j(t)^{\binom{k+1}j}.
\end{gathered}
\end{equation}
The quantities $q_k(t)$ and $r_k(t)$ have the same probabilistic interpretations as in the static model.
The functions $q_k$ and $r_k$ determine the normalization and the time scale, respectively, in the limit theorems below.

\subsection{Main Results}
Throughout this subsection, fix $k\in\N_0$ and a sequence $\bm p=(p_j)_{j=0}^\infty$ of distribution functions on $[0,\infty]$.
For each $n\in\N$, let $\cX_n=\{X_n(t)\}_{t\ge0}$ be the multiparameter complex process on $[n]$ associated with $\bm p$.

\subsubsection*{A law of large numbers for persistent Betti numbers}
For a simplicial pair $K\subset X$~(i.e., an inclusion pair of simplicial complexes), let
\[
\b_k(K\subset X)
:=\rank\bigl(H_k(K)\longrightarrow H_k(X)\bigr)
\]
denote its $k$th persistent Betti number, where the homomorphism on homology is induced by inclusion.
We refer to Section~\ref{ssec:H_PB_PD} for further details and the homological conventions used throughout the paper.

Our first main result is a law of large numbers for the persistent Betti numbers of the multiparameter complex process $\cX_n=\{X_n(t)\}_{t\ge0}$.
The limit in our first result involves, in addition to $h_k(c)$, a deterministic constant $h_k(c,\bm\theta)$ with $c\ge0$ and $\bm\theta=(\theta_j)_{j=0}^{k+1}\in[0,1]^{k+2}$.
As in~\eqref{eq:hk_c}, its definition involves maximizing a function over the fixed points of a map.
To formulate this optimization problem, let
\[
\cT_k
:=\bigl\{\xi\colon2^{[k+1]}\setminus\{\emptyset\}\to\{0,1\}\bigr\}
\]
be the \emph{type space} of the $k$-simplex $[k+1]=\{1,2,\ldots,k+1\}$.
Thus, each type $\xi\in\cT_k$ records a zero--one labeling on the nonempty faces of $[k+1]$.
No downward-closed condition is imposed on these labelings.
For $c\ge0$ and $\bm\theta=(\theta_j)_{j=0}^{k+1}\in[0,1]^{k+2}$, let
\[
\Phi_{k,c,\bm\theta}
\colon[0,1]^{\cT_k}\longrightarrow[0,1]^{\cT_k}
\quad\text{and}\quad
\Psi_{k,c,\bm\theta}
\colon[0,1]^{\cT_k}\longrightarrow\R
\]
denote the continuous map and function defined explicitly in Section~\ref{ssec:hk_main}, respectively.
We then set
\begin{equation}\label{eq:hk_c_theta}
h_k(c,\bm\theta)
:=\max\Bigl\{\Psi_{k,c,\bm\theta}(u)\Bigm|u\in[0,1]^{\cT_k},\,u=\Phi_{k,c,\bm\theta}(u)\Bigr\}.
\end{equation}
Thus, for each fixed $k$, the quantity $h_k(c,\bm\theta)$ is defined by a finite-dimensional optimization problem on a cube of dimension $|\cT_k|=2^{2^{k+1}-1}$.
In the multiparameter setting, the maximum in~\eqref{eq:hk_c_theta} may be attained at fixed points other than the least and greatest ones in the coordinatewise order, reflecting a richer fixed-point structure than in~\eqref{eq:hk_c}.
The two extremal cases reduce to
\begin{equation}\label{eq:hk_extremal}
h_k(c,\bm0)=h_k(c)
\quad\text{and}\quad
h_k(c,\bm1)=0.
\end{equation}
Further simplifications of $h_k(c,\bm\theta)$ under additional conditions on $c$ and $\bm\theta$ are given in Section~S2 of the Supplementary Material~\cite{HKsupp}.
The scalar reductions there are useful for numerical evaluation.

Our law of large numbers for persistent Betti numbers now takes the following form.

\begin{thm}\label{thm:LLN_PB}
  Let $c>0$.
  Suppose that $0\le s_n\le t_n<\infty$ satisfy
  \[
  \lim_{n\to\infty}\tbinom n{k+1}q_k(t_n)=\infty
  \quad\text{and}\quad
  \lim_{n\to\infty}nr_k(t_n)=c,
  \]
  and that the limits
  \[
  \theta_j
  :=\lim_{n\to\infty}\frac{p_j(s_n)}{p_j(t_n)}
  \in[0,1],
  \qquad0\le j\le k+1,
  \]
  exist.
  Set $\bm\theta:=(\theta_j)_{j=0}^{k+1}$.
  Then, for every $p\in[1,\infty)$,
  \begin{equation}\label{eq:LLN_PB}
  \frac{\b_k\bigl(X_n(s_n)\subset X_n(t_n)\bigr)}{\binom n{k+1}q_k(t_n)}
  \xrightarrow{L^p}h_k(c)-h_k(c,\bm\theta)
  \quad\text{as $n\to\infty$.}
  \end{equation}
\end{thm}

If $s_n=t_n$, then $\bm\theta=\bm1$, and~\eqref{eq:hk_extremal} shows that Theorem~\ref{thm:LLN_PB} recovers the one-time law of large numbers~\eqref{eq:LLN_Betti}.
The limit exhibits universality in the following sense: for fixed $k$, the limiting constant in~\eqref{eq:LLN_PB} depends on the sequence $\bm p=(p_j)_{j=0}^\infty$ of distribution functions and on the time sequences $(s_n)$ and $(t_n)$ only through $c$ and $\bm\theta$.

\begin{rmk}
  The convergence assumption for $j=k+1$ in Theorem~\ref{thm:LLN_PB} can be omitted.
  Indeed, the ratios $p_{k+1}(s_n)/p_{k+1}(t_n)$ lie in $[0,1]$, so every subsequence has a further subsequence along which they converge.
  The theorem gives the same $L^p$ limit along each such further subsequence, since $h_k(c,\bm\theta)$ is independent of $\theta_{k+1}$~(Remark~\ref{rmk:last_theta}).
\end{rmk}

Our proof strategy and its relation to earlier methods are described in Section~\ref{ssec:pf_strategy}.

\subsubsection*{A law of large numbers for persistence diagrams}
Let $\cX=\{X(t)\}_{t\ge0}$ be an increasing family of simplicial complexes such that $\bigcup_{t\ge0}X(t)$ is finite and
\[
X(t)=\bigcap_{t'>t}X(t')
\qquad\text{for every $t\ge0$.}
\]
Such a family is called a \emph{finite right-continuous filtration}.
The interval decomposition of persistent homology determines a finite multiset of birth--death pairs
\[
\{(b_i,d_i)\}_{i=1}^m,
\qquad
(b_i,d_i)\in\Dl:=\{(s,t)\in[0,\infty]^2\mid s<t\}.
\]
Each pair $(b_i,d_i)$ represents a homology class that appears at time $b_i$ and persists over the interval $[b_i,d_i)$, with $d_i=\infty$ if it persists at all later times.
We refer to Section~\ref{ssec:H_PB_PD} for the precise construction and the conventions used below.
We regard the $k$th persistence diagram of $\cX$ as the counting measure
\[
\Dgm_k^\cX:=\sum_{i=1}^m\dl_{(b_i,d_i)}
\]
on $\Dl$.

We identify $[0,\infty]$ with $[0,1]$ via the bijection $t\mapsto t/(1+t)$ for $t\in[0,\infty)$ and $\infty\mapsto1$.
Under the induced identification $[0,\infty]^2\simeq[0,1]^2$, we equip $\Dl$ with the relative topology inherited from $[0,1]^2$.
Then $\Dl$ is a locally compact, second-countable, Hausdorff space.
On this space, a Borel measure $\sfD$ is a \emph{Radon measure} if it is finite on every compact subset; regularity then follows automatically.
Let $\cM(\Dl)$ denote the space of Radon measures on $\Dl$, and set
\[
\cM_\fin(\Dl):=\{\sfD\in\cM(\Dl)\mid\sfD(\Dl)<\infty\}.
\]

Let $C_c(\Dl)$ and $C_b(\Dl)$ denote the spaces of compactly supported continuous functions and bounded continuous functions from $\Dl$ to $\R$, respectively.
For $\sfD_n,\sfD\in\cM(\Dl)$, \emph{vague convergence} means that
\[
\int_\Dl f\,d\sfD_n\longrightarrow\int_\Dl f\,d\sfD
\quad\text{as $n\to\infty$}
\]
for every $f\in C_c(\Dl)$.
For $\sfD_n,\sfD\in\cM_\fin(\Dl)$, \emph{weak convergence} is defined by the same condition for every $f\in C_b(\Dl)$.
Weak convergence implies vague convergence, and vague convergence of finite Radon measures together with convergence of their total masses implies weak convergence.

\smallskip
Fix $k\in\N_0$.
The expected number of $(k+1)$-cofaces of a given $k$-simplex at time $t$, conditional on its presence, is $(n-k-1)r_k(t)$.
We focus on the \emph{constant-degree regime}, described by time sequences $(t_n)$ satisfying $nr_k(t_n)\to c\in(0,\infty)$.
Under the assumptions below, this regime captures the maximal order of the expected $k$th Betti number along the process, as follows from the one-time law of large numbers~\eqref{eq:LLN_Betti} and the upper bound in~\cite[Theorem~3.6]{HK2019}.
We therefore use $nr_k(t)$ as a deterministic clock, so that the constant-degree regime corresponds to times of order one.

Define the right-continuous generalized inverse of $r_k$ by
\[
r_k^\rightarrow(u):=\inf\{t\in[0,\infty]\mid r_k(t)>u\},
\qquad u\in[0,\infty),
\]
where the infimum of the empty set is understood to be $\infty$.
Setting $X_n(\infty):=K(n)$, define the time-changed filtration
\[
\tilde\cX_n:=\{\tilde X_n(x)\}_{x\ge0}
\quad\text{with}\quad
\tilde X_n(x):=X_n\bigl(r_k^\rightarrow(x/n)\bigr).
\]
The dependence of $\tilde\cX_n$ and $\tilde X_n$ on the fixed dimension $k$ is suppressed from the notation.
This is a finite right-continuous filtration, and $\tilde X_n(x)=K(n)$ for $x\ge n$.

We impose the following conditions.
\begin{enumerate}[label=(C\arabic*),ref=(C\arabic*),leftmargin=*]
  \item\label{cond:C1} It holds that
  \[
  \lim_{n\to\infty}\binom n{k+1}q_k\bigl(r_k^\rightarrow(1/n)\bigr)=\infty.
  \]
  \item\label{cond:C2} We have $r_k(0)=0$, and there exists $t_0>0$ such that $r_k$ is continuous on $[0,t_0]$ and $r_k(t_0)>0$.
  \item\label{cond:C3} For every $0<x<y<\infty$, the limits
  \[
  \theta_{k,j}(x,y):=\lim_{n\to\infty}\frac{p_j\bigl(r_k^\rightarrow(x/n)\bigr)}{p_j\bigl(r_k^\rightarrow(y/n)\bigr)},
  \qquad 0\le j\le k+1,
  \]
  exist.
  Set $\bm\theta_k(x,y):=(\theta_{k,j}(x,y))_{j=0}^{k+1}$.
\end{enumerate}
Condition~\ref{cond:C1} ensures that the expected number of $k$-simplices at time $r_k^\rightarrow(1/n)$ tends to infinity.
Under Condition~\ref{cond:C2}, $r_k^\rightarrow$ is a right inverse of $r_k$ near zero, even if $r_k$ has flat intervals: for each fixed $x>0$ and all sufficiently large $n$,
\[
nr_k\bigl(r_k^\rightarrow(x/n)\bigr)=x.
\]
Thus, the new time parameter $x$ is the limiting conditional expected degree of a $k$-simplex in $\tilde X_n(x)$.
Condition~\ref{cond:C3} specifies the limiting thinning probabilities in each dimension for the earlier complex relative to the later one.

As shown in Lemma~\ref{lem:theta_Theta}, the limits
\[
\Theta_k(y):=\lim_{n\to\infty}\frac{q_k\bigl(r_k^\rightarrow(y/n)\bigr)}{q_k\bigl(r_k^\rightarrow(1/n)\bigr)}\in(0,\infty),
\qquad y>0,
\]
exist.
Define
\[
\hat\b_k(x,y):=\Theta_k(y)\bigl\{h_k(y)-h_k\bigl(y,\bm\theta_k(x,y)\bigr)\bigr\},
\qquad 0<x<y<\infty.
\]
Lemmas~\ref{lem:theta_Theta} and~\ref{lem:hk_conti} imply that $\hat\b_k$ is continuous on this domain.

For all sufficiently large $n$, define
\begin{equation}\label{eq:def_Pi_kn}
\sfD_{k,n}:=\frac{\Dgm_k^{\tilde\cX_n}}{\binom n{k+1}q_k\bigl(r_k^\rightarrow(1/n)\bigr)}.
\end{equation}
Thus, $\sfD_{k,n}$ is the persistence diagram of the time-changed filtration, normalized by the expected number of $k$-simplices in $\tilde X_n(1)$.

Our second main result is the following.
In what follows, $\xrightarrow{\mathrm p}$ denotes convergence in probability.
\begin{thm}\label{thm:LLN_PD}
  Assume that Conditions~\textup{\ref{cond:C1}}--\textup{\ref{cond:C3}} hold.
  Then there exists a unique deterministic measure $\sfD_k\in\cM_\fin(\Dl)$ satisfying
  \[
  \sfD_k\bigl([0,x]\times(y,\infty]\bigr)=\hat\b_k(x,y)
  \]
  for every $0<x<y<\infty$.
  Moreover,
  \[
  \sfD_{k,n}\xrightarrow{\mathrm p}\sfD_k
  \quad\text{and}\quad
  \E[\sfD_{k,n}]\longrightarrow\sfD_k
  \quad\text{as $n\to\infty$}
  \]
  in $\cM_\fin(\Dl)$ equipped with the weak topology.
\end{thm}

The limiting measure $\sfD_k$ has no atoms.
Indeed, by continuity of $\hat\b_k$, taking $y\uparrow y_0$ or $x\uparrow x_0$ in the quadrant identity of Theorem~\ref{thm:LLN_PD} shows that $\sfD_k$ charges no horizontal line $\{y=y_0\}$ with $0<y_0<\infty$ and no vertical line $\{x=x_0\}$ with $x_0>0$.
Theorem~\ref{thm:LLN_PD_test} also excludes mass at $y=\infty$, so every individual point has zero mass.
The birth axis $\{x=0\}$ may nevertheless carry positive mass, as in the Linial--Meshulam example below.

For numerical illustrations of these limits in the random flag complex process, see Figures~\ref{fig:dgm_k1}--\ref{fig:persbetti_k2}.

\subsubsection*{Convergence for unbounded test functions}
Our third main result extends the weak convergence in Theorem~\ref{thm:LLN_PD} to continuous test functions that may grow polynomially in the death coordinate and diverge at the diagonal.
Write
\[
\Dl_\fin:=\Dl\cap[0,\infty)^2.
\]
The following result concerns integrals against the normalized persistence diagrams of the time-changed filtrations.

\begin{thm}\label{thm:LLN_PD_test}
  Assume that Conditions~\textup{\ref{cond:C1}}--\textup{\ref{cond:C3}} hold.
  Then $\sfD_k(\Dl\setminus\Dl_\fin)=0$.
  Let $f\colon\Dl_\fin\to\R$ be continuous, and suppose that there exist $C<\infty$, $\a\ge0$, and $\gm\in[0,1)$ such that
  \begin{equation}\label{eq:PD_test_growth}
  |f(x,y)|\le C(1+y^\a)\bigl\{1+(y-x)^{-\gm}\bigr\},
  \qquad (x,y)\in\Dl_\fin.
  \end{equation}
  Then $\int_{\Dl_\fin}|f|\,d\sfD_k<\infty$ and
  \begin{equation}\label{eq:LLN_PD_test}
  \int_{\Dl_\fin}f\,d\sfD_{k,n}\xrightarrow{L^1}\int_{\Dl_\fin}f\,d\sfD_k
  \quad\text{as $n\to\infty$.}
  \end{equation}
\end{thm}

Condition~\eqref{eq:PD_test_growth} allows polynomial growth of any fixed order in the death coordinate and a singularity of order $(y-x)^{-\gm}$, with $\gm<1$, at the diagonal.
Both are measured in the time parameter of $\tilde\cX_n$ and may occur simultaneously.
The restriction $\gm<1$ comes from the short-lifetime estimate~\eqref{eq:diagonal_tail}: for bounded death coordinates, the expected mass within distance $h$ of the diagonal is bounded by a constant multiple of $h$, ensuring integrability of $(y-x)^{-\gm}$ for $\gm<1$.

\subsubsection*{Consequences}
The preceding theorems yield convergence of lifetime measures and total persistence, as well as convergence in the optimal transport distance.

\paragraph*{Additive and multiplicative lifetimes}
The convergence of lifetime measures follows directly from Theorem~\ref{thm:LLN_PD}.
Equip $(0,\infty]$ and $(1,\infty]$ with the relative topologies inherited from $[0,\infty]$, and let $\cM_\fin((0,\infty])$ and $\cM_\fin((1,\infty])$ denote the corresponding spaces of finite Borel measures, each equipped with the weak topology.
Write
\begin{equation}\label{eq:Dgm_tildeXn}
\Dgm_k^{\tilde\cX_n}=\sum_{i=1}^{m_n}\dl_{(b_i,d_i)}.
\end{equation}
For all sufficiently large $n$, define the \emph{normalized additive and multiplicative lifetime measures} of $\tilde\cX_n$ by
\[
\Sg_{k,n}
:=\frac{\sum_{i=1}^{m_n}\dl_{d_i-b_i}}{\binom n{k+1}q_k\bigl(r_k^\rightarrow(1/n)\bigr)}
\quad\text{and}\quad
\Pi_{k,n}
:=\frac{\sum_{i=1}^{m_n}\dl_{d_i/b_i}}{\binom n{k+1}q_k\bigl(r_k^\rightarrow(1/n)\bigr)},
\]
respectively.
Here and below, we use the conventions $\infty-b:=\infty$ for $0\le b<\infty$, $\infty/b:=\infty$ for $0<b<\infty$, and $d/0:=\infty$ for $0<d\le\infty$.

Define finite measures $\Sg_k$ and $\Pi_k$ on $(0,\infty]$ and $(1,\infty]$, respectively, by
\begin{align*}
\Sg_k(A)&:=\sfD_k\bigl(\{(x,y)\in\Dl\mid y-x\in A\}\bigr)
\shortintertext{and}
\Pi_k(B)&:=\sfD_k\bigl(\{(x,y)\in\Dl\mid y/x\in B\}\bigr)
\end{align*}
for Borel sets $A\subset(0,\infty]$ and $B\subset(1,\infty]$.

\begin{cor}
  Assume that Conditions~\textup{\ref{cond:C1}}--\textup{\ref{cond:C3}} hold.
  Then, as $n\to\infty$,
  \[
  \Sg_{k,n}\xrightarrow{\mathrm p}\Sg_k
  \quad\text{and}\quad
  \E[\Sg_{k,n}]\longrightarrow\Sg_k
  \]
  in $\cM_\fin((0,\infty])$, and
  \[
  \Pi_{k,n}\xrightarrow{\mathrm p}\Pi_k
  \quad\text{and}\quad
  \E[\Pi_{k,n}]\longrightarrow\Pi_k
  \]
  in $\cM_\fin((1,\infty])$, where both spaces are equipped with the weak topology.
\end{cor}
\begin{proof}
  The measures $\Sg_{k,n}$ and $\Pi_{k,n}$ are the pushforwards of $\sfD_{k,n}$ under the continuous maps $(x,y)\mapsto y-x$ and $(x,y)\mapsto y/x$, respectively.
  Similarly, $\Sg_k$ and $\Pi_k$ are the pushforwards of $\sfD_k$ under these maps.
  The continuous mapping theorem applied to these pushforward maps, which commute with taking mean measures, gives the assertions from Theorem~\ref{thm:LLN_PD}.
\end{proof}

\paragraph*{Total persistence}
For $p\in\R$ and a Radon measure $\sfD$ on $\Dl_\fin$, define
\[
\Pers_p(\sfD)
:=\int_{\Dl_\fin}(y-x)^p\,d\sfD(x,y).
\]
For $p>0$, this is the \emph{$p$th total persistence} of $\sfD$.
For $p=0$, it is the total mass of $\sfD$, and for $p<0$, it is a negative moment of the associated additive lifetime measure.

All death times of the time-changed filtration $\tilde\cX_n$ are finite, since $\tilde X_n(n)=K(n)$ and we use reduced homology.
Hence, $\Dgm_k^{\tilde\cX_n}$ and $\sfD_{k,n}$ are supported on $\Dl_\fin$.
If $\Dgm_k^{\tilde\cX_n}$ is given by~\eqref{eq:Dgm_tildeXn}, then
\[
\Pers_p(\sfD_{k,n})
=\frac{\Pers_p\bigl(\Dgm_k^{\tilde\cX_n}\bigr)}{\binom n{k+1}q_k\bigl(r_k^\rightarrow(1/n)\bigr)}
=\frac1{\binom n{k+1}q_k\bigl(r_k^\rightarrow(1/n)\bigr)}\sum_{i=1}^{m_n}(d_i-b_i)^p.
\]
Theorem~\ref{thm:LLN_PD_test} immediately implies convergence of these quantities for every $p>-1$.
\begin{cor}\label{cor:LLN_Pers}
  Assume that Conditions~\textup{\ref{cond:C1}}--\textup{\ref{cond:C3}} hold, and let $p\in(-1,\infty)$.
  Then
  \[
  \Pers_p(\sfD_{k,n})\xrightarrow{L^1}\Pers_p(\sfD_k)<\infty
  \quad\text{as $n\to\infty$.}
  \]
\end{cor}
\begin{proof}
  The function $f(x,y):=(y-x)^p$ satisfies~\eqref{eq:PD_test_growth} with $\a=\max\{p,0\}$ and $\gm=\max\{-p,0\}<1$.
  The conclusion follows from Theorem~\ref{thm:LLN_PD_test}.
\end{proof}

\paragraph*{Convergence in the optimal transport distance}
Building on the transport distance with boundary introduced by Figalli and Gigli~\cite{FG2010} for measures on a bounded domain, Divol and Lacombe~\cite{DL2021} extended partial matching distances between persistence diagrams to Radon measures on $\Dl_\fin$.
They characterized convergence in these distances by vague convergence together with convergence of total persistence.
We first briefly recall the definition.

Set
\[
\overline\Dl_\fin:=\{(x,y)\in[0,\infty)^2\mid x\le y\}.
\]
Fix $p\in[1,\infty)$, and let $\sfD,\sfD'$ be Radon measures on $\Dl_\fin$ satisfying
\begin{equation}\label{eq:pers_condi}
\Pers_p(\sfD)<\infty
\quad\text{and}\quad
\Pers_p(\sfD')<\infty.
\end{equation}
Let $\Gm(\sfD,\sfD')$ denote the set of Borel measures $\gm$ on $\overline\Dl_\fin\times\overline\Dl_\fin$ satisfying
\[
\gm\bigl(A\times\overline\Dl_\fin\bigr)=\sfD(A)
\quad\text{and}\quad
\gm\bigl(\overline\Dl_\fin\times B\bigr)=\sfD'(B)
\]
for all Borel sets $A,B\subset\Dl_\fin$.
Such measures are called \emph{admissible transport plans}.
The marginal constraints apply only off the diagonal, allowing mass to be transported to and from the diagonal.
Define the \emph{optimal transport distance} by
\[
\OT_p(\sfD,\sfD')
:=\left(\inf_{\gm\in\Gm(\sfD,\sfD')}
\int_{\overline\Dl_\fin\times\overline\Dl_\fin}\|z-z'\|^p\,d\gm(z,z')\right)^{1/p},
\]
where $\|\cdot\|$ denotes any norm on $\R^2$.
This defines a metric on the class of Radon measures satisfying the above moment condition~\eqref{eq:pers_condi}.
Since all norms on $\R^2$ are equivalent, neither this class nor the induced convergence depends on the choice of $\|\cdot\|$.

Theorems~\ref{thm:LLN_PD} and~\ref{thm:LLN_PD_test} yield the following result through Corollary~\ref{cor:LLN_Pers}.
Here, $\sfD_{k,n}$ and $\sfD_k$ are regarded as measures on $\Dl_\fin$.
\begin{cor}\label{cor:LLN_OT}
  Assume that Conditions~\textup{\ref{cond:C1}}--\textup{\ref{cond:C3}} hold.
  Then, for every $p\in[1,\infty)$,
  \[
  \OT_p(\sfD_{k,n},\sfD_k)\xrightarrow{L^p}0
  \quad\text{and}\quad
  \OT_p(\E[\sfD_{k,n}],\sfD_k)\longrightarrow0
  \quad\text{as $n\to\infty$.}
  \]
\end{cor}
The proof is given in Section~S3.4 of the Supplementary Material~\cite{HKsupp}.

\subsection{Typical Examples}
Our results apply to a general sequence $\bm p=(p_j)_{j=0}^\infty$ of distribution functions on $[0,\infty]$.
We now specialize them to power-law parameter functions, including the Linial--Meshulam and random flag complex processes.
In this setting, the time change, the normalization, and the functions $\Theta_k$ and $\bm\theta_k$ can all be written explicitly in terms of the exponents.
Throughout, we write $a\wedge b:=\min\{a,b\}$ and $a\vee b:=\max\{a,b\}$.

\paragraph*{Power-law parameter functions}
Given $\bm\a=(\a_j)_{j=0}^\infty\in[0,\infty]^{\N_0}$, define
\[
p_j(t):=
\begin{cases}
t^{\a_j}\wedge1
&\text{if $\a_j<\infty$,}\\
0
&\text{if $\a_j=\infty$,}
\end{cases}
\qquad t\in[0,\infty),
\quad\text{and}\quad
p_j(\infty):=1.
\]
Here, we use the convention $t^0:=1$ for every finite $t\ge0$.
Thus, $\a_j=0$ and $\a_j=\infty$ correspond to deterministic weights $0$ and $\infty$, respectively.

Fix $k\in\N_0$, and assume that $\a_j<\infty$ for every $0\le j\le k+1$ and that at least one of these exponents is positive.
Set
\[
\a_k^{(q)}:=\sum_{j=0}^k\binom{k+1}{j+1}\a_j
\quad\text{and}\quad
\a_k^{(r)}:=\sum_{j=0}^{k+1}\binom{k+1}j\a_j.
\]
Then $\a_k^{(q)}\ge0$ and $\a_k^{(r)}>0$.
These are the power-law exponents of the $k$-simplex probability $q_k(t)$ and the conditional coface probability $r_k(t)$, respectively.
The product formulas~\eqref{eq:qkt-rkt} give
\[
q_k(t)=t^{\a_k^{(q)}}\wedge1
\quad\text{and}\quad
r_k(t)=t^{\a_k^{(r)}}\wedge1,
\qquad t\in[0,\infty).
\]
Hence, Condition~\ref{cond:C2} holds with $t_0=1$, and
\[
r_k^\rightarrow(u)=u^{1/\a_k^{(r)}},
\qquad 0\le u<1.
\]
In particular, the time-changed filtration satisfies
\[
\tilde X_n(x)=X_n\bigl((x/n)^{1/\a_k^{(r)}}\bigr),
\qquad 0\le x<n,
\]
and $\tilde X_n(x)=K(n)$ for $x\ge n$.

The limits in Condition~\ref{cond:C3} and the corresponding function $\Theta_k$ are
\[
\Theta_k(y)=y^{\a_k^{(q)}/\a_k^{(r)}}
\quad\text{and}\quad
\theta_{k,j}(x,y)=\left(\frac xy\right)^{\a_j/\a_k^{(r)}},
\qquad 0<x<y<\infty,\quad 0\le j\le k+1.
\]
Moreover,
\[
\binom n{k+1}q_k\bigl(r_k^\rightarrow(1/n)\bigr)
=\binom n{k+1}n^{-\a_k^{(q)}/\a_k^{(r)}}
\sim\frac{n^{k+1-\a_k^{(q)}/\a_k^{(r)}}}{(k+1)!},
\]
where $a_n\sim b_n$ means $a_n/b_n\to1$ as $n\to\infty$.
Therefore, Condition~\ref{cond:C1} is equivalent to
\begin{equation}\label{eq:criterion_C1}
\a_k^{(q)}<(k+1)\a_k^{(r)}.
\end{equation}

Substituting these expressions into Theorem~\ref{thm:LLN_PD} gives the following result.
\begin{cor}\label{cor:power-law}
  Let $\bm\a=(\a_j)_{j=0}^\infty\in[0,\infty]^{\N_0}$ and $k\in\N_0$, and let $\cX_n$ be the multiparameter complex process with the power-law parameter functions defined above.
  Assume that $\a_j<\infty$ for every $0\le j\le k+1$ and that $\a_j>0$ for at least one such $j$.
  Suppose further that~\eqref{eq:criterion_C1} holds.
  Then
  \[
  \frac{(k+1)!}{n^{k+1-\a_k^{(q)}/\a_k^{(r)}}}\Dgm_k^{\tilde\cX_n}
  \xrightarrow{\mathrm p}\sfD_k
  \quad\text{as $n\to\infty$}
  \]
  in $\cM_\fin(\Dl)$ equipped with the weak topology.
  The corresponding mean measures also converge weakly to $\sfD_k$.
  The limiting measure $\sfD_k$ is uniquely characterized by
  \[
  \sfD_k\bigl([0,x]\times(y,\infty]\bigr)
  =y^{\a_k^{(q)}/\a_k^{(r)}}
  \Bigl\{h_k(y)-h_k\Bigl(y,\bigl((x/y)^{\a_j/\a_k^{(r)}}\bigr)_{j=0}^{k+1}\Bigr)\Bigr\},
  \qquad 0<x<y<\infty.
  \]
\end{cor}

Under the same condition, Theorem~\ref{thm:LLN_PD_test} and the corollaries concerning lifetime measures, total persistence, and the optimal transport distance also apply.
In the examples below, all measure convergence statements use the weak topology on $\cM_\fin(\Dl)$, and the corresponding mean measures converge to the same limits.

\paragraph*{The Linial--Meshulam complex process}
Let $d\in\N$, and take
\[
\a_j:=
\begin{cases}
  0         &\text{if $0\le j\le d-1$,}\\
  1         &\text{if $j=d$,}\\
  \infty    &\text{if $j\ge d+1$.}
\end{cases}
\]
Thus, at each finite time $t$, $X_n(t)$ has a complete $(d-1)$-skeleton, each $d$-simplex is present independently with probability $t\wedge1$, and no simplices of dimension greater than $d$ are present.

For $k=d-1$, we have
\[
\a_{d-1}^{(q)}=0
\quad\text{and}\quad
\a_{d-1}^{(r)}=1,
\]
so~\eqref{eq:criterion_C1} holds.
The time change and the convergence in Corollary~\ref{cor:power-law} become
\[
\tilde X_n(x)=X_n(x/n),
\qquad 0\le x<n,
\]
and
\[
\frac{d!}{n^d}\Dgm_{d-1}^{\tilde\cX_n}
\xrightarrow{\mathrm p}\sfD_{d-1}
\quad\text{as $n\to\infty$.}
\]
Here,
\[
\Theta_{d-1}(y)=1
\quad\text{and}\quad
\bm\theta_{d-1}(x,y)
=\biggl(1,\ldots,1,\frac xy\biggr),
\qquad 0<x<y<\infty.
\]
Since $h_{d-1}$ does not depend on the last coordinate of its parameter vector by Remark~\ref{rmk:last_theta}, it follows from~\eqref{eq:hk_extremal} that
\[
h_{d-1}\bigl(y,\bm\theta_{d-1}(x,y)\bigr)=h_{d-1}(y,\bm1)=0.
\]
Consequently,
\[
\sfD_{d-1}\bigl([0,x]\times(y,\infty]\bigr)=h_{d-1}(y),
\qquad 0<x<y<\infty.
\]
In this model, every birth--death pair in degree $d-1$ has birth coordinate zero, since the complete $(d-1)$-skeleton is present from time zero.
Accordingly, $\sfD_{d-1}$ is supported on $\{0\}\times(0,\infty)$ and is determined by the death coordinate alone.
A limit theorem for the empirical distribution of these death times was proved by Fraiman, Mukherjee, and Thoppe~\cite{FMT2023}.

\paragraph*{The random flag complex process}
Take $\a_1:=1$ and $\a_j:=0$ for $j\neq1$.
Then $X_n(t)$ is the flag complex (also called the clique complex) of the Erd\H{o}s--R\'enyi graph with edge probability $t\wedge1$; its simplices are the vertex sets spanning complete subgraphs.

Fix $k\ge1$.
We have
\[
\a_k^{(q)}=\binom{k+1}2
\quad\text{and}\quad
\a_k^{(r)}=k+1,
\]
so that $\a_k^{(q)}/\a_k^{(r)}=k/2$ and~\eqref{eq:criterion_C1} holds.
The time change is
\[
\tilde X_n(x)=X_n\bigl((x/n)^{1/(k+1)}\bigr),
\qquad 0\le x<n,
\]
and Corollary~\ref{cor:power-law} yields
\[
\frac{(k+1)!}{n^{1+k/2}}\Dgm_k^{\tilde\cX_n}
\xrightarrow{\mathrm p}\sfD_k
\quad\text{as $n\to\infty$.}
\]
The functions determining the limit are
\[
\Theta_k(y)=y^{k/2}
\quad\text{and}\quad
\bm\theta_k(x,y)=\left(1,\left(\frac xy\right)^{1/(k+1)},1,\ldots,1\right),
\qquad 0<x<y<\infty.
\]
Hence, for $0<x<y<\infty$,
\[
\sfD_k\bigl([0,x]\times(y,\infty]\bigr)
=y^{k/2}
\Bigl\{h_k(y)-h_k\Bigl(y,\bigl(1,(x/y)^{1/(k+1)},1,\ldots,1\bigr)\Bigr)\Bigr\}.
\]

In degree one, the term $h_1(y,\bm\theta_1(x,y))$ in this formula can be evaluated through a scalar fixed-point equation.
For $0<x<y<\infty$, set $\eta:=1-\sqrt{x/y}$, and let $v_{\min}$ and $v_{\max}$ denote the smallest and largest solutions in $(0,1]$ of
\[
v=\exp\bigl(-y(1-\eta v)^2\bigr).
\]
Then
\begin{equation}\label{eq:flag_scalar_intro}
h_1(y,\bm\theta_1(x,y))
=\max_{v\in\{v_{\min},v_{\max}\}}
\left\{\eta v-y\eta^2v^2+\frac{2y}{3}\eta^3v^3\right\}.
\end{equation}
This follows from the scalar reduction in Section~S2 of the Supplementary Material~\cite{HKsupp}; see Remark~S2.2 there.

\paragraph*{Uniform simplex weights}
Take $\a_j:=1$ for every $j\ge0$, so that every nonempty simplex receives an independent weight that is uniform on $[0,1]$.
For every $k\in\N_0$,
\[
\a_k^{(q)}=2^{k+1}-1
\quad\text{and}\quad
\a_k^{(r)}=2^{k+1}.
\]
Thus, $\a_k^{(q)}/\a_k^{(r)}=1-2^{-(k+1)}$, and~\eqref{eq:criterion_C1} holds.
The time change is
\[
\tilde X_n(x)=X_n\bigl((x/n)^{2^{-(k+1)}}\bigr),
\qquad 0\le x<n,
\]
and Corollary~\ref{cor:power-law} gives
\[
\frac{(k+1)!}{n^{k+2^{-(k+1)}}}\Dgm_k^{\tilde\cX_n}
\xrightarrow{\mathrm p}\sfD_k
\quad\text{as $n\to\infty$.}
\]
In this case,
\[
\Theta_k(y)=y^{1-2^{-(k+1)}}
\quad\text{and}\quad
\bm\theta_k(x,y)=\left(\frac xy\right)^{2^{-(k+1)}}\bm1.
\]
Therefore,
\[
\sfD_k\bigl([0,x]\times(y,\infty]\bigr)
=y^{1-2^{-(k+1)}}
\Bigl\{h_k(y)-h_k\Bigl(y,(x/y)^{2^{-(k+1)}}\bm1\Bigr)\Bigr\},
\qquad 0<x<y<\infty.
\]

\paragraph*{Numerical illustrations}
Figures~\ref{fig:dgm_k1}--\ref{fig:persbetti_k2} provide numerical illustrations of these limits for the random flag complex process in degrees one and two.
All homology computations use coefficients in $\Q$.
The normalization in~\eqref{eq:def_Pi_kn} is $\binom n{k+1}n^{-k/2}$ for the random flag complex process.
The limiting cell masses and persistent Betti numbers are evaluated numerically using the formulas for the random flag complex process.
In degree one, we use the scalar formula~\eqref{eq:flag_scalar_intro}.
In degree two, we evaluate the maximum in~\eqref{eq:hk_c_theta} using interval branch and bound over the seven relevant type coordinates.
The search also includes fixed points other than the least and greatest ones with respect to the coordinatewise order.

The limiting persistent Betti function extends continuously to the diagonal, with $\hat\b_k(x,x)\!=x^{k/2}h_k(x)$ for $x>0$.
The diagonal sections $x=y$ of the persistent Betti plots in Figures~\ref{fig:persbetti_k1} and~\ref{fig:persbetti_k2} give the corresponding scaled Betti curves, namely, the normalized Betti numbers viewed as functions of the scaled filtration parameter.
The formula in~\cite[Theorem~3]{Kan2022} shows that the limiting curves $\hat\b_k(x,x)=x^{k/2}h_k(x)$ attain their maxima at $x=k$.
The dashed guides mark the corresponding points $(k,k)$ on the diagonal.

Numerical signatures of the critical constants $c_d^*$ for the homological phase transition in the Linial--Meshulam model, identified by Linial and Peled~\cite{LP2016a}, are also visible in the limiting persistence-diagram heatmaps and persistent Betti contours near the death coordinate $y=c_{k+1}^*$, with $c_2^*\approx2.754$ and $c_3^*\approx3.907$.
These constants also mark a change in the maximizing scalar fixed point in~\eqref{eq:hk_c}: $u=1$ attains the maximum for $c\le c_{k+1}^*$, with $h_k(c)=1-\tfrac c{k+2}$, while the least fixed point attains it for $c\ge c_{k+1}^*$; see~\cite{LP2016a} and Proposition~S2.3 of the Supplementary Material~\cite{HKsupp}.

\begin{figure}[htbp]
  \centering
  \includegraphics[width=\textwidth]{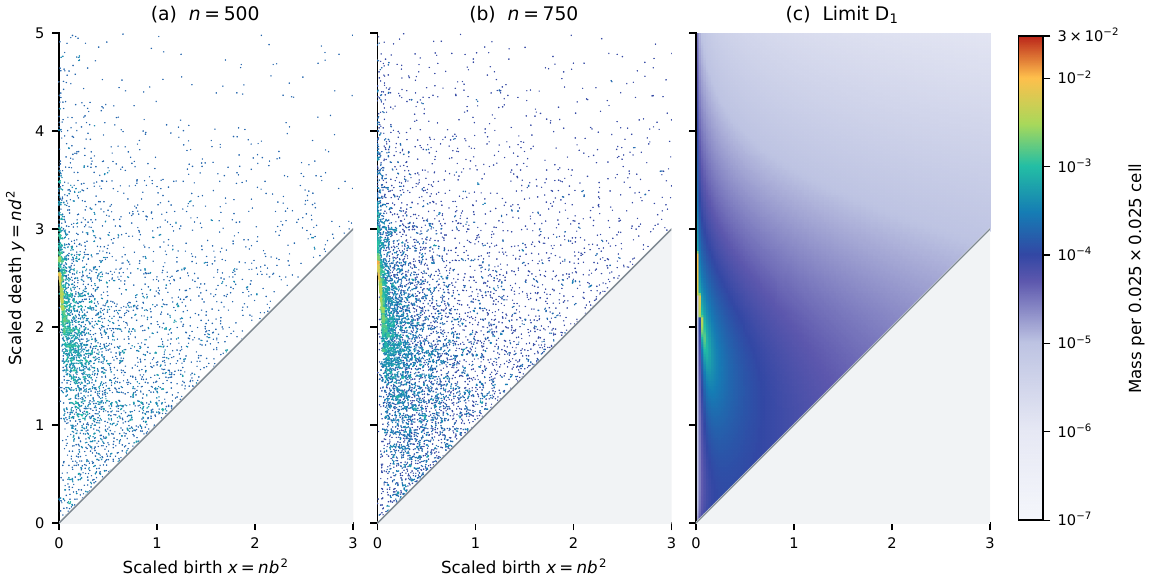}
  \caption{Persistence diagrams of the random flag complex process in degree one: one realization for each of $n=500$ and $n=750$, and the limiting measure $\sfD_1$.
  The coordinates are $(x,y)=(nb^2,nd^2)$, where $b$ and $d$ are the original birth and death times.
  Each point in panels~(a) and~(b) is colored by the mass assigned by $\sfD_{1,n}$ to its $0.025\times0.025$ cell.
  Panel~(c) shows the masses assigned by $\sfD_1$ to the same cells.
  All three panels use a common logarithmic color scale.}
  \label{fig:dgm_k1}
  \vspace{5mm}
  \includegraphics[width=\textwidth]{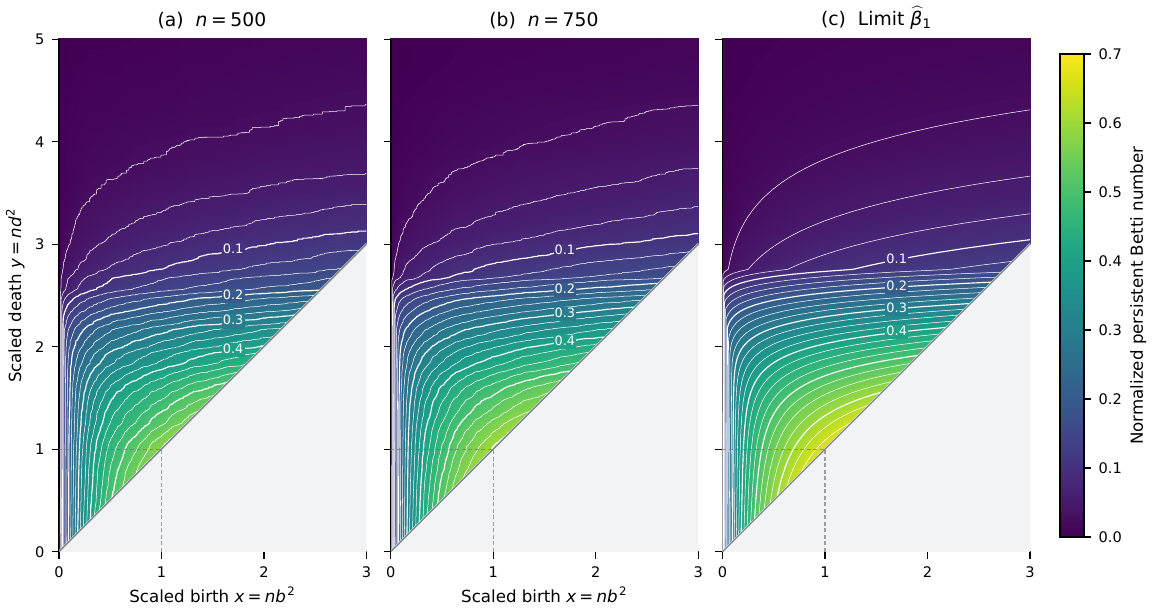}
  \caption{Normalized persistent Betti numbers $\b_1(\tilde X_n(x)\subset\tilde X_n(y))/\bigl(\binom n2 n^{-1/2}\bigr)$ for the same realizations as in Figure~\ref{fig:dgm_k1}.
  Panel~(c) shows their limit $\hat\b_1(x,y)=\sfD_1([0,x]\times(y,\infty])$.
  All three panels use a common linear color scale.
  White contours are spaced at intervals of $0.025$, with multiples of $0.1$ emphasized.}
  \label{fig:persbetti_k1}
\end{figure}

\begin{figure}[htbp]
  \centering
  \includegraphics[width=\textwidth]{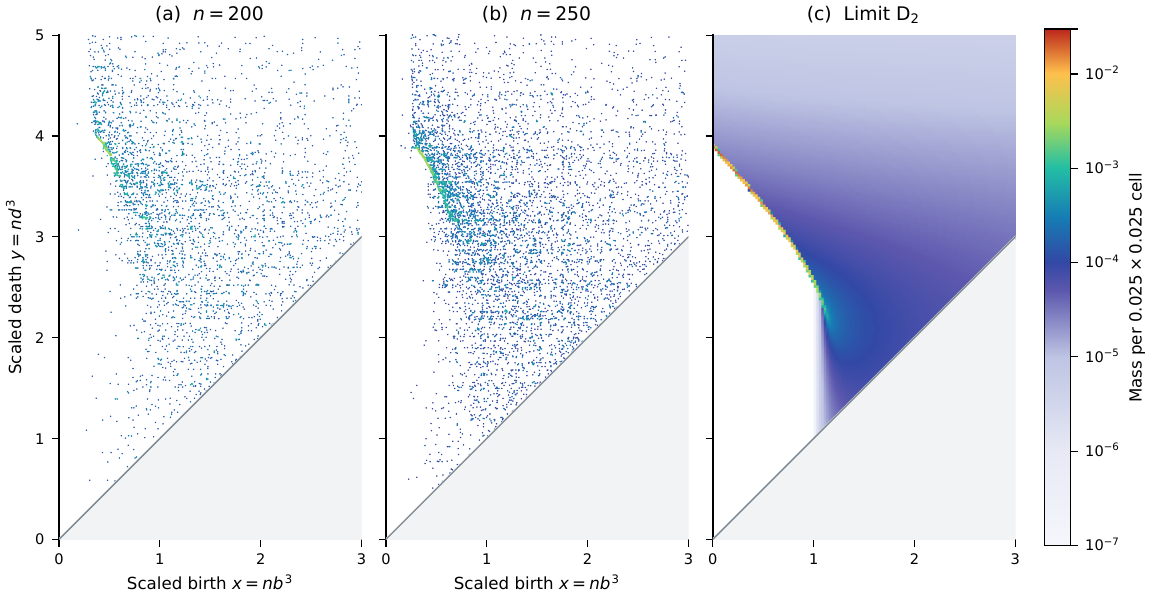}
  \caption{Persistence diagrams of the random flag complex process in degree two: one realization for each of $n=200$ and $n=250$, and the limiting measure $\sfD_2$.
  The coordinates are $(x,y)=(nb^3,nd^3)$, where $b$ and $d$ are the original birth and death times.
  Each point in panels~(a) and~(b) is colored by the mass assigned by $\sfD_{2,n}$ to its $0.025\times0.025$ cell.
  Panel~(c) shows the masses assigned by $\sfD_2$ to the same cells, with masses below $10^{-7}$ left white.
  All three panels use a common logarithmic color scale.}
  \label{fig:dgm_k2}
  \vspace{5mm}
  \includegraphics[width=\textwidth]{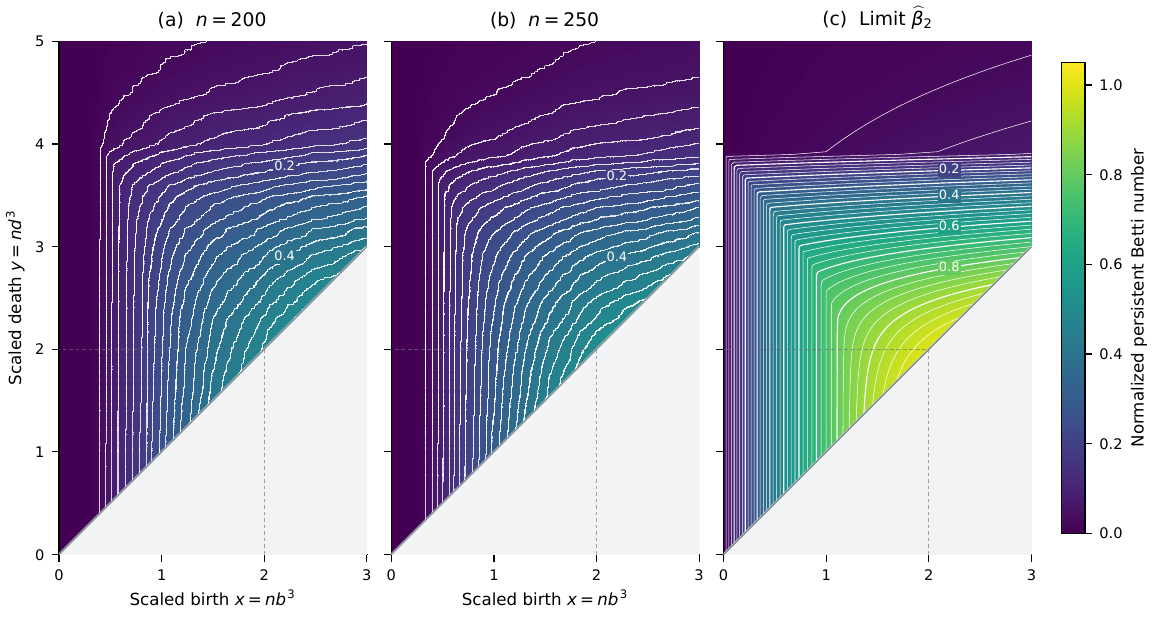}
  \caption{Normalized persistent Betti numbers $\b_2(\tilde X_n(x)\subset\tilde X_n(y))/\bigl(\binom n3 n^{-1}\bigr)$ for the same realizations as in Figure~\ref{fig:dgm_k2}.
  Panel~(c) shows their limit $\hat\b_2(x,y)=\sfD_2([0,x]\times(y,\infty])$.
  All three panels use a common linear color scale.
  White contours are spaced at intervals of $0.025$, with multiples of $0.1$ emphasized.}
  \label{fig:persbetti_k2}
\end{figure}

\FloatBarrier
\subsection{Proof Strategy and Related Work}
\label{ssec:pf_strategy}
Our proof of Theorem~\ref{thm:LLN_PB} combines local weak convergence with spectral analysis.
Informally, local weak convergence describes convergence of the distributions of neighborhoods of each fixed radius around a uniformly chosen root.
The graph framework was introduced by Benjamini and Schramm~\cite{BS2001} and developed in the objective method of Aldous and Steele~\cite{AS2004}; see also Aldous and Lyons~\cite{AL2007}.
Bordenave and Lelarge~\cite{BL2010} studied limiting spectral measures in this framework, and Bordenave, Lelarge, and Salez~\cite{BLS2011} analyzed their atomic mass at zero through a distributional recursion.
Linial and Peled~\cite{LP2016a} and later the second author~\cite{Kan2022} adapted this spectral approach to Betti numbers of random simplicial complexes.

We first express a persistent Betti number as a difference of nullities of an up Laplacian and its compression, up to a term that becomes negligible after normalization.
For a finite simplicial pair $K\subset X$, the compressed up Laplacian is the principal submatrix indexed by the $k$-simplices outside $K$; see Section~\ref{ssec:comp_L}.
This reduces the problem to determining normalized spectral masses at zero.
Our use of this reduction is inspired by the matrix-rank approach of M\'esz\'aros~\cite[Section~3.2]{Me2025}, whose rank asymptotics combine local weak convergence, spectral estimates, and leaf removal.

We then establish local weak convergence in probability of the simplicial pairs $X_n(s_n)\subset X_n(t_n)$, rooted at a uniformly chosen $k$-simplex of $X_n(t_n)$.
The key representation is the equality in law of pairs
\[
\bigl(X_n(s_n)\subset X_n(t_n)\bigr)
\stackrel{\mathrm d}{=}
\bigl(K_n(s_n,t_n)\subset X_n(t_n)\bigr),
\]
where $K_n(s_n,t_n)$ is obtained by independently thinning face labels of $X_n(t_n)$ with success probabilities $p_j(s_n)/p_j(t_n)$ in dimension $j$; see Section~\ref{ssec:LWC_MCP}.
The limit consists of a Poisson Galton--Watson simplicial tree with mean offspring $c$ and its thinning governed by $\bm\theta$; see Sections~\ref{ssec:LWC} and~\ref{ssec:PGW} for the definitions.
Its dependence only on $c$ and $\bm\theta$ explains the universality in Theorem~\ref{thm:LLN_PB}.

The thinning requires additional information about shared faces.
In M\'esz\'aros's modified Linial--Meshulam complex process, a lower-dimensional face appears when the first top-dimensional simplex containing it is added.
In our process, a simplex appears when all the independent weights assigned to its nonempty faces, including itself, are at most the current time.
Thus, the events that neighboring simplices of $X_n(t_n)$ already belong to $X_n(s_n)$ involve shared face weights and are generally dependent.
We record this information through the face types in $\cT_k$.

Local weak convergence gives weak convergence of the averaged rooted spectral measures, but convergence of their masses at zero requires a separate argument.
For finite simplicial pairs, Proposition~\ref{prop:conti_ARSM0} establishes this convergence under a uniform bound on average $k$-simplex degrees.
The key estimate, Lemma~\ref{lem:small_eigen}, bounds the number of small positive eigenvalues using integrality and the trace.
This replaces the uniform spectral bound used in Elek's argument~\cite{El2010} with a condition on average degrees.
Together with the identification of the limiting atomic mass described below, this criterion also gives a direct spectral derivation of the limit theorems for Betti numbers in~\cite{Kan2022,LP2016a}.
This approach provides a simple alternative to the lower-bound arguments based on elementary collapses; see Remark~\ref{rmk:integrality}.

It remains to compute the expected atomic mass at zero on the limiting pair.
The tree structure gives a recursion in terms of descendant branches.
Conditioning on face types yields a system of distributional fixed-point equations that accounts for the shared labels.
Extending the analysis of~\cite{BLS2011,LP2016a}, we identify its spectral solution and show that the expected mass is the maximum of $\Psi_{k,c,\bm\theta}$ over the fixed points of $\Phi_{k,c,\bm\theta}$.
Accounting for the dependence through face types and reducing the recursion to this finite-dimensional optimization problem are the main additional ingredients in identifying $h_k(c,\bm\theta)$; see Sections~\ref{ssec:DFPE} and~\ref{ssec:reduction}.
Moment bounds for simplex counts then give the $L^p$ convergence in Theorem~\ref{thm:LLN_PB}.

For persistence diagrams, the quadrant identity converts the persistent Betti limits into limits of quadrant masses.
Following the random measure argument of Hiraoka, Shirai, and Trinh~\cite{HST2018}, we obtain vague convergence of $\sfD_{k,n}$ and their mean measures.
To obtain weak convergence, we control large death coordinates and short lifetimes.
These estimates combine the Betti-number bound in~\cite[Theorem~3.6]{HK2019} with a deterministic bound on the number of birth--death pairs in a common time interval, obtained by counting incident simplices appearing at distinct times.
Their proofs are given in Appendix~\ref{app:PD_tails}.
They yield convergence of the total masses in $L^1$ and hence weak convergence.

The same estimates control truncation errors for unbounded test functions.
As in~\cite[Section~4.3]{Me2025}, the death-coordinate estimates allow polynomial growth, while the short-lifetime estimate additionally permits the diagonal singularity in Theorem~\ref{thm:LLN_PD_test}.
Convergence of lifetime measures follows by continuous pushforward.
Finally, convergence of total persistence and the criterion of Divol and Lacombe~\cite[Theorem~3.7]{DL2021}, together with the corresponding moment bounds, give convergence in the optimal transport distance.

\subsection*{Organization of the paper}
Section~\ref{sec:prelim} introduces the homological conventions, compressed Laplacians, local weak convergence, and thinned Poisson Galton--Watson simplicial trees.
Section~\ref{sec:RSM_CompL} develops the spectral framework and establishes convergence of normalized nullities of compressed Laplacians.
Section~\ref{sec:mass_zero} computes the expected atomic mass at zero on the limiting trees and derives the finite-dimensional characterization of $h_k(c,\bm\theta)$.
Section~\ref{sec:conv_PB} establishes local weak convergence of the relevant simplicial pairs and proves Theorem~\ref{thm:LLN_PB}.
Section~\ref{sec:conv_PD} proves Theorems~\ref{thm:LLN_PD} and~\ref{thm:LLN_PD_test} and discusses Corollary~\ref{cor:LLN_OT}.
Appendix~\ref{app:PD_tails} proves the tail estimates for persistence diagrams.
The Supplementary Material~\cite{HKsupp} contains the essential self-adjointness arguments needed for the spectral analysis in Section~S1, simplifications of the optimization defining $h_k(c,\bm\theta)$ under additional conditions on $c$ and $\bm\theta$ in Section~S2, and auxiliary proofs for convergence of persistence diagrams in Section~S3.

\section{Preliminaries}
\label{sec:prelim}
In this section, we collect the notation and auxiliary results used throughout the paper.

\subsection{Persistent Betti Numbers and Persistence Diagrams}
\label{ssec:H_PB_PD}
A \emph{simplicial complex} $X$ is a nonempty collection of finite sets that is closed under taking subsets; that is, if $\sg\in X$ and $\rho\subset\sg$, then $\rho\in X$.
For $\sg\in X$, the \emph{dimension} of $\sg$ is defined by $\dim\sg:=|\sg|-1$.
An element $\sg\in X$ with $\dim\sg=k$ is called a \emph{$k$-simplex} in $X$.
Any subset $\rho\subset\sg$ is called a \emph{face} of $\sg$, and $\sg$ is called a \emph{coface} of $\rho$.
If $\rho\subsetneq\sg$, then $\rho$ is called a \emph{proper face} of $\sg$.
We denote by $F_k(X)$ the set of all $k$-simplices in $X$, and set
\[
f_k(X):=|F_k(X)|.
\]
Note that $F_{-1}(X)=\{\emptyset\}$ and $f_{-1}(X)=1$.
For $k\in\N_0$ and $\sg\in F_k(X)$, define the \emph{degree} of $\sg$ in $X$ by $\deg_X(\sg):=|\{\tau\in F_{k+1}(X)\mid\sg\subset\tau\}|$.
The \emph{dimension} of $X$ is the supremum of the dimensions of its simplices.
The set $V(X):=\bigcup_{\sg\in X}\sg$ is called the \emph{vertex set} of $X$.

A sequence $\hat\sg=(v_0,v_1,\ldots,v_k)$ of $k+1$ distinct vertices in $X$ is called an \emph{ordered $k$-simplex} in $X$ if the corresponding unordered set $\sg=\{v_0,v_1,\ldots,v_k\}$ is a $k$-simplex in $X$.
We denote by $\hat F_k(X)$ the set of all ordered $k$-simplices in $X$.
By convention, we set $\hat F_{-1}(X):=\{\emptyset\}$.
Two ordered simplices are said to be \emph{equivalent} if they are transformed into each other by an even permutation.
The equivalence class of an ordered $k$-simplex $\hat\sg=(v_0,v_1,\ldots,v_k)$ is denoted by $\la\sg\ra$ or $\la v_0,v_1,\ldots,v_k\ra$ and is called the \emph{oriented $k$-simplex}.

\paragraph*{Simplicial (co)homology}
Throughout this subsection, we work over a field $\F$ of characteristic zero.
For $k\in\N_0$, the chain group $C_k(X)=C_k(X;\F)$ is the $\F$-vector space consisting of all finite formal linear combinations of oriented $k$-simplices in $X$, subject to the relations
\begin{equation}\label{eq:permutation}
\la v_{\pi(0)},v_{\pi(1)},\ldots,v_{\pi(k)}\ra
=(\sgn\pi)\la v_0,v_1,\ldots,v_k\ra
\end{equation}
for every permutation $\pi$ of $\{0,1,\ldots,k\}$.
By convention, we set $C_{-1}(X):=\F$.
For $k\in\N_0$, the $k$th \emph{boundary operator} $\del_k^X\colon C_k(X)\to C_{k-1}(X)$ is defined as the linear extension of
\[
\del_k^X\la v_0,v_1,\ldots,v_k\ra:=\sum_{i=0}^k(-1)^i\la v_0,\ldots,v_{i-1},v_{i+1},\ldots,v_k\ra
\]
for $\la v_0,v_1,\ldots,v_k\ra\in C_k(X)$.
A straightforward calculation shows that $\del_k^X\circ\del_{k+1}^X=0$; equivalently, $\Ker\del_k\supset \im\del_{k+1}$.
The vector spaces
\[
Z_k(X):=\Ker\del_k^X
\quad\text{and}\quad
B_k(X):=\im\del_{k+1}^X
\]
are called the \emph{$k$-cycle} and \emph{$k$-boundary groups}, respectively.
The $k$th (reduced) \emph{homology group} of $X$ is defined by
\[
H_k(X):=Z_k(X)/B_k(X).
\]
If $X$ is finite, the dimension $\b_k(X):=\dim H_k(X)$ is called the $k$th \emph{Betti number} of $X$.
In this case, $\b_k(X)$ is independent of the choice of the field $\F$ of characteristic zero.

For $k\ge-1$, the (algebraic) dual space $C^k(X)=C^k(X;\F):=\Hom(C_k(X),\F)$ is called the $k$th \emph{cochain group} of $X$.
The $k$th \emph{coboundary operator} $d_k^X\colon C^k(X)\to C^{k+1}(X)$ is defined as the dual operator of $\del_{k+1}^X$.
Similarly, one has $d_k^X\circ d_{k-1}^X=0$.
The \emph{$k$-cocycle group} and \emph{$k$-coboundary group} are defined as
\[
Z^k(X):=\Ker d_k^X
\quad\text{and}\quad
B^k(X):=\im d_{k-1}^X,
\]
respectively; in particular, $Z^k(X)\supset B^k(X)$.
The $k$th (reduced) \emph{cohomology group} of $X$ is defined by
\[
H^k(X):=Z^k(X)/B^k(X).
\]
If $X$ is finite, then $\dim H^k(X)=\dim H_k(X)$, since we work over a field.

\paragraph*{Persistent Betti numbers}
We next review the notion of persistent Betti numbers and derive a useful expression.
Let $K\subset X$ be a simplicial pair.
Let $\iota_K^X$ denote the inclusion from $K$ to $X$.
The inclusion $\iota_K^X$ induces a linear map $\bigl(\iota_K^X\bigr)_*\colon H_k(K)\to H_k(X)$ given by
\[
\bigl(\iota_K^X\bigr)_*\bigl(c+B_k(K)\bigr)
:=c+B_k(X),
\qquad c\in Z_k(K).
\]
We now assume that $X$ is finite.
The $k$th \emph{persistent Betti number} of a simplicial pair $K\subset X$ is defined by
\[
\b_k(K\subset X):=\rank\bigl(\bigl(\iota_K^X\bigr)_*\colon H_k(K)\longrightarrow H_k(X)\bigr).
\]
This number is independent of the choice of the field $\F$ of characteristic zero.
Moreover, if $K=X$, then $\b_k(K\subset X)=\b_k(X)$.
Since
\[
\im\bigl(\iota_K^X\bigr)_*
\cong\frac{H_k(K)}{\Ker\bigl(\iota_K^X\bigr)_*}
=\frac{Z_k(K)/B_k(K)}{(Z_k(K)\cap B_k(X))/B_k(K)}
\cong\frac{Z_k(K)}{Z_k(K)\cap B_k(X)},
\]
we obtain
\begin{equation}\label{eq:PB}
\b_k(K\subset X)
=\dim\frac{Z_k(K)}{Z_k(K)\cap B_k(X)}
=\dim\frac{Z_k(K)}{C_k(K)\cap B_k(X)}.
\end{equation}

\smallskip
For each $k\in\N_0$, let $C_k(X\setminus K)$ be the vector space spanned by the oriented $k$-simplices in $X$ that do not belong to $K$, subject to the relation~\eqref{eq:permutation}.
Then $C_k(X)=C_k(K)\oplus C_k(X\setminus K)$ as vector spaces.
Furthermore, we define the restricted boundary operator
\[
\del_{k+1}^{X\setminus K}\colon C_{k+1}(X\setminus K)\longrightarrow C_k(X\setminus K)
\]
by
\[
\del_{k+1}^{X\setminus K}:=P_{C_k(X\setminus K)}\,\del_{k+1}^X\,J_{C_{k+1}(X\setminus K)},
\]
where $P_{C_k(X\setminus K)}$ denotes the projection from $C_k(X)$ onto $C_k(X\setminus K)$, and $J_{C_{k+1}(X\setminus K)}$ denotes the canonical inclusion from $C_{k+1}(X\setminus K)$ into $C_{k+1}(X)$.
By convention, we omit the composition symbol throughout.
\begin{rmk}
  Under the canonical identifications $C_k(X)/C_k(K)\cong C_k(X\setminus K)$, the restricted boundary operators $(\del_k^{X\setminus K})_k$ correspond to the boundary operators of the relative chain complex $(C_k(X,K),\del_k^{(X,K)})_k$.
  The present formulation, however, is more convenient for the operator-theoretic analysis carried out in the subsequent sections.
\end{rmk}
We also introduce the dual space $C^k(X\setminus K):=\Hom(C_k(X\setminus K),\F)$, and denote by
\[
d_k^{X\setminus K}\colon C^k(X\setminus K)\to C^{k+1}(X\setminus K)
\]
the dual operator of $\del_{k+1}^{X\setminus K}$.
The following lemma provides a useful expression for the persistent Betti numbers in terms of the nullities of the corresponding coboundary operators.
\begin{lem}\label{lem:PB}
  Let $K\subset X$ be a finite simplicial pair and $k\in\N_0$.
  Then
  \[
  \b_k(K\subset X)
  =\nullity\bigl(d_k^X\bigr)-\nullity\bigl(d_k^{X\setminus K}\bigr)-\dim B^k(K).
  \]
\end{lem}
\begin{proof}
  By~\eqref{eq:PB} and the rank--nullity theorem, we have
  \begin{align*}
  \b_k(K\subset X)
  &=\dim Z_k(K)-\dim(C_k(K)\cap B_k(X))\\
  &=f_k(K)-\dim B_{k-1}(K)-\dim(C_k(K)\cap B_k(X)).
  \end{align*}
  Using the projection $P_{C_k(X\setminus K)}$ onto $C_k(X\setminus K)$, we obtain
  \begin{equation}\label{eq:CK_BX}
  \begin{aligned}
  \dim(C_k(K)\cap B_k(X))
  &=\nullity\bigl(P_{C_k(X\setminus K)}|_{B_k(X)}\bigr)\\
  &=\dim B_k(X)-\rank\bigl(P_{C_k(X\setminus K)}|_{B_k(X)}\bigr)\\
  &=\dim B_k(X)-\rank\bigl(P_{C_k(X\setminus K)}\,\del_{k+1}^X\bigr)\\
  &=\dim B_k(X)-\rank\bigl(\del_{k+1}^{X\setminus K}\bigr).
  \end{aligned}
  \end{equation}
  Here, the second equality follows from the rank--nullity theorem, and the last equality follows from the fact that $P_{C_k(X\setminus K)}\,\del_{k+1}^X\,c=0$ for any $c\in C_{k+1}(K)$.
  Combining the above identities yields
  \begin{equation}\label{eq:rank-rank}
  \b_k(K\subset X)
  =f_k(K)-\dim B_k(X)+\rank\bigl(\del_{k+1}^{X\setminus K}\bigr)-\dim B_{k-1}(K).
  \end{equation}
  Applying the rank--nullity theorem once more,
  \begin{align*}
  \dim B_k(X)
  &=\rank\bigl(d_k^X\bigr)
  =f_k(X)-\nullity\bigl(d_k^X\bigr)
  \shortintertext{and}
  \rank\bigl(\del_{k+1}^{X\setminus K}\bigr)
  &=\rank\bigl(d_k^{X\setminus K}\bigr)
  =\{f_k(X)-f_k(K)\}-\nullity\bigl(d_k^{X\setminus K}\bigr).
  \end{align*}
  Substituting these into~\eqref{eq:rank-rank}, together with $\dim B_{k-1}(K)=\dim B^k(K)$, completes the proof.
\end{proof}

\paragraph*{Persistence diagrams}
Let $\cX=\{X(t)\}_{t\ge0}$ be a finite right-continuous filtration, as defined in Section~\ref{sec:intro}.
For $0\le s\le t$, let $\iota_s^t\colon X(s)\longrightarrow X(t)$ denote the inclusion map.
The collection
\[
H_k(\cX)
:=
\bigl(H_k(X(t)),(\iota_s^t)_*\bigr)_{0\le s\le t<\infty}
\]
is called the \emph{$k$th persistence module} associated with $\cX$.
For $0\le b<d\le\infty$, let $\I_{[b,d)}$ denote the interval persistence module defined by
\[
\I_{[b,d)}(t):=
\begin{cases}
  \F,
  &\text{if $b\le t<d$,}\\
  0,
  &\text{otherwise.}
\end{cases}
\]
For $0\le s\le t<\infty$, the linear map $\I_{[b,d)}(s)\to\I_{[b,d)}(t)$ is defined to be the identity map on $\F$ if $b\le s\le t<d$, and the zero map otherwise.

Since $\cX$ is finite and right-continuous, it is determined by a finite sequence of simplicial complexes together with the corresponding parameter values.
Applying~\cite[Corollary~3.1]{ZC2005} to the corresponding discrete persistence module, we obtain, for each $k\in\N_0$, a finite multiset
\[
\{(b_i,d_i)\}_{i=1}^m,
\qquad 0\le b_i<d_i\le\infty,
\]
such that
\[
H_k(\cX)\cong\bigoplus_{i=1}^m\I_{[b_i,d_i)}
\]
as persistence modules.
The multiset $\{(b_i,d_i)\}_{i=1}^m$ is uniquely determined by $\cX$ and $k$.
The numbers $b_i$ and $d_i$ are called the \emph{birth time} and the \emph{death time} of the corresponding interval, respectively.
Thus, the summand $\I_{[b_i,d_i)}$ records a homology class that is present precisely during the time interval $[b_i,d_i)$.
The pair $(b_i,d_i)$ is called a \emph{birth--death pair}.

On the birth--death plane $\Dl=\{(s,t)\in[0,\infty]^2\mid s<t\}$ introduced in Section~\ref{sec:intro}, we define the \emph{$k$th persistence diagram} of $\cX$ by
\[
\Dgm_k^\cX:=\sum_{i=1}^m\dl_{(b_i,d_i)}.
\]
The uniqueness of the interval decomposition shows that this definition is independent of all choices.
Since $\Dgm_k^\cX$ is a finite counting measure, it belongs to
$\cM_\fin(\Dl)$.
For $0\le s\le t<\infty$, write
\[
\b_k^\cX(s,t)
:=\b_k\bigl(X(s)\subset X(t)\bigr).
\]

The persistence diagram and the persistent Betti numbers are related as follows:
for every $0\le s\le t<\infty$,
\begin{equation}\label{eq:PD-PB}
\Dgm_k^\cX\bigl([0,s]\times(t,\infty]\bigr)
=\b_k^\cX(s,t);
\end{equation}
see~\cite{CEH2007,ELZ2002}.

\begin{rmk}\label{rmk:unique}
  The quadrants $[0,s]\times(t,\infty]$ with $0<s<t<\infty$ form a $\pi$-system generating the Borel $\sg$-algebra of $\Dl$.
  Each such quadrant is relatively compact under the identification $[0,\infty]^2\simeq[0,1]^2$ introduced in Section~\ref{sec:intro}.
  Moreover, the quadrants with rational $s$ and $t$ cover $\Dl$.
  Since every $\sfD\in\cM(\Dl)$ is finite on these quadrants, the uniqueness theorem for measures shows that $\sfD$ is uniquely determined by its values on them.
  In particular, $\Dgm_k^\cX$ is the unique Radon measure on $\Dl$ satisfying~\eqref{eq:PD-PB} for every $0<s<t<\infty$.
  We will use this observation to identify limiting persistence-diagram measures from persistent Betti numbers.
\end{rmk}

\subsection{Laplacians and Their Compressions}
\label{ssec:comp_L}
In this subsection, we review the combinatorial Laplacians and introduce the notion of the \emph{compressed Laplacian} for simplicial pairs.
For the spectral analysis, we now take the coefficient field to be $\C$.
This does not change the relevant ranks for finite simplicial pairs, since their boundary matrices have integer entries and hence the same rank over every field of characteristic zero.
Let $X$ be a simplicial complex.
A $\C$-valued function $\phi$ on $\hat F_k(X)$ is called a \emph{$k$-cochain} of $X$ if it is alternating; that is,
\[
\phi(v_{\pi(0)},v_{\pi(1)},\ldots,v_{\pi(k)})=(\sgn\pi)\,\phi(v_0,v_1,\ldots,v_k)
\]
for every $(v_0,v_1,\ldots,v_k)\in\hat F_k(X)$ and every permutation $\pi$ of $\{0,1,\ldots,k\}$.
We denote by $C^k(X)=C^k(X;\C)$ the $\C$-vector space of all $k$-cochains of $X$, and refer to it as the \emph{$k$-cochain group} of $X$.
Note that $C^{-1}(X)=\C^{\{\emptyset\}}\cong\C$.
This definition of $C^k(X;\C)$ differs from that in Section~\ref{ssec:H_PB_PD}, where $C^k(X;\C)$ is defined as $\Hom(C_k(X;\C),\C)$.
However, the two definitions are canonically isomorphic, and are often identified; we adopt this convention here.

For $k\ge-1$, we consider the \emph{$\ell^2$-cochain group} defined as the Hilbert space
\[
C^k_{(2)}(X):=\left\{\phi\in C^k(X)\relmiddle|\sum_{\sg\in\hat F_k(X)}|\phi(\sg)|^2<\infty\right\},
\]
equipped with the inner product
\[
\la\phi,\phi'\ra_k:=\frac1{(k+1)!}\sum_{\sg\in\hat F_k(X)}\phi(\sg)\overline{\phi'(\sg)}.
\]
For the operator-theoretic constructions below, fix $k\in\N_0$ and assume that $\deg_X(\sg)<\infty$ for every $\sg\in F_k(X)$.
This assumption ensures that $d_k^X$ maps finitely supported $k$-cochains to finitely supported $(k+1)$-cochains.
The \emph{$\ell^2$-coboundary operator} $d_k^X\colon C^k_{(2)}(X)\to C^{k+1}_{(2)}(X)$ is defined on the dense domain
\[
C^k_c(X):=\{\phi\in C^k(X)\mid|\Supp(\phi)|<\infty\}
\]
by the linear extension of
\[
\bigl(d_k^X\phi\bigr)(v_0,\ldots,v_{k+1}):=\sum_{i=0}^{k+1} (-1)^i\phi(v_0,\ldots,v_{i-1},v_{i+1},\ldots,v_{k+1})
\]
for $\phi\in C^k_c(X)$ and $(v_0,\ldots,v_{k+1})\in\hat F_{k+1}(X)$.
Let $\bigl(d_k^X\bigr)^*$ denote the adjoint operator of $d_k^X$ with respect to the inner products $\la\cdot,\cdot\ra_k$ and $\la\cdot,\cdot\ra_{k+1}$.
By the definition of adjoint operators, it is straightforward to verify that $\Dom\bigl(\bigl(d_k^X\bigr)^*\bigr)\supset C^{k+1}_c(X)$, and that for every $\psi\in C^{k+1}_c(X)$ and $\hat\sg=(v_0,v_1,\ldots,v_k)\in\hat F_k(X)$,
\[
\bigl(\bigl(d_k^X\bigr)^*\psi\bigr)(\hat\sg)
=\sum_{\substack{v\in V(X)\colon\\\sg\cup\{v\}\in F_{k+1}(X)}}\psi(v,v_0,v_1,\ldots,v_k).
\]
The $k$th \emph{up Laplacian}~$\Dl_k^\up(X)$ on $C^k_{(2)}(X)$ is then defined by $\Dl_k^\up(X):=\bigl(d_k^X\bigr)^*d_k^X$.
It holds that $\Dom\bigl(\Dl_k^\up(X)\bigr)=C^k_c(X)$, and that for every $\phi\in C^k_c(X)$ and $\hat\sg=(v_0,v_1,\ldots,v_k)\in\hat F_k(X)$,
\begin{equation}\label{eq:L_expression}
\bigl(\Dl_k^\up(X)\phi\bigr)(\hat\sg)
=\deg_X(\sg)\phi(\hat\sg)
-\!\!\!\!\!\!\!\!\sum_{\substack{v\in V(X)\colon\\\sg\cup\{v\}\in F_{k+1}(X)}}\!\!\sum_{i=0}^k(-1)^i\phi(v,v_0,\ldots,v_{i-1},v_{i+1},\ldots,v_k),
\end{equation}
where $\deg_X(\sg):=|\{\tau\in F_{k+1}(X)\mid\tau\supset\sg\}|$ is the \emph{degree} of $\sg$ in $X$.
Consequently, $\Dl_k^\up(X)$ is a densely defined symmetric positive semidefinite operator on $C^k_{(2)}(X)$ with respect to the inner product $\la\cdot,\cdot\ra_k$.

\smallskip
We now consider a subcomplex $K\subset X$ and the closed subspace
\[
C^k_{(2)}(X\setminus K)
:=\{\phi\in C^k_{(2)}(X)\mid\Supp(\phi)\subset\hat F_k(X)\setminus\hat F_k(K)\}.
\]
Then we have the orthogonal decomposition $C^k_{(2)}(X)=C^k_{(2)}(K)\oplus C^k_{(2)}(X\setminus K)$.
We define the restricted $\ell^2$-coboundary operator
\[
d_k^{X\setminus K}\colon C^k_{(2)}(X\setminus K)\longrightarrow C^{k+1}_{(2)}(X\setminus K)
\]
with $\Dom(d_k^{X\setminus K})=C^k_c(X\setminus K):=C^k_{(2)}(X\setminus K)\cap C^k_c(X)$ by
\[
d_k^{X\setminus K}
:=P_{C^{k+1}_{(2)}(X\setminus K)}\,d_k^X\,J_{C^k_{(2)}(X\setminus K)}.
\]
Here, $P_{C^{k+1}_{(2)}(X\setminus K)}$ denotes the orthogonal projection from $C^{k+1}_{(2)}(X)$ onto $C^{k+1}_{(2)}(X\setminus K)$, and $J_{C^k_{(2)}(X\setminus K)}$ denotes the canonical inclusion from $C^k_{(2)}(X\setminus K)$ into $C^k_{(2)}(X)$.
\begin{df}[Compressed Laplacian]
  Define the operator $\Dl_k^\up(X\setminus K)$ on $C^k_{(2)}(X\setminus K)$ by
  \[
  \Dl_k^\up(X\setminus K)
  :=\bigl(d_k^{X\setminus K}\bigr)^*\,d_k^{X\setminus K}
  \]
  with $\Dom\bigl(\Dl_k^\up(X\setminus K)\bigr)=C^k_c(X\setminus K)$, where $\bigl(d_k^{X\setminus K}\bigr)^*$ is the adjoint operator of $d_k^{X\setminus K}$ with respect to the inner products induced on $C^k_{(2)}(X\setminus K)$ and $C^{k+1}_{(2)}(X\setminus K)$.
  The operator $\Dl_k^\up(X\setminus K)$ is a densely defined symmetric positive semidefinite operator, and we call it the $k$th \emph{compressed up Laplacian} of the simplicial pair $K\subset X$.
  If $K$ is the empty complex, i.e., $K=\{\emptyset\}$, then $\Dl_k^\up(X\setminus K)=\Dl_k^\up(X)$.
\end{df}

The following lemma shows that $\Dl_k^\up(X\setminus K)$ is indeed the compression of $\Dl_k^\up(X)$ to the closed subspace $C^k_{(2)}(X\setminus K)$.
\begin{lem}\label{lem:comp_L}
  Let $k\in\N_0$, and let $K\subset X$ be a simplicial pair such that $\deg_X(\sg)<\infty$ for every $\sg\in F_k(X)$.
  The following hold:
  \begin{enumerate}
    \item It holds that
    \[
    \Dl_k^\up(X\setminus K)
    =P_{C^k_{(2)}(X\setminus K)}\,\Dl_k^\up(X)\,J_{C^k_{(2)}(X\setminus K)}.
    \]
    \item If $X$ is finite, then
    \[
    \b_k(K\subset X)
    =\nullity\bigl(\Dl_k^\up(X)\bigr)-\nullity\bigl(\Dl_k^\up(X\setminus K)\bigr)-\dim B^k(K).
    \]
    This is Lemma~\ref{lem:PB} rewritten in terms of the Laplacians.
  \end{enumerate}
\end{lem}
\begin{proof}
  (1)~Write
  \[
  P:=P_{C^k_{(2)}(X\setminus K)}
  \quad\text{and}\quad
  J:=J_{C^k_{(2)}(X\setminus K)}.
  \]
  Since $K$ is a subcomplex, $d_k^XJ\phi$ vanishes on every $(k+1)$-simplex of $K$ for each $\phi\in C^k_c(X\setminus K)$.
  Using $J^*=P$, we therefore obtain, for $\phi,\phi'\in C^k_c(X\setminus K)$,
  \begin{align*}
  \bigl\la\Dl_k^\up(X\setminus K)\phi,\phi'\bigr\ra_k
  &=\bigl\la d_k^{X\setminus K}\phi,d_k^{X\setminus K}\phi'\bigr\ra_{k+1}\\
  &=\bigl\la d_k^XJ\phi,d_k^XJ\phi'\bigr\ra_{k+1}\\
  &=\bigl\la P\Dl_k^\up(X)J\phi,\phi'\bigr\ra_k.
  \end{align*}
  Since $C^k_c(X\setminus K)$ is dense in $C^k_{(2)}(X\setminus K)$ and both operators have domain $C^k_c(X\setminus K)$, the desired equality follows.
  
  \smallskip
  (2)~When $X$ is finite, we have $C^k_{(2)}(X)=C^k(X)$ and $C^k_{(2)}(X\setminus K)=C^k(X\setminus K)$ under the canonical identification.
  Moreover, the definitions of $d_k^X$ and $d_k^{X\setminus K}$ in this subsection coincide with those in Section~\ref{ssec:H_PB_PD} with $\F=\C$.
  A straightforward linear algebra calculation yields
  \[
  \Ker\bigl(\Dl_k^\up(X)\bigr)=\Ker\bigl(d_k^X\bigr)
  \quad\text{and}\quad
  \Ker\bigl(\Dl_k^\up(X\setminus K)\bigr)=\Ker\bigl(d_k^{X\setminus K}\bigr).
  \]
  Therefore the conclusion follows from Lemma~\ref{lem:PB}.
\end{proof}

\subsection{Local Weak Convergence for Simplicial Pairs}
\label{ssec:LWC}
In this subsection, we introduce the notion of \emph{local weak convergence} for simplicial pairs.
For a simplicial pair $K\subset X$, we view the subcomplex $K$ as a ``marking'' of $X$.
Within this framework, we refer to $K$ as the \emph{marking complex} and $X$ as the \emph{ambient complex}.
All notions introduced below also apply to the case without marking complexes by simply ignoring the marking complexes, which corresponds to the framework in~\cite{Kan2022}.

A \emph{$k$-rooted simplicial pair} is a simplicial pair $\bm X=(K\subset X)$ together with a distinguished $k$-simplex $\sg\in F_k(X)$, called the \emph{$k$-root}.
We denote it by $(\bm X,\sg)$, or by $(K\subset X,\sg)$ when it is convenient to specify the simplicial complexes explicitly.

Given a $k$-rooted simplicial pair $(K\subset X,\sg)$, we define the sequence $\{(K\subset X,\sg)_l\}_{l=0}^\infty$ of $k$-rooted balls centered at the root inductively as follows.
We first set
\[
(K\subset X,\sg)_0=(2^\sg\,\cap K\subset2^\sg,\sg),
\]
where $2^\sg$ denotes the simplicial complex consisting of $\sg$ and all its faces.
For each $l\ge1$, define the ambient complex of $(K\subset X,\sg)_l$ as the subcomplex of $X$ obtained from the ambient complex of $(K\subset X,\sg)_{l-1}$ by adding every $(k+1)$-simplex $\tau\in F_{k+1}(X)$ that shares at least one $k$-face with it, together with all faces of $\tau$.
The marking complex of $(K\subset X,\sg)_l$ is then defined as the intersection of its ambient complex with $K$.
A simplex $\tau$ in $X$ is said to be at \emph{distance} $l$ from the $k$-root $\sg$, denoted by $\dist_\sg(\tau)=l$, if $\tau$ belongs to the ambient complex of $(K\subset X,\sg)_l$ but not that of $(K\subset X,\sg)_{l-1}$.
Moreover, we define
\[
(K\subset X,\sg)_\infty
:=\bigcup_{l=0}^\infty(K\subset X,\sg)_l,
\]
where the union is taken componentwise over the marking and ambient complexes, respectively.
We call $(K\subset X,\sg)_\infty$ the \emph{connected component} of $\sg$ in the pair $K\subset X$.
Here, connectivity is defined via adjacency of $k$-simplices through common $(k+1)$-cofaces.
We say that a $k$-rooted simplicial pair $(K\subset X,\sg)$ is \emph{connected} if its balls exhaust the entire pair; that is, $(K\subset X,\sg)_\infty=(K\subset X,\sg)$.
In particular, the ambient complex of every connected $k$-rooted simplicial pair has dimension at most $k+1$.

Two $k$-rooted simplicial pairs $(K\subset X,\sg)$ and $(K'\subset X',\sg')$ are said to be \emph{isomorphic}, denoted by
\[
(K\subset X,\sg)\simeq(K'\subset X',\sg'),
\]
if there exists a bijection $h\colon V(X)\to V(X')$ such that, for every
$\tau\subset V(X)$,
\[
\tau\in X\Longleftrightarrow h(\tau)\in X',
\quad \tau\in K\Longleftrightarrow h(\tau)\in K',
\quad\text{and}\quad
h(\sg)=\sg'.
\]
We write $[K\subset X,\sg]$ for the isomorphism class of a $k$-rooted simplicial pair $(K\subset X,\sg)$.
Furthermore, we denote by $[K\subset X,\sg]_l$ the isomorphism class of $(K\subset X,\sg)_l$ for $l\in\N_0\cup\{\infty\}$.
Similarly, we write $[\bm X,\sg]$ and $[\bm X,\sg]_l$ when $\bm X=(K\subset X)$.
\begin{df}[Local convergence]
  Let $\cS_k$ denote the set of all isomorphism classes $[K\subset X,\sg]$ of connected $k$-rooted simplicial pairs satisfying $\deg_X(\sg')<\infty$ for every $\sg'\in F_k(X)$.
  We define $d_{\loc}\colon\cS_k\times\cS_k\to[0,\infty)$ by
  \[
  d_{\loc}([\bm X,\sg],[\bm X',\sg']):=2^{-L},
  \]
  where
  \[
  L:=\sup\{l\in\N_0\mid(\bm X,\sg)_l\simeq(\bm X',\sg')_l\}.
  \]
  By convention, $2^{-\infty}:=0$, and we set $L=0$ if the supremum is taken over the empty set.
  It is straightforward to verify that $d_{\loc}$ is an ultrametric on $\cS_k$; that is, for any $[\bm X_1,\sg_1],[\bm X_2,\sg_2],[\bm X_3,\sg_3]\in\cS_k$,
  \[
  d_{\loc}([\bm X_1,\sg_1],[\bm X_3,\sg_3])
  \le\max\{d_{\loc}([\bm X_1,\sg_1],[\bm X_2,\sg_2]),d_{\loc}([\bm X_2,\sg_2],[\bm X_3,\sg_3])\}.
  \]
  The function $d_{\loc}$ is called the \emph{local metric}, and the induced notion of convergence is referred to as \emph{local convergence}.
\end{df}

In what follows, let $\cP(\cS_k)$ denote the set of all Borel probability measures on $(\cS_k,d_{\loc})$, equipped with the weak topology (i.e.,~the topology of weak convergence).
For a sequence $(\nu_n)_{n=1}^\infty$ in $\cP(\cS_k)$ and $\nu\in\cP(\cS_k)$, we write
\[
\nu_n\xrightarrow{\mathrm{w}}\nu
\quad\text{as $n\to\infty$}
\]
if $\nu_n$ converges weakly to $\nu$.

The following lemma summarizes basic properties of the metric space $(\cS_k,d_{\loc})$.
\begin{lem}\label{lem:Sk_basics}
  The following hold:
  \begin{enumerate}
    \item The space $(\cS_k,d_{\loc})$ is a complete separable metric space $($and hence $\cP(\cS_k)$ is also separable and completely metrizable$)$.
    \item Let $(\nu_n)_{n=1}^\infty$ be a sequence in $\cP(\cS_k)$, and let $\nu\in\cP(\cS_k)$.
    Then $\nu_n\xrightarrow{\mathrm{w}}\nu$ if and only if, for every $[\bm Y,\tau]\in\cS_k$ and $l\in\N$,
    \[
    \lim_{n\to\infty}\nu_n(\{\a\in\cS_k\mid\a_l=[\bm Y,\tau]_l\})
    =\nu(\{\a\in\cS_k\mid\a_l=[\bm Y,\tau]_l\}).
    \]
  \end{enumerate}
\end{lem}
Note that the set $\{\a\in\cS_k\mid\a_l=[\bm Y,\tau]_l\}$ is the closed ball of radius $2^{-l}$ centered at $[\bm Y,\tau]$ in $\cS_k$.
Lemma~\ref{lem:Sk_basics}(2) is a direct consequence of standard properties of separable ultrametric spaces.
Indeed, the necessity follows from the fact that every ball of positive radius in an ultrametric space is both open and closed.
The sufficiency follows from the fact that every open set in a separable ultrametric space can be written as a countable union of disjoint balls.
To prove completeness, note that the rooted balls of any fixed radius eventually stabilize along a Cauchy sequence.
Choosing compatible representatives of these stabilized balls and taking their union gives the desired connected rooted pair.
Separability follows because the isomorphism classes of finite connected $k$-rooted simplicial pairs form a countable dense subset.

\smallskip
We use weak convergence in $\cP(\cS_k)$ to define local weak convergence of finite unrooted simplicial pairs.
Given a finite simplicial pair $\bm X=(K\subset X)$ with $\dim X\ge k$, we define $U_k(\bm X)=U_k(K\subset X)\in\cP(\cS_k)$ by
\[
U_k(\bm X):=\frac1{f_k(X)}\sum_{\sg\in F_k(X)}\dl_{[\bm X,\sg]_\infty}.
\]
Here, $\dl_{[\bm X,\sg]_\infty}$ denotes the Dirac measure at $[\bm X,\sg]_\infty\in\cS_k$.
The probability measure $U_k(\bm X)$ on $\cS_k$ can be interpreted as the distribution of the local structure around a uniformly chosen $k$-simplex in $X$.
Using this notation, the local weak convergence of a sequence of finite simplicial pairs is defined as follows. 
\begin{df}[Local weak convergence]
  Let $(\bm X_n)_{n=1}^\infty$ be a sequence of finite simplicial pairs $\bm X_n=(K_n\subset X_n)$ with $\dim X_n\ge k$, and let $\nu\in\cP(\cS_k)$.
  We say that $\bm X_n$ \emph{converges locally weakly} to $\nu$ if $U_k(\bm X_n)\xrightarrow{\mathrm{w}}\nu$ as $n\to\infty$.
  In this case, we write
  \[
  \bm X_n\xrightarrow{\LW}\nu
  \quad\text{as $n\to\infty$.}
  \]
\end{df}

\begin{rmk}
  By Lemma~\ref{lem:Sk_basics}(2), the above definition is equivalent to stating that, for every $k$-rooted simplicial pair $(\bm Y,\tau)$ and $l\in\N$,
  \[
  \lim_{n\to\infty}\frac{|\{\sg\in F_k(X_n)\mid(\bm X_n,\sg)_l\simeq(\bm Y,\tau)_l\}|}{f_k(X_n)}
  =\nu(\{\a\in\cS_k\mid\a_l=[\bm Y,\tau]_l\}).
  \]
\end{rmk}

\subsection{Poisson Galton--Watson Simplicial Trees and Their Thinning}
\label{ssec:PGW}
We conclude the preliminaries section by reviewing a higher-dimensional generalization of the Poisson Galton--Watson tree and introducing its thinned variant, which arises as the local weak limit of the multiparameter complex pair; see Section~\ref{sec:conv_PB}.

Let $k\in\N_0$.
We construct a $k$-rooted simplicial complex $(T,\sg_o)$ as follows.
We begin with the simplicial complex $2^{\sg_o}$ generated by a $k$-simplex $\sg_o$.
For each step $l\in\N_0$, every $k$-simplex $\sg$ at distance $l$ from $\sg_o$ generates $m_\sg\in\N_0$ new vertices $v_1,v_2,\ldots,v_{m_\sg}$, and we add the $(k+1)$-simplices $\sg\cup\{v_i\}$, $i=1,2,\ldots,m_\sg$, together with all their faces, to the complex constructed up to step $l$.
We call $(T,\sg_o)$ a \emph{$k$-rooted simplicial tree}.
When $k$ is understood, we refer to $T$ simply as a \emph{simplicial tree}.\footnote{
There are several notions of simplicial tree, depending on which properties of tree graphs one chooses to emphasize.
Every finite simplicial tree $T$ in the present sense satisfies $f_{k+1}(T)=f_0(T)-(k+1)$~(see, e.g.,~\cite[Lemma~2.3]{KT2025}).
When $k=0$, this reduces to the well-known identity $|E(T)|=|V(T)|-1$ for finite trees.}

When each $m_\sg$ in the above construction is an independent $\Poisson(c)$ random variable, the resulting complex is called the \emph{$k$-rooted Poisson Galton--Watson (PGW) simplicial tree} with parameter $c$~(see~\cite{ALLM2013,Kan2022,LP2016b}).
We denote it by $(\PGW_k(c),\sg_o)$, defined on a probability space $(\Om',\cF',\P')$.

We next introduce the notion of a \emph{thinned $k$-rooted PGW simplicial tree}.
\begin{df}[thinned $k$-rooted PGW simplicial tree]
\label{df:PGW}
  Let $k\in\N_0$, $c>0$, and $\bm\theta=(\theta_j)_{j=0}^{k+1}\in[0,1]^{k+2}$.
  Given a $k$-rooted PGW simplicial tree $(\PGW_k(c),\sg_o)$, assign to each nonempty $j$-simplex $\rho$~($0\le j\le k+1$) an independent random variable
  \[
  \eta_{\bm\theta}(\rho)\sim\Bernoulli(\theta_j),
  \]
  and define the subcomplex
  \[
  \PGW_k(c,\bm\theta)
  :=\{\sg\in\PGW_k(c)\mid\text{$\eta_{\bm\theta}(\rho)=1$ for every nonempty face $\rho\subset\sg$}\}.
  \]
  We call $\PGW_k(c,\bm\theta)$ the \emph{$\bm\theta$-thinned $k$-rooted PGW simplicial tree} with parameter $c$.
  In the two extreme cases, $\PGW_k(c,\bm1)=\PGW_k(c)$ and $\PGW_k(c,\bm0)=\{\emptyset\}$.
\end{df}
For convenience, we write the random simplicial pair
\[
\PGW_k(c,\bm\theta\le\bm1)
:=(\PGW_k(c,\bm\theta)\subset\PGW_k(c)).
\]
Since every $k$-simplex has finitely many offspring cofaces and at most one parent coface, almost surely $\deg_{\PGW_k(c)}(\sg)<\infty$ for every $\sg\in F_k(\PGW_k(c))$.
Note that the equivalence class $[\PGW_k(c,\bm\theta\le\bm1),\sg_o]$ defines an $\cS_k$-valued random variable on $(\Om',\cF',\P')$, and we denote its law by $\nu_k(c,\bm\theta\le\bm1)\in\cP(\cS_k)$.
  
\section{Rooted Spectral Measures of Compressed Laplacians}
\label{sec:RSM_CompL}
In this section, we develop several spectral properties of the compressed up Laplacians defined in Section~\ref{ssec:comp_L}.
These properties play a key role in the analysis of the asymptotics of persistent Betti numbers in Section~\ref{sec:conv_PB}.
In Section~\ref{ssec:RSM}, we introduce the notion of rooted spectral measures for rooted simplicial pairs and show their continuity with respect to the local metric.
In Section~\ref{ssec:conv_ARSM0}, we establish the continuity of the averaged atomic mass at zero along locally weakly convergent sequences of measures arising from finite simplicial pairs, provided that the corresponding average degrees are uniformly bounded.

\subsection{Rooted Spectral Measures}
\label{ssec:RSM}
Let $k\in\N_0$, and let $K\subset X$ be a simplicial pair such that $\deg_X(\sg)<\infty$ for every $\sg\in F_k(X)$.
In what follows, fix an orientation of each $k$-simplex in $X$, and let $(e_\sg)_{\sg\in F_k(X)}$ be the resulting canonical orthonormal basis of $C^k_{(2)}(X)$.
The rooted spectral measures defined below do not depend on this choice.
Furthermore, we use the additional notation
\[
F_k(X\setminus K):=F_k(X)\setminus F_k(K)
\quad\text{and}\quad
f_k(X\setminus K):=|F_k(X\setminus K)|.
\]
To define the spectral measure of the $k$th compressed up Laplacian $\Dl_k^\up(X\setminus K)$, we assume that $\Dl_k^\up(X\setminus K)$ is \emph{essentially self-adjoint}; that is, its closure $\overline{\Dl_k^\up(X\setminus K)}$ is self-adjoint.
Then, for every $\sg\in F_k(X\setminus K)$, there exists a unique Borel probability measure $\mu$ on $\R$ such that
\begin{equation}\label{eq:RSM}
\bigl\la\bigl(\,\overline{\Dl_k^\up(X\setminus K)}-zI\bigr)^{-1}e_\sg,e_\sg\bigr\ra_k
=\int_\R\frac1{\lm-z}\,d\mu(\lm)
\end{equation}
for all $z\in\C_+:=\{z\in\C\mid\Im z>0\}$.
The existence of such a measure is guaranteed by the spectral theorem~(see, e.g.,~\cite[Theorem~VIII.6]{RS1972}) for self-adjoint operators, and its uniqueness follows from the injectivity of the Stieltjes transform:
\[
\cP(\R)\ni\mu\mapsto S_\mu(z):=\int_\R\frac1{\lm-z}\,d\mu(\lm),
\quad z\in\C_+.
\]
We denote the unique probability measure by $\mu_{(K\subset X,\sg)}$ and refer to it as the \emph{rooted spectral measure} associated with the $k$-rooted simplicial pair $(K\subset X,\sg)$.
When $\sg\in K$, we set $\mu_{(K\subset X,\sg)}:=0$, the zero measure on $\R$.

\begin{ex}[Rooted spectral measure in the finite-dimensional case]
  Suppose that
  \[
  1\le f_k(X\setminus K)<\infty.
  \]
  Let $\lm_i$ and $\phi_i$~($i=1,2,\ldots,f_k(X\setminus K)$) be the eigenvalues of $\Dl_k^\up(X\setminus K)$ and a corresponding orthonormal basis of eigenvectors, respectively.
  Then, for each $\sg\in F_k(X\setminus K)$, a straightforward calculation yields
  \[
  \mu_{(K\subset X,\sg)}
  =\sum_{i=1}^{f_k(X\setminus K)}|\la e_\sg,\phi_i\ra_k|^2\dl_{\lm_i}.
  \]
  The \emph{empirical spectral distribution} of $\Dl_k^\up(X\setminus K)$ is defined by
  \[
  \mu_{\Dl_k^\up(X\setminus K)}
  :=\frac1{f_k(X\setminus K)}\sum_{i=1}^{f_k(X\setminus K)}\dl_{\lm_i}.
  \]
  It is the average of the rooted spectral measures:
  \begin{equation}\label{eq:ESD_decomp}
  \begin{aligned}
  \frac1{f_k(X\setminus K)}\sum_{\sg\in F_k(X\setminus K)}\mu_{(K\subset X,\sg)}
  &=\frac1{f_k(X\setminus K)}\sum_{\sg\in F_k(X\setminus K)}\Biggl(\sum_{i=1}^{f_k(X\setminus K)}|\la e_\sg,\phi_i\ra_k|^2\dl_{\lm_i}\Biggr)\\
  &=\frac1{f_k(X\setminus K)}\sum_{i=1}^{f_k(X\setminus K)}\Biggl(\sum_{\sg\in F_k(X\setminus K)}|\la\phi_i,e_\sg\ra_k|^2\Biggr)\dl_{\lm_i}\\
  &=\frac1{f_k(X\setminus K)}\sum_{i=1}^{f_k(X\setminus K)}\dl_{\lm_i}\\
  &=\mu_{\Dl_k^\up(X\setminus K)}.
  \end{aligned}
  \end{equation}
  In particular, evaluation of~\eqref{eq:ESD_decomp} at ${0}$ gives
  \begin{equation}\label{eq:loc_contrib}
  \nullity\bigl(\Dl_k^\up(X\setminus K)\bigr)
  =\sum_{\sg\in F_k(X\setminus K)}\mu_{(K\subset X,\sg)}(\{0\})
  =\sum_{\sg\in F_k(X)}\mu_{(K\subset X,\sg)}(\{0\}).
  \end{equation}
  Thus, $\mu_{(K\subset X,\sg)}(\{0\})$ may be viewed as the local contribution of $\sg$ to the nullity.
\end{ex}

\smallskip
A further useful property of the rooted spectral measures associated with $k$-rooted simplicial pairs is their continuity with respect to the local metric.
Let $\cS_k^{\esa}\subset\cS_k$ denote the set of all equivalence classes $[K\subset X,\sg]\in\cS_k$ such that $\sg\in K$ or $\Dl_k^\up(X\setminus K)$ is essentially self-adjoint.
\begin{lem}\label{lem:esa_Borel}
  The set $\cS_k^{\esa}$ is a Borel subset of $\cS_k$.
\end{lem}
Since the proof uses a technical measurability argument, we defer it to Section~S1.1 of the Supplementary Material~\cite{HKsupp}.
Let $\cM_{\le1}(\R)$ denote the space of Borel measures on $\R$ with total mass at most one, equipped with the weak topology.
We define a map
\[
\RSM\colon\cS_k^{\esa}\longrightarrow\cM_{\le1}(\R)
\]
by
\[
\RSM([K\subset X,\sg]):=\mu_{(K\subset X,\sg)},
\qquad[K\subset X,\sg]\in\cS_k^{\esa}.
\]
Recall that $\mu_{(K\subset X,\sg)}=0$ if $\sg\in K$.

\begin{lem}\label{lem:conti_RSM}
  The map $\RSM\colon\cS_k^{\esa}\to\cM_{\le1}(\R)$ is continuous, where $\cS_k^{\esa}$ is equipped with the metric induced from $\cS_k$.
\end{lem}
\begin{proof}
  Suppose that $[K_n\subset X_n,\sg_n]\to[K\subset X,\sg]$ in $\cS_k^{\esa}$.
  If $\sg\in K$, then $\sg_n\in K_n$ for all sufficiently large $n$, and hence both the limiting measure and all but finitely many measures in the sequence are the zero measure.

  We may therefore assume that $\sg\notin K$; in this case, $\sg_n\notin K_n$ for all sufficiently large $n$.
  Set
  \[
  L_n:=\overline{\Dl_k^\up(X_n\setminus K_n)}
  \quad\text{and}\quad
  L:=\overline{\Dl_k^\up(X\setminus K)}.
  \]
  Fix $z\in\C_+$ and $\eps>0$.
  Since $C_c^k(X\setminus K)$ is a core for $L$, the subspace $(L-zI)C_c^k(X\setminus K)$ is dense in $C_{(2)}^k(X\setminus K)$.
  Hence, there exists $\phi\in C_c^k(X\setminus K)$ such that
  \[
  \bigl\|e_\sg-(L-zI)\phi\bigr\|_k\le\eps.
  \]
  Choose $l\in\N$ such that $\Supp(\phi)$ is contained in the ambient complex of $(K\subset X,\sg)_{l-1}$.
  By local convergence, for all sufficiently large $n$, there exists an isomorphism
  \[
  (K_n\subset X_n,\sg_n)_l\simeq(K\subset X,\sg)_l.
  \]
  Let $\phi_n\in C_c^k(X_n\setminus K_n)$ be the cochain obtained by transporting $\phi$ under this isomorphism.
  By the explicit expression~\eqref{eq:L_expression} together with Lemma~\ref{lem:comp_L}(1), the vectors
  \[
  e_{\sg_n}-(L_n-zI)\phi_n
  \quad\text{and}\quad
  e_\sg-(L-zI)\phi
  \]
  correspond to each other under the same isomorphism.
  Consequently, we have
  \[
  \la\phi_n,e_{\sg_n}\ra_k=\la\phi,e_\sg\ra_k
  \quad\text{and}\quad
  \bigl\|e_{\sg_n}-(L_n-zI)\phi_n\bigr\|_k\le\eps
  \]
  for all sufficiently large $n$.
  Since $L_n$ and $L$ are self-adjoint,
  \[
  \bigl\|(L_n-zI)^{-1}\bigr\|\le\frac1{\Im z}
  \quad\text{and}\quad
  \bigl\|(L-zI)^{-1}\bigr\|\le\frac1{\Im z}.
  \]
  Therefore,
  \begin{align*}
  &\bigl|\bigl\la(L_n-zI)^{-1}e_{\sg_n},e_{\sg_n}\bigr\ra_k-\bigl\la(L-zI)^{-1}e_\sg,e_\sg\bigr\ra_k\bigr|\\
  &=\bigl|\bigl\la(L_n-zI)^{-1}e_{\sg_n}-\phi_n,e_{\sg_n}\bigr\ra_k-\bigl\la(L-zI)^{-1}e_\sg-\phi,e_\sg\bigr\ra_k\bigr|\\
  &\le\bigl\|(L_n-zI)^{-1}e_{\sg_n}-\phi_n\bigr\|_k+\bigl\|(L-zI)^{-1}e_\sg-\phi\bigr\|_k\\
  &\le\frac1{\Im z}\bigl\|e_{\sg_n}-(L_n-zI)\phi_n\bigr\|_k+\frac1{\Im z}\bigl\|e_\sg-(L-zI)\phi\bigr\|_k
  \le\frac{2\eps}{\Im z}
  \end{align*}
  for all sufficiently large $n$.
  Since $\eps>0$ was arbitrary, we obtain
  \[
  \la(L_n-zI)^{-1}e_{\sg_n},e_{\sg_n}\ra_k
  \xrightarrow[n\to\infty]{}\la(L-zI)^{-1}e_\sg,e_\sg\ra_k
  \]
  for every $z\in\C_+$.
  Therefore, by~\eqref{eq:RSM}, the Stieltjes transform of $\mu_{(K_n\subset X_n,\sg_n)}$ converges pointwise to that of $\mu_{(K\subset X,\sg)}$.
  By the continuity theorem for Stieltjes transforms, this is equivalent to weak convergence of $\mu_{(K_n\subset X_n,\sg_n)}$ to $\mu_{(K\subset X,\sg)}$.
\end{proof}

\smallskip
We also consider the rooted spectral measures averaged with respect to probability measures on rooted simplicial pairs.
Since $\cS_k^{\esa}$ is a Borel subset of $\cS_k$ by Lemma~\ref{lem:esa_Borel}, the space $\cP(\cS_k^{\esa})$ can be identified with the subset of $\cP(\cS_k)$ consisting of all probability measures concentrated on $\cS_k^{\esa}$.
Under this identification, the weak topology on $\cP(\cS_k^{\esa})$ agrees with the topology induced from $\cP(\cS_k)$.
We use this identification throughout without further comment.
Define the map
\[
\overline\RSM\colon\cP(\cS_k^{\esa})\longrightarrow\cM_{\le1}(\R)
\]
by
\[
\overline\RSM(\nu)(\cdot)
:=\int_{\cS_k^{\esa}}\mu_{(K\subset X,\sg)}(\cdot)\,d\nu([K\subset X,\sg]),
\qquad\nu\in\cP(\cS_k^{\esa}).
\]
The averaged rooted spectral measure map $\overline\RSM$ is continuous with respect to the weak topology on $\cP(\cS_k^{\esa})$.
Indeed, for every bounded continuous function $g\colon\R\to\R$,
\[
\int_\R g\,d\overline\RSM(\nu)
=\int_{\cS_k^{\esa}}\left(\int_\R g\,d\mu_{(K\subset X,\sg)}\right)d\nu([K\subset X,\sg]),
\]
and the inner integral defines a bounded continuous function on
$\cS_k^{\esa}$ by Lemma~\ref{lem:conti_RSM}.

\subsection{Convergence of Averaged Atomic Masses at Zero}
\label{ssec:conv_ARSM0}
In this subsection, we study the atomic mass at zero of rooted spectral measures and its average with respect to probability measures on rooted simplicial pairs.
Define the maps
\[
\RSM_0\colon\cS_k^{\esa}\longrightarrow[0,1]
\quad\text{and}\quad
\overline\RSM_0\colon\cP(\cS_k^{\esa})\longrightarrow[0,1]
\]
by
\begin{align*}
\RSM_0([K\subset X,\sg])
&:=\mu_{(K\subset X,\sg)}(\{0\}),
\qquad [K\subset X,\sg]\in\cS_k^{\esa},
\shortintertext{and}
\overline\RSM_0(\nu)
&:=\int_{\cS_k^{\esa}}\mu_{(K\subset X,\sg)}(\{0\})\,d\nu([K\subset X,\sg]),
\qquad \nu\in\cP(\cS_k^{\esa}),
\end{align*}
respectively.
In contrast to the two rooted spectral measure maps $\RSM$ and $\overline\RSM$, the corresponding atomic-mass maps $\RSM_0$ and $\overline\RSM_0$ are not necessarily continuous and are in general only upper semicontinuous.
Indeed, this follows from the continuity of $\RSM$ and $\overline\RSM$, together with the upper semicontinuity of the map $\mu\longmapsto\mu(\{0\})$, which follows from the portmanteau lemma, since $\{0\}$ is closed in $\R$.

Nevertheless, the following proposition shows that $\overline\RSM_0$ is continuous along locally weakly convergent sequences of finite simplicial pairs, provided that the corresponding average degrees are uniformly bounded.
\begin{prop}\label{prop:conti_ARSM0}
  Let $(K_n\subset X_n)_{n=1}^\infty$ be a sequence of finite simplicial pairs with $\dim X_n\ge k$, and let $\nu\in\cP(\cS_k^{\esa})$.
  Assume that
  \begin{equation}\label{eq:conti_ARSM0_1}
  (K_n\subset X_n)\xrightarrow{\LW}\nu
  \quad\text{as $n\to\infty$}
  \end{equation}
  and
  \[
  \sup_{n\in\N}\frac1{f_k(X_n)}\sum_{\sg\in F_k(X_n)}\deg_{X_n}(\sg)<\infty.
  \]
  Then
  \[
  \frac{\nullity\bigl(\Dl_k^\up(X_n\setminus K_n)\bigr)}{f_k(X_n)}
  \longrightarrow\overline\RSM_0(\nu)
  \quad\text{as $n\to\infty$.}
  \]
\end{prop}

\begin{rmk}\label{rmk:integrality}
  Proposition~\ref{prop:conti_ARSM0} simplifies the local weak convergence approach to the limit theorems for Betti numbers in~\cite{Kan2022,LP2016a}.
  The proofs there combine spectral upper bounds with lower bounds obtained by elementary collapses.
  In these random settings, the convergence-in-probability version in Corollary~\ref{cor:conti_ARSM0} below, applied with the empty marking complex $\{\emptyset\}$, gives convergence of normalized nullities directly from local weak convergence and control of the average degrees, without a separate argument using elementary collapses.
\end{rmk}

The proof of Proposition~\ref{prop:conti_ARSM0} relies on the fact that the finite compressed up Laplacians are positive semidefinite matrices with integer entries.
Combined with a bound on the corresponding average degrees, this prevents a positive proportion of their nonzero eigenvalues from accumulating at zero.
We first show the required spectral estimate.
\begin{lem}\label{lem:small_eigen}
  Let $K\subset X$ be a finite simplicial pair, and let $\lm_1,\lm_2,\ldots,\lm_{f_k(X\setminus K)}$ be the eigenvalues of $\Dl_k^\up(X\setminus K)$.
  Then, for every $0<\eps<1$,
  \begin{equation}\label{eq:small_eigen}
  \bigl|\bigl\{i=1,2,\ldots,f_k(X\setminus K)\bigm|0<\lm_i\le\eps\bigr\}\bigr|
  \le\frac1{|\log\eps|}\sum_{\sg\in F_k(X)}\deg_X(\sg).
  \end{equation}
\end{lem}
\begin{proof}
  Fix $0<\eps<1$.
  With respect to the canonical orthonormal basis $(e_\sg)_{\sg\in F_k(X\setminus K)}$, the compressed up Laplacian $\Dl_k^\up(X\setminus K)$ is represented by a symmetric positive semidefinite matrix with integer entries.
  Let $m$ denote the number of positive eigenvalues.
  If $m=0$, then~\eqref{eq:small_eigen} is immediate.
  We may therefore assume that $m\ge1$ and, after relabeling, that $\lm_1,\ldots,\lm_m$ are the positive eigenvalues.
  The product $\prod_{i=1}^m\lm_i$ is the absolute value of the first nonzero coefficient of the characteristic polynomial of $\Dl_k^\up(X\setminus K)$.
  It is therefore a positive integer, and in particular
  \[
  \sum_{i=1}^m\log\lm_i\ge0.
  \]
  Write $s:=|\{i=1,2,\ldots,m\mid0<\lm_i\le\eps\}|$.
  Since $\log x\le x$ for every $x>0$, we obtain
  \[
  0\le\sum_{i=1}^m\log\lm_i
  \le s\log\eps+\sum_{\substack{1\le i\le m\colon\\\lm_i>\eps}}\lm_i
  \le-s|\log\eps|+\tr\bigl(\Dl_k^\up(X\setminus K)\bigr).
  \]
  By Lemma~\ref{lem:comp_L}(1) and~\eqref{eq:L_expression},
  \[
  \tr\bigl(\Dl_k^\up(X\setminus K)\bigr)
  =\sum_{\sg\in F_k(X\setminus K)}\deg_X(\sg)
  \le\sum_{\sg\in F_k(X)}\deg_X(\sg).
  \]
  Combining the above estimates proves~\eqref{eq:small_eigen}.
\end{proof}

We now prove Proposition~\ref{prop:conti_ARSM0} using Lemma~\ref{lem:small_eigen}.
\begin{proof}[Proof of Proposition~\ref{prop:conti_ARSM0}]
  Set
  \[
  \mu_n
  :=\overline\RSM\bigl(U_k(K_n\subset X_n)\bigr)
  =\frac1{f_k(X_n)}\sum_{\sg\in F_k(X_n)}\mu_{(K_n\subset X_n,\sg)}
  \]
  and $\mu:=\overline\RSM(\nu)$.
  By~\eqref{eq:loc_contrib},
  \[
  \mu_n(\{0\})
  =\overline\RSM_0\bigl(U_k(K_n\subset X_n)\bigr)
  =\frac{\nullity\bigl(\Dl_k^\up(X_n\setminus K_n)\bigr)}{f_k(X_n)}.
  \]
  Furthermore, by~\eqref{eq:conti_ARSM0_1} and the continuity of $\overline\RSM$, the measure $\mu_n$ converges weakly to $\mu$.
  Let
  \[
  M:=\sup_{n\in\N}\frac1{f_k(X_n)}\sum_{\sg\in F_k(X_n)}\deg_{X_n}(\sg)<\infty.
  \]
  By Lemma~\ref{lem:small_eigen} together with~\eqref{eq:ESD_decomp},
  \[
  \sup_{n\in\N}\mu_n((0,\eps])\le\frac M{|\log\eps|}
  \]
  for every $0<\eps<1$.
  Choose a sequence $(\eps_l)_{l=1}^\infty$ such that $\eps_l\downarrow0$ and $\mu(\{\eps_l\})=0$ for every $l\in\N$.
  Such a sequence exists since a finite Borel measure has at most countably many atoms.
  Fix $l\in\N$.
  Since all the measures under consideration are supported on $[0,\infty)$, we have
  \[
  \mu_n([0,\eps_l])=\mu_n((-\infty,\eps_l])
  \quad\text{and}\quad
  \mu([0,\eps_l])=\mu((-\infty,\eps_l]).
  \]
  The set $(-\infty,\eps_l]$ is a continuity set of $\mu$, and hence the weak convergence implies that
  \[
  \mu_n([0,\eps_l])\xrightarrow[n\to\infty]{}\mu([0,\eps_l]).
  \]
  Consequently,
  \begin{align*}
  \limsup_{n\to\infty}\mu_n(\{0\})
  &\le\limsup_{n\to\infty}\mu_n([0,\eps_l])
  =\mu([0,\eps_l]),\\
  \liminf_{n\to\infty}\mu_n(\{0\})
  &=\liminf_{n\to\infty}\bigl\{\mu_n([0,\eps_l])-\mu_n((0,\eps_l])\bigr\}
  \ge\mu([0,\eps_l])-\frac M{|\log\eps_l|}.
  \end{align*}
  Letting $l\to\infty$ yields
  \[
  \lim_{n\to\infty}\mu_n(\{0\})
  =\mu(\{0\})
  =\overline\RSM_0(\nu),
  \]
  which completes the proof.
\end{proof}

The following convergence-in-probability statement for random finite simplicial pairs is an immediate consequence of Proposition~\ref{prop:conti_ARSM0}.
Here and below, for any sequence $(K_n\subset X_n)_{n=1}^\infty$ of random finite simplicial pairs satisfying $\dim X_n\ge k$ almost surely and any $\nu\in\cP(\cS_k)$, the notation
\[
(K_n\subset X_n)\xrightarrow{\LW,\,\mathrm{p}}\nu
\]
means that $U_k(K_n\subset X_n)$ converges in probability to $\nu$ with respect to the weak topology on $\cP(\cS_k)$.
\begin{cor}\label{cor:conti_ARSM0}
  Let $(K_n\subset X_n)_{n=1}^\infty$ be a sequence of random finite simplicial pairs with $\dim X_n\ge k$ almost surely, and let $\nu\in\cP(\cS_k^{\esa})$.
  Assume that
  \[
  (K_n\subset X_n)\xrightarrow{\LW,\,\mathrm{p}}\nu
  \]
  and that
  \[
  \frac1{f_k(X_n)}\sum_{\sg\in F_k(X_n)}\deg_{X_n}(\sg)\xrightarrow{\mathrm p}C
  \quad\text{for some $C<\infty$}
  \]
  as $n\to\infty$.
  Then
  \[
  \frac{\nullity\bigl(\Dl_k^\up(X_n\setminus K_n)\bigr)}{f_k(X_n)}
  \xrightarrow{\mathrm p}\overline\RSM_0(\nu)
  \quad\text{as $n\to\infty$.}
  \]
\end{cor}
\begin{proof}
  Let $(n_j)_{j=1}^\infty$ be an arbitrary subsequence.
  By the two convergence assumptions, it has a further subsequence $(n_{j_l})_{l=1}^\infty$ along which both asserted convergences hold almost surely.
  Applying Proposition~\ref{prop:conti_ARSM0}, we obtain
  \[
  \frac{\nullity\bigl(\Dl_k^\up(X_{n_{j_l}}\setminus K_{n_{j_l}})\bigr)}{f_k(X_{n_{j_l}})}
  \xrightarrow{\mathrm{a.s.}}\overline\RSM_0(\nu).
  \]
  The subsequence characterization of convergence in probability completes the proof.
\end{proof}

\section{Atomic Mass at Zero on Simplicial Trees}
\label{sec:mass_zero}
In this section, we compute the expected atomic mass at zero of the rooted spectral measure associated with the random $k$-rooted simplicial pair
\[
(\PGW_k(c,\bm\theta\le\bm1),\sg_o)
=(\PGW_k(c,\bm\theta)\subset\PGW_k(c),\sg_o),
\]
introduced in Section~\ref{ssec:PGW}.
Our main result, Proposition~\ref{prop:E_RSM0-h}, identifies this expectation with the constant $h_k(c,\bm\theta)$ introduced in~\eqref{eq:hk_c_theta}.
Thus, this expectation is determined by maximizing an explicit function over the fixed points of a map on a finite-dimensional cube.
We first give the formulas defining this maximization problem and state the main result.

\subsection{The Limiting Constant and the Main Result}
\label{ssec:hk_main}
Let $k\in\N_0$, $c\ge0$, and $\bm\theta=(\theta_j)_{j=0}^{k+1}\in[0,1]^{k+2}$.
We use the conventions $0/0:=0$, $\binom{0}{0}:=1$, and $0^0:=1$.
For $\ell\in\{k,k+1\}$, recall the type space
\[
\cT_\ell
:=\bigl\{\xi\colon2^{[\ell+1]}\setminus\{\emptyset\}\to\{0,1\}\bigr\}.
\]
A type is an arbitrary zero--one labeling on the nonempty faces of $[\ell+1]$; no downward-closed condition is imposed.
In particular, $|\cT_k|=2^{2^{k+1}-1}$.
For $\xi\in\cT_\ell$, set
\[
\theta^{(\xi)}
:=\prod_{\emptyset\ne\rho\subset[\ell+1]}\theta_{\dim\rho}^{\,\xi(\rho)}(1-\theta_{\dim\rho})^{1-\xi(\rho)}.
\]
Thus, $\theta^{(\xi)}$ is the probability of the labeling $\xi$ when the labels of the faces are independent and each $j$-dimensional face is labeled $1$ with probability $\theta_j$.

For $\zeta\in\cT_{k+1}$ and $j\in[k+2]$, let $\zeta\setminus\{j\}\in\cT_k$ be obtained by restricting $\zeta$ to $2^{[k+2]\setminus\{j\}}\setminus\{\emptyset\}$ and then relabeling the remaining vertices in the order
\[
j+1,\ldots,k+2,1,\ldots,j-1
\]
as $1,\ldots,k+1$, with empty ranges omitted.
In particular, $\zeta\setminus\{k+2\}$ is the usual restriction of $\zeta$ to $[k+1]$.
This cyclic convention is chosen so that cyclic permutations of $[k+2]$ act compatibly on the restricted types; it is used in the proof of Lemma~\ref{lem:uGm-Phiu}.
With this notation, define
\[
\Phi_{k,c,\bm\theta}
=\bigl(\Phi_{k,c,\bm\theta}^{(\xi)}\bigr)_{\xi\in\cT_k}\colon[0,1]^{\cT_k}\longrightarrow[0,1]^{\cT_k}
\]
by
\begin{equation}\label{eq:def_Phi}
\Phi_{k,c,\bm\theta}^{(\xi)}(u)
:=\exp\left(-c\sum_{\substack{\zeta\in\cT_{k+1}\colon\\\zeta\setminus\{k+2\}=\xi}}\frac{\theta^{(\zeta)}}{\theta^{(\xi)}}\prod_{j=1}^{k+1}\bigl(1-u^{(\zeta\setminus\{j\})}\bigr)\right)\cdot\1_{\{\xi\neq1_k\}}
\end{equation}
for every $u=\bigl(u^{(\xi)}\bigr)_{\xi\in\cT_k}\in[0,1]^{\cT_k}$ and $\xi\in\cT_k$.
We also define
\[
\Psi_{k,c,\bm\theta}\colon[0,1]^{\cT_k}\longrightarrow\R
\]
by
\begin{equation}\label{eq:def_Psi}
\begin{aligned}
\Psi_{k,c,\bm\theta}(u)
:=\sum_{\xi\in\cT_k}\theta^{(\xi)}u^{(\xi)}
&+c\sum_{\zeta\in\cT_{k+1}}\theta^{(\zeta)}u^{(\zeta\setminus\{k+2\})}\prod_{j=1}^{k+1}\bigl(1-u^{(\zeta\setminus\{j\})}\bigr)\\
&\qquad-\frac c{k+2}\Biggl(1-\sum_{\zeta\in\cT_{k+1}}\theta^{(\zeta)}\prod_{j=1}^{k+2}\bigl(1-u^{(\zeta\setminus\{j\})}\bigr)\Biggr).
\end{aligned}
\end{equation}
The constant introduced in~\eqref{eq:hk_c_theta} is
\[
h_k(c,\bm\theta)
:=\max\Bigl\{\Psi_{k,c,\bm\theta}(u)\Bigm|u\in[0,1]^{\cT_k},\,u=\Phi_{k,c,\bm\theta}(u)\Bigr\}.
\]
The maximum is attained, since the fixed-point set of $\Phi_{k,c,\bm\theta}$ is nonempty by Brouwer's fixed-point theorem and compact, while $\Psi_{k,c,\bm\theta}$ is continuous.
Cases in which this optimization reduces to a scalar problem or to evaluation at the least fixed point are discussed in Section~S2 of the Supplementary Material~\cite{HKsupp}.
\begin{rmk}\label{rmk:last_theta}
  The map $\Phi_{k,c,\bm\theta}$ and the function $\Psi_{k,c,\bm\theta}$ are independent of $\theta_{k+1}$.
  Indeed, the restrictions $\zeta\setminus\{j\}$ do not involve the label of the $(k+1)$-simplex itself, and summing over its two possible labels eliminates $\theta_{k+1}$.
  Consequently, $h_k(c,\bm\theta)$ depends only on $c$ and $\theta_0,\theta_1,\ldots,\theta_k$.
  We retain $\theta_{k+1}$ to use a consistent notation for types and their probabilities in dimensions $k$ and $k+1$.
\end{rmk}

Before stating the main result, we verify that the rooted spectral measure appearing in it is well-defined.
\begin{lem}\label{lem:esa_PGW}
  Let $k\in\N_0$, $c>0$, and $\bm\theta=(\theta_j)_{j=0}^{k+1}\in[0,1]^{k+2}$.
  Then the $k$th compressed up Laplacian $\Dl_k^\up\bigl(\PGW_k(c)\setminus\PGW_k(c,\bm\theta)\bigr)$ is almost surely essentially self-adjoint.
\end{lem}
The proof follows closely the argument of~\cite[Claim~3.3]{LP2016a}, which treats the case without marking complexes~(see also~\cite{BLS2011} for the graph case).
We therefore defer it to Section~S1.2 of the Supplementary Material~\cite{HKsupp}.

The main result of this section is the following identification, which will be used in the next section.
\begin{prop}\label{prop:E_RSM0-h}
  Let $k\in\N_0$, $c>0$, and $\bm\theta=(\theta_j)_{j=0}^{k+1}\in[0,1]^{k+2}$.
  Then
  \[
  \E'\bigl[\mu_{(\PGW_k(c,\bm\theta\le\bm1),\sg_o)}(\{0\})\bigr]
  =h_k(c,\bm\theta).
  \]
\end{prop}

We next verify the extremal cases in~\eqref{eq:hk_extremal}.
\begin{rmk}
  We verify the two extremal cases in~\eqref{eq:hk_extremal}.
  For $\ell\in\{k,k+1\}$, let $0_\ell$ and $1_\ell$ denote the constant types taking the values $0$ and $1$, respectively.
  When $\bm\theta=\bm0$, we have $\theta^{(\xi)}=\1_{\{\xi=0_k\}}$ for $\xi\in\cT_k$ and $\theta^{(\zeta)}=\1_{\{\zeta=0_{k+1}\}}$ for $\zeta\in\cT_{k+1}$.
  Hence,
  \[
  \Phi_{k,c,\bm0}^{(\xi)}(u)=
  \begin{cases}
  \Phi_{k,c}(u^{(0_k)})
  &\text{if $\xi=0_k$,}\\
  0
  &\text{if $\xi=1_k$,}\\
  1
  &\text{otherwise,}
  \end{cases}
  \quad\text{and}\quad
  \Psi_{k,c,\bm0}(u)=\Psi_{k,c}(u^{(0_k)}).
  \]
  It follows that $h_k(c,\bm0)=h_k(c)$.
  When $\bm\theta=\bm1$, we similarly have
  \[
  \Psi_{k,c,\bm1}(u)=\Psi_{k,c}(u^{(1_k)}).
  \]
  Every fixed point of $\Phi_{k,c,\bm1}$ satisfies $u^{(1_k)}=0$.
  Since $\Psi_{k,c}(0)=0$, this gives $h_k(c,\bm1)=0$.
\end{rmk}

The proof of Proposition~\ref{prop:E_RSM0-h} occupies the remaining three subsections.
In Section~\ref{ssec:recursion}, we derive a deterministic recursion for the atomic mass at zero when the ambient complex is a simplicial tree.
In Section~\ref{ssec:DFPE}, we formulate the corresponding distributional fixed-point equation and identify its spectral solution as the largest fixed point.
Finally, in Section~\ref{ssec:reduction}, we relate this solution to the finite-dimensional maximization defining $h_k(c,\bm\theta)$ and complete the proof.

\subsection{A Recursion for the Atomic Mass at Zero}
\label{ssec:recursion}
The atomic mass at zero links the spectral analysis to persistent Betti numbers. On simplicial trees it satisfies an explicit recursion.
By adapting the resolvent method developed in~\cite{BL2010,BLS2011,LP2016a}, we obtain the following deterministic recursive formula.
\begin{lem}\label{lem:x_recurrence}
  Let $k\in\N_0$, and let $(K_o\subset T_o,\sg_o)$ be a $k$-rooted simplicial pair whose ambient complex $T_o$ is a simplicial tree.
  Let $\tau_1,\tau_2,\dots,\tau_m\in F_{k+1}(T_o)$ be the $(k+1)$-cofaces of $\sg_o$.
  For each $i=1,2,\ldots,m$, let $\sg_{i,1},\sg_{i,2},\dots,\sg_{i,k+1}$ denote the descendant $k$-faces of $\tau_i$, that is, the $k$-faces of $\tau_i$ distinct from $\sg_o$.
  Further, let $(K_{i,j}\subset T_{i,j},\sg_{i,j})$ be the $k$-rooted simplicial pair corresponding to the subtree of $T_o$ rooted at $\sg_{i,j}$.
  Assume that the compressed up Laplacians $\Dl_k^\up(T_o\setminus K_o)$ and $\Dl_k^\up(T_{i,j} \setminus K_{i,j})$ are all essentially self-adjoint, and define
  \[
  x_o:=\mu_{(K_o\subset T_o,\sg_o)}(\{0\})
  \quad\text{and}\quad
  x_{i,j}:=\mu_{(K_{i,j}\subset T_{i,j},\sg_{i,j})}(\{0\})
  \]
  for $i\in[m]$ and $j\in[k+1]$.
  Then
  \begin{equation}\label{eq:x_recurrence}
  x_o=\left(1+\sum_{i=1}^m\Biggl(\sum_{j=1}^{k+1}x_{i,j}\Biggr)^{-1}\right)^{-1}\cdot\1_{\{\sg_o\notin K_o\}}.
  \end{equation}
  Here, we adopt the conventions $1/0=\infty$ and $1/\infty=0$.
\end{lem}
The proof follows the approach in~\cite{BL2010,BLS2011} for graphs and in~\cite{LP2016a} for simplicial complexes.
Our setting, however, involves the additional structure of the marking complex $K_o$, leading to a different recursion.
We therefore include an adapted proof for completeness.
\begin{proof}
  If $\sg_o\in F_k(K_o)$, then~\eqref{eq:x_recurrence} holds trivially, since $\mu_{(K_o\subset T_o,\sg_o)}$ is the zero measure in this case.
  Hence we assume that $\sg_o\notin K_o$.
  Set
  \[
  T_f:=2^{\sg_o}\cup\bigcup_{i=1}^m2^{\tau_i},
  \qquad K_f:=T_f\cap K_o.
  \]
  We write
  \[
  L_o:=\overline{\Dl_k^\up(T_o\setminus K_o)},
  \quad L_f:=\Dl_k^\up(T_f\setminus K_f),
  \text{ and }
  L_{i,j}:=\overline{\Dl_k^\up(T_{i,j}\setminus K_{i,j})}
  \]
  for each $i\in[m]$ and $j\in[k+1]$.
  Note that all operators can be viewed as acting on $C^k_{(2)}(T_o)$ by extending them by zero on the orthogonal complements.
  Then it follows that
  \[
  L_o=L_f+\tilde L
  \quad\text{with }\tilde L:=\bigoplus_{i,j}L_{i,j}
  \]
  by the simplicial tree structure.
  For $s>0$, let $R_o$ and $\tilde R$ denote the resolvents of $L_o$ and $\tilde L$ at $-is\in\C$, respectively:
  \[
  R_o:=(L_o+isI)^{-1}
  \quad\text{and}\quad
  \tilde R:=(\tilde L+isI)^{-1}.
  \]
  For each $(i,j)$, also set
  \[
  R_{i,j}:=(L_{i,j}+isI)^{-1}.
  \]
  Then each descendant subspace is invariant under $\tilde R$, and the restriction of $\tilde R$ to the subspace corresponding to $T_{i,j}$ equals $R_{i,j}$.

  Below, we write $e_o:=e_{\sg_o}$, $e_{i,j}:=e_{\sg_{i,j}}$, and $\1_{i,j}:=\1_{\{\sg_{i,j}\in F_k(T_{i,j}\setminus K_{i,j})\}}$ for simplicity.
  By the second resolvent identity, we have
  \[
  R_oL_f\tilde R=\tilde R-R_o.
  \]
  We evaluate both sides of
  \[
  \la(R_oL_f\tilde R)e_o,e_o\ra_k
  =\la(\tilde R-R_o)e_o,e_o\ra_k.
  \]
  Since $\tilde Re_o=e_o/(is)$, the explicit expression~\eqref{eq:L_expression}, together with Lemma~\ref{lem:comp_L}(1), implies that~(after possibly flipping the sign of each $e_{i,j}$)
  \[
  \la(R_oL_f\tilde R)e_o,e_o\ra_k
  =\frac1{is}\la R_oL_f\,e_o,e_o\ra_k
  =\frac1{is}\Biggl(m\la R_o e_o,e_o\ra_k+\sum_{i=1}^m\sum_{j=1}^{k+1}\la R_o e_{i,j},e_o\ra_k\,\1_{i,j}\Biggr).
  \]
  On the other hand,
  \[
  \la(\tilde R-R_o)e_o,e_o\ra_k
  =\frac1{is}-\la R_o e_o,e_o\ra_k.
  \]
  Combining the above equalities, we obtain
  \begin{equation}\label{eq:e_o-e_o}
  is\la R_o e_o,e_o\ra_k+\sum_{i=1}^m\Biggl(\la R_o e_o,e_o\ra_k+\sum_{j=1}^{k+1}\la R_o e_{i,j},e_o\ra_k\,\1_{i,j}\Biggr)
  =1.
  \end{equation}
  We next evaluate both sides of
  \[
  \la(R_oL_f\tilde R)e_{i,j},e_o\ra_k
  =\la(\tilde R-R_o)e_{i,j},e_o\ra_k
  \]
  for a fixed $(i,j)$.
  Since $\la\tilde Re_{i,j},e_o\ra_k=0$ and $\la\tilde Re_{i,j},e_{i',j'}\ra_k=0$ for $(i',j')\neq(i,j)$ by the definition of $\tilde R$, the explicit expression~\eqref{eq:L_expression}, together with Lemma~\ref{lem:comp_L}(1), again yields
  \begin{align*}
  \la(R_oL_f\tilde R)e_{i,j},e_o\ra_k
  &=\la\tilde R e_{i,j},e_{i,j}\ra_k\cdot
  \la R_oL_f\,e_{i,j},e_o\ra_k\\
  &=\la\tilde R e_{i,j},e_{i,j}\ra_k\cdot
  \Biggl(\la R_o e_o,e_o\ra_k+\sum_{j'=1}^{k+1}\la R_o e_{i,j'},e_o\ra_k\,\1_{i,j'}\Biggr)\,\1_{i,j}.
  \end{align*}
  On the other hand,
  \[
  \la(\tilde R-R_o)e_{i,j},e_o\ra_k
  =-\la R_o e_{i,j},e_o\ra_k
  =-\la R_o e_{i,j},e_o\ra_k\,\1_{i,j}.
  \]
  The last equality follows from the fact that if $\sg_{i,j}\in K_{i,j}$, then $R_o e_{i,j}=e_{i,j}/(is)$.
  Combining the above equalities and taking the sum over $1\le j\le k+1$, we have
  \[
  \Biggl(\sum_{j=1}^{k+1}\la\tilde R e_{i,j},e_{i,j}\ra_k\,\1_{i,j}\Biggr)
  \Biggl(\la R_o e_o,e_o\ra_k+\sum_{j'=1}^{k+1}\la R_o e_{i,j'},e_o\ra_k\,\1_{i,j'}\Biggr)
  =-\sum_{j=1}^{k+1}\la R_o e_{i,j},e_o\ra_k\,\1_{i,j},
  \]
  which implies that
  \begin{equation}\label{eq:e_ij-e_o}
  \sum_{j=1}^{k+1}\la R_o e_{i,j},e_o\ra_k\,\1_{i,j}
  =-\la R_o e_o,e_o\ra_k\cdot
  \frac{\sum_{j=1}^{k+1}\la\tilde R e_{i,j},e_{i,j}\ra_k\,\1_{i,j}}{1+\sum_{j=1}^{k+1}\la\tilde R e_{i,j},e_{i,j}\ra_k\,\1_{i,j}}.
  \end{equation}
  As will be seen below, the denominator in~\eqref{eq:e_ij-e_o} is nonzero.
  Substituting~\eqref{eq:e_ij-e_o} into~\eqref{eq:e_o-e_o}, we obtain
  \begin{equation}\label{eq:R_recurrence}
  \la R_o e_o,e_o\ra_k\cdot
  \Biggl(is+\sum_{i=1}^m
  \frac1{1+\sum_{j=1}^{k+1}\la\tilde R e_{i,j},e_{i,j}\ra_k\,\1_{i,j}}\Biggr)
  =1.
  \end{equation}
  
  Finally, we derive the recursive formula~\eqref{eq:x_recurrence} from the resolvent recursion~\eqref{eq:R_recurrence} by letting $s\downarrow0$ appropriately.
  Define
  \[
  h_{i,j}(s)
  :=is\la\tilde R e_{i,j},e_{i,j}\ra_k\,\1_{i,j}
  =is\la R_{i,j}e_{i,j},e_{i,j}\ra_k\,\1_{i,j}.
  \]
  Using~\eqref{eq:RSM}, which also holds at $z=-is$ by the spectral theorem, we have
  \begin{align*}
  h_{i,j}(s)
  &=is\int_\R\frac1{\lm+is}\,d\mu_{(K_{i,j}\subset T_{i,j},\sg_{i,j})}(\lm)\\
  &=\int_\R\frac{s^2}{\lm^2+s^2}\,d\mu_{(K_{i,j}\subset T_{i,j},\sg_{i,j})}(\lm)
  +i\int_\R\frac{\lm s}{\lm^2+s^2}\,d\mu_{(K_{i,j}\subset T_{i,j},\sg_{i,j})}(\lm)\\
  &\xrightarrow[s\downarrow0]{}\mu_{(K_{i,j}\subset T_{i,j},\sg_{i,j})}(\{0\})=x_{i,j},
  \end{align*}
  where the convergence follows from the dominated convergence theorem.
  Similarly, we write $h_o(s):=is\la R_o e_o,e_o\ra_k$.
  (Recall that we assume that $\sg_o\notin K_o$.)
  The same argument shows that
  \[
  h_o(s)\xrightarrow[s\downarrow0]{}\,\mu_{(K_o\subset T_o,\sg_o)}(\{0\})=x_o.
  \]  
  Therefore, by rewriting~\eqref{eq:R_recurrence}, we obtain
  \[
  h_o(s)\cdot
  \Biggl(1+\sum_{i=1}^m\frac1{is+\sum_{j=1}^{k+1}h_{i,j}(s)}\Biggr)
  =1.
  \]
  The preceding representation and the nonnegativity of $L_{i,j}$ show that both the real and imaginary parts of $h_{i,j}(s)$ are nonnegative.
  Hence, each denominator
  \[
  is+\sum_{j=1}^{k+1}h_{i,j}(s)
  =is\left(1+\sum_{j=1}^{k+1}\la\tilde R e_{i,j},e_{i,j}\ra_k\,\1_{i,j}\right)
  \]
  is nonzero, which also verifies that the denominator in~\eqref{eq:e_ij-e_o} is nonzero.

  If $\sum_jx_{i,j}=0$ for some $i$, then $is+\sum_jh_{i,j}(s)$ tends to zero through the first quadrant.
  Its reciprocal therefore tends to infinity while remaining in the fourth quadrant.
  Since all the other reciprocal terms also lie in the fourth quadrant, no cancellation is possible, and the preceding identity gives $h_o(s)\to0$.
  Otherwise, we may pass directly to the limit in that identity.
  In either case, letting $s\downarrow0$ yields~\eqref{eq:x_recurrence}, with the conventions $1/0=\infty$ and $1/\infty=0$.
\end{proof}

\subsection{Distributional Fixed-Point Equation}
\label{ssec:DFPE}
We apply the recursion in Lemma~\ref{lem:x_recurrence} to the random simplicial pair $\PGW_k(c,\bm\theta\le\bm1)$ to obtain a distributional fixed-point equation.
We then show that the conditional laws of the atomic mass at zero, indexed by the root type, form its largest fixed point.

Fix $k\in\N_0$, $c>0$, and $\bm\theta=(\theta_j)_{j=0}^{k+1}\in[0,1]^{k+2}$.
We identify the $k$-root $\sg_o$ with $[k+1]$.
For each $\xi\in\cT_k$, define the event
\[
A^{(\xi)}
:=\{\eta_{\bm\theta}(\rho)=\xi(\rho)\text{ for every nonempty face $\rho\subset[k+1]$}\}.
\]
Here, $\{\eta_{\bm\theta}(\rho)\}_\rho$ is the family of Bernoulli random variables introduced in Definition~\ref{df:PGW}.
Then $\P'(A^{(\xi)})=\theta^{(\xi)}$.
Fix $\xi\in\cT_k$.
On the event $A^{(\xi)}$, each $(k+1)$-simplex incident to the root corresponds to a type $\zeta\in\cT_{k+1}$ satisfying $\zeta\setminus\{k+2\}=\xi$.
There are
\[
H:=2^{2^{k+1}}
\]
such extensions, since an extension is determined by its values on the $2^{k+1}$ faces containing the vertex $k+2$, each of which may be labeled either $0$ or $1$.
We enumerate these extensions as
\[
\zeta_1^{(\xi)},\zeta_2^{(\xi)},\ldots,\zeta_H^{(\xi)},
\]
and set
\[
c_h^{(\xi)}
:=c\,\frac{\theta^{(\zeta_h^{(\xi)})}}{\theta^{(\xi)}},
\qquad h=1,2,\ldots,H.
\]
Here, we again use the convention $0/0=0$ from Section~\ref{ssec:hk_main}.
Whenever $\theta^{(\xi)}>0$, conditional on $A^{(\xi)}$, the numbers of incident $(k+1)$-simplices of the respective types are independent Poisson random variables with parameters $c_h^{(\xi)}$.
Each extension $\zeta_h^{(\xi)}$ gives rise to $k+1$ descendant $k$-faces of types
\[
\zeta_h^{(\xi)}\setminus\{j\}\in\cT_k,
\qquad j=1,2,\ldots,k+1.
\]

We equip $\cP([0,1])^{\cT_k}$ with the product topology induced by the weak topology on $\cP([0,1])$.
For $\mu,\mu'\in\cP([0,1])$, we say that $\mu'$ \emph{stochastically dominates} $\mu$, and write $\mu\le_{\st}\mu'$, if
\[
\mu((x,1])\le\mu'((x,1])
\qquad\text{for every $x\in[0,1]$.}
\]
We further equip $\cP([0,1])^{\cT_k}$ with the coordinatewise stochastic order, defined by
\[
(\cD^{(\xi)})_{\xi\in\cT_k}
\le_{\st}
(\cD'^{(\xi)})_{\xi\in\cT_k}
\]
if and only if $\cD^{(\xi)}\le_{\st}\cD'^{(\xi)}$ for every $\xi\in\cT_k$.
Define
\[
\Gm_{k,c,\bm\theta}
=\bigl(\Gm_{k,c,\bm\theta}^{(\xi)}\bigr)_{\xi\in\cT_k}\colon\cP([0,1])^{\cT_k}\longrightarrow\cP([0,1])^{\cT_k}
\]
by
\[
\Gm_{k,c,\bm\theta}^{(\xi)}(\cD)
:=\Law\left(\Biggl(1+\sum_{h=1}^H\sum_{i=1}^{M_h^{(\xi)}}\biggl(\sum_{j=1}^{k+1}X_{h,i,j}^{(\xi)}\biggr)^{-1}\Biggr)^{-1}\!\cdot\1_{\{\xi\neq1_k\}}\right)
\]
for every $\cD=\bigl(\cD^{(\xi)}\bigr)_{\xi\in\cT_k}\in\cP([0,1])^{\cT_k}$, where
\[
M_h^{(\xi)}\sim\Poisson(c_h^{(\xi)})
\quad\text{and}\quad
X_{h,i,j}^{(\xi)}\sim\cD^{(\zeta_h^{(\xi)}\setminus\{j\})}
\]
are independent random variables.
The map $\Gm_{k,c,\bm\theta}$ is order-preserving.
Moreover, $\Gm_{k,c,\bm\theta}$ is continuous from above: if $\cD_l\downarrow\cD$, then
\[
\Gm_{k,c,\bm\theta}(\cD_l)\downarrow \Gm_{k,c,\bm\theta}(\cD)
\]
in the product weak topology.
Both properties follow from the standard monotone coupling obtained from quantile functions.

Since $(\dl_1)_{\xi\in\cT_k}$ is the largest element in the coordinatewise
stochastic order, the iterates form a stochastically decreasing sequence.
Compactness of $\cP([0,1])^{\cT_k}$ gives subsequential limits, and stochastic monotonicity implies that all such limits coincide.
Hence the iterates converge in the product weak topology.
Set
\[
\cD_{\max}
:=\lim_{l\to\infty}\Gm_{k,c,\bm\theta}^l\bigl((\dl_1)_{\xi\in\cT_k}\bigr).
\]
By continuity from above, $\cD_{\max}$ is a fixed point of $\Gm_{k,c,\bm\theta}$.
It is the largest fixed point, since any fixed point $\cD$ satisfies
\[
\cD
=\Gm_{k,c,\bm\theta}^l(\cD)
\le_{\st}\Gm_{k,c,\bm\theta}^l\bigl((\dl_1)_{\xi\in\cT_k}\bigr)
\]
for every $l\in\N$.

For each $\xi\in\cT_k$, let $\cD_o^{(\xi)}$ denote the conditional law of
\[
X_o:=\mu_{(\PGW_k(c,\bm\theta\le\bm1),\sg_o)}(\{0\})
\]
given $A^{(\xi)}$, and write $\cD_o=(\cD_o^{(\xi)})_{\xi\in\cT_k}$.
If $\theta^{(\xi)}=0$, the event $A^{(\xi)}$ is null and the conditional law is not intrinsically defined.
In this case, we set
\begin{equation}\label{eq:Do_convention}
\cD_o^{(\xi)}
:=
\begin{cases}
  \dl_0
    &\text{if $\xi=1_k$,}\\
  \dl_1
    &\text{if $\xi\neq1_k$.}
\end{cases}
\end{equation}
Then Lemma~\ref{lem:x_recurrence} together with the recursive structure of $\PGW_k(c,\bm\theta\le\bm1)$ implies that $\cD_o$ satisfies the \emph{distributional fixed-point equation}
\begin{equation}\label{eq:DFPE}
\cD_o=\Gm_{k,c,\bm\theta}(\cD_o).
\end{equation}
Note that the convention~\eqref{eq:Do_convention} is consistent with~\eqref{eq:DFPE}, since $\theta^{(\xi)}=0$ implies that $c_h^{(\xi)}=0$ for every $h$, and hence the $\xi$-coordinate of $\Gm_{k,c,\bm\theta}(\cD_o)$ is $\dl_0$ if $\xi=1_k$ and $\dl_1$ otherwise.

\begin{lem}\label{lem:Do-Dmax}
  It holds that
  \[
  \cD_o=\cD_{\max}.
  \]
  In particular, $\cD_o$ is the largest fixed point of $\Gm_{k,c,\bm\theta}$.
\end{lem}
\begin{proof}
  Set
  \[
  T:=\PGW_k(c)
  \quad\text{and}\quad
  K:=\PGW_k(c,\bm\theta).
  \]
  By Lemma~\ref{lem:esa_PGW} and countability, the compressed up Laplacians associated with $(K\subset T,\sg_o)$ and all its descendant subtrees are almost surely essentially self-adjoint.
  
  For $l\in\N$, let
  \[
  T^{(l)}
  :=\{\sg\in T\mid\dist_{\sg_o}(\sg)\le l\}
  \]
  denote the finite truncation of $T$ at distance $l$, and define its set of boundary $k$-simplices by
  \[
  \del_kT^{(l)}
  :=\{\sg\in F_k\bigl(T^{(l)}\bigr)\mid\dist_{\sg_o}(\sg)=l\}.
  \]
  Define
  \[
  K^{(l)}
  :=\bigl\{\tau\in K\cap T^{(l)}\bigm|\tau\not\supset\sg\text{ for every $\sg\in\del_kT^{(l)}$}
  \bigr\}.
  \]
  Thus, $K^{(l)}$ is obtained from $K\cap T^{(l)}$ by removing every boundary $k$-simplex at distance $l$ and all its $(k+1)$-cofaces.
  Consequently, no boundary $k$-simplex belongs to $K^{(l)}$, and its atomic mass at zero is $1$ when it is regarded as the root of its descendant subtree in $T^{(l)}$.
  Define
  \[
  X_o^{(l)}
  :=\mu_{(K^{(l)}\subset T^{(l)},\sg_o)}(\{0\}).
  \]
  Applying~\eqref{eq:x_recurrence} from the boundary toward the root shows that $X_o^{(l)}$ is obtained by imposing the boundary value $1$ on every $k$-simplex in $\del_kT^{(l)}$.
  The corresponding atomic masses in the original descendant subtrees belong to $[0,1]$, and the right-hand side of~\eqref{eq:x_recurrence} is nondecreasing in each descendant atomic mass.
  Replacing these boundary values by $1$ and iterating the recursion toward the root therefore yields
  \begin{equation}\label{eq:Xo_upper}
  X_o\le X_o^{(l)}
  \end{equation}
  for every $l\in\N$.

  On the other hand, the rooted finite truncations $[K^{(l)}\subset T^{(l)},\sg_o]$ converge almost surely to $[K\subset T,\sg_o]$ as $l\to\infty$.
  The upper semicontinuity of $\RSM_0$ therefore implies that
  \begin{equation}\label{eq:Xo_lower}
  \limsup_{l\to\infty}X_o^{(l)}
  \le X_o.
  \end{equation}
  Combining~\eqref{eq:Xo_upper} and~\eqref{eq:Xo_lower}, we obtain
  \[
  X_o^{(l)}\xrightarrow{\mathrm{a.s.}}X_o
  \quad\text{as $l\to\infty$.}
  \]

  Let $\xi\in\cT_k$ satisfy $\theta^{(\xi)}>0$.
  Conditional on $A^{(\xi)}$, the law of $X_o^{(l)}$ is
  \[
  \Bigl(
  \Gm_{k,c,\bm\theta}^l
  \bigl((\dl_1)_{\zeta\in\cT_k}\bigr)
  \Bigr)^{(\xi)}.
  \]
  The conditional laws converge weakly by the almost sure convergence above.
  By the definition of $\cD_{\max}$, we obtain
  \[
  \cD_o^{(\xi)}
  =\lim_{l\to\infty}\Bigl(\Gm_{k,c,\bm\theta}^l\bigl((\dl_1)_{\zeta\in\cT_k}\bigr)\Bigr)^{(\xi)}
  =\cD_{\max}^{(\xi)}.
  \]
  If $\theta^{(\xi)}=0$, then $c_h^{(\xi)}=0$ for every $h$, and the conventions defining $\cD_o^{(\xi)}$ and $\Gm_{k,c,\bm\theta}$ directly give
  \[
  \cD_o^{(\xi)}=\cD_{\max}^{(\xi)}.
  \]
  Hence, $\cD_o=\cD_{\max}$.
\end{proof}

\subsection{Reduction to the Finite-Dimensional Fixed-Point Problem}
\label{ssec:reduction}
By Lemma~\ref{lem:Do-Dmax}, the conditional laws of the atomic mass at zero form the largest fixed point $\cD_{\max}$ of $\Gm_{k,c,\bm\theta}$.
To prove Proposition~\ref{prop:E_RSM0-h}, it remains to show that the corresponding expected atomic mass equals the maximum defining $h_k(c,\bm\theta)$.
We do this by relating the distributional fixed-point equation to the map $\Phi_{k,c,\bm\theta}$ and the function $\Psi_{k,c,\bm\theta}$.
More precisely, we pass from the conditional laws $\cD_o^{(\xi)}$ of $X_o$ to their positivity probabilities $\P(X_o>0\mid A^{(\xi)})=\cD_o^{(\xi)}((0,1])$, which form a finite-dimensional vector.

For $\cD=\bigl(\cD^{(\xi)}\bigr)_{\xi\in\cT_k}\in\cP([0,1])^{\cT_k}$, define
\begin{align*}
u(\cD)
&:=\bigl(u^{(\xi)}(\cD)\bigr)_{\xi\in\cT_k}\in[0,1]^{\cT_k},
\qquad u^{(\xi)}(\cD):=\cD^{(\xi)}((0,1]),
\shortintertext{and}
E_{k,\bm\theta}(\cD)
&:=\sum_{\xi\in\cT_k}\theta^{(\xi)}\int_{[0,1]}x\,d\cD^{(\xi)}(x).
\end{align*}
The following lemma relates these quantities to $\Phi_{k,c,\bm\theta}$ and $\Psi_{k,c,\bm\theta}$.

\begin{lem}\label{lem:uGm-Phiu}
  For every $\cD\in\cP([0,1])^{\cT_k}$,
  \begin{equation}\label{eq:uTheta-Phi}
  u(\Gm_{k,c,\bm\theta}(\cD))
  =\Phi_{k,c,\bm\theta}(u(\cD)).
  \end{equation}
  Moreover, if
  \[
  u(\cD)=\Phi_{k,c,\bm\theta}(u(\cD))
  \quad\text{and}\quad
  \Gm_{k,c,\bm\theta}(\cD)\le_{\st}\cD,
  \]
  then
  \begin{equation}\label{eq:DFPE-comparison}
  E_{k,\bm\theta}(\cD)
  \ge\Psi_{k,c,\bm\theta}(u(\cD)).
  \end{equation}
  If $\Gm_{k,c,\bm\theta}(\cD)=\cD$, then equality holds in~\eqref{eq:DFPE-comparison}.
\end{lem}
\begin{proof}
  We first prove~\eqref{eq:uTheta-Phi}, then derive the expectation identity~\eqref{eq:X_o-xi}, and finally use cyclic symmetrization to obtain~\eqref{eq:DFPE-comparison}.
  Let $\cD=\bigl(\cD^{(\xi)}\bigr)_{\xi\in\cT_k}\in\cP([0,1])^{\cT_k}$, and let
  \[
  M_h^{(\xi)}\sim\Poisson(c_h^{(\xi)})
  \quad\text{and}\quad
  X_{h,i,j}^{(\xi)}\sim\cD^{(\zeta_h^{(\xi)}\setminus\{j\})}
  \]
  be independent random variables for $h\in[H]$, $i\in\N$, and $j\in[k+1]$.
  We also set
  \[
  S_{h,i}^{(\xi)}
  :=\sum_{j=1}^{k+1}X_{h,i,j}^{(\xi)},
  \qquad h\in[H],\quad i\in\N.
  \]
  Then $\Gm_{k,c,\bm\theta}^{(\xi)}(\cD)$ is the law of
  \[
  X^{(\xi)}:=\Biggl(1+\sum_{h=1}^H\sum_{i=1}^{M_h^{(\xi)}}\Bigl(S_{h,i}^{(\xi)}\Bigr)^{-1}\Biggr)^{-1}\cdot\1_{\bigl\{S_{h,i}^{(\xi)}>0,\, \forall h\in[H],\,\forall i\in[M_h^{(\xi)}]\bigr\}}\cdot\1_{\{\xi\neq1_k\}}.
  \]
  Therefore,
  \begin{align*}
  u^{(\xi)}(\Gm_{k,c,\bm\theta}(\cD))
  &=\P\Bigl(S_{h,i}^{(\xi)}>0,\,\forall h\in[H],\,\forall i\in\bigl[M_h^{(\xi)}\bigr]\Bigr)\cdot\1_{\{\xi\neq1_k\}}\\
  &=\prod_{h=1}^H\P\Bigl(S_{h,i}^{(\xi)}>0,\,\forall i\in\bigl[M_h^{(\xi)}\bigr]\Bigr)\cdot\1_{\{\xi\neq1_k\}}\\
  &=\prod_{h=1}^H\sum_{m=0}^\infty\frac{\bigl(c_h^{(\xi)}\bigr)^m}{m!}\exp\bigl(-c_h^{(\xi)}\bigr)\Biggl\{1-\prod_{j=1}^{k+1}\Bigl(1-u^{(\zeta_h^{(\xi)}\setminus\{j\})}(\cD)\Bigr)\Biggr\}^m\cdot\1_{\{\xi\neq1_k\}}\\
  &=\prod_{h=1}^H\exp\Biggl(-c_h^{(\xi)}\prod_{j=1}^{k+1}\Bigl(1-u^{(\zeta_h^{(\xi)}\setminus\{j\})}(\cD)\Bigr)\Biggr)\cdot\1_{\{\xi\neq1_k\}}\\
  &=\exp\Biggl(-\sum_{h=1}^H c_h^{(\xi)}\prod_{j=1}^{k+1}\Bigl(1-u^{(\zeta_h^{(\xi)}\setminus\{j\})}(\cD)\Bigr)\Biggr)\cdot\1_{\{\xi\neq1_k\}}\\
  &=\Phi_{k,c,\bm\theta}^{(\xi)}(u(\cD)),
  \end{align*}
  which proves~\eqref{eq:uTheta-Phi}.

  We now assume the hypotheses of the second assertion.
  By~\eqref{eq:uTheta-Phi} and the assumption $u(\cD)=\Phi_{k,c,\bm\theta}(u(\cD))$, we have
  \[
  \P\bigl(X^{(\xi)}>0\bigr)
  =u^{(\xi)}(\Gm_{k,c,\bm\theta}(\cD))
  =\Phi_{k,c,\bm\theta}^{(\xi)}(u(\cD))
  =u^{(\xi)}(\cD)
  \]
  for every $\xi\in\cT_k$.
  For each $h\in[H]$, let $S_h^{(\xi)}$, independent of $X^{(\xi)}$, have the same law as $S_{h,1}^{(\xi)}$.
  Then, for $\xi\neq1_k$,
  \begin{equation}\label{eq:X_o-xi}
  \begin{aligned}
  &\E[X^{(\xi)}]\\
  &=\E\left[\left(1-\frac{\sum_{h=1}^H\sum_{i=1}^{M_h^{(\xi)}}\Bigl(S_{h,i}^{(\xi)}\Bigr)^{-1}}{1+\sum_{h=1}^H\sum_{i=1}^{M_h^{(\xi)}}\Bigl(S_{h,i}^{(\xi)}\Bigr)^{-1}}\right)\cdot\1_{\bigl\{S_{h,i}^{(\xi)}>0,\,\forall h\in[H],\,\forall i\in[M_h^{(\xi)}]\bigr\}}\right]\\
  &=\P\bigl(X^{(\xi)}>0\bigr)\\*
  &\;-\sum_{h=1}^H\E\left[\sum_{i=1}^{M_h^{(\xi)}}\frac{\Bigl(S_{h,i}^{(\xi)}\Bigr)^{-1}\cdot\1_{\{S_{h,i}^{(\xi)}>0\}}\cdot\1_{\{S_{h,i'}^{(\xi)}>0,\,\forall i'\in[M_h^{(\xi)}]\setminus\{i\}\}}\cdot\1_{\{S_{h',i'}^{(\xi)}>0,\,\forall h'\neq h,\,\forall i'\in[M_{h'}^{(\xi)}]\}}}{1+\Bigl(S_{h,i}^{(\xi)}\Bigr)^{-1}+\sum_{i'\in[M_h^{(\xi)}]\setminus\{i\}}\Bigl(S_{h,i'}^{(\xi)}\Bigr)^{-1}+\sum_{h'\neq h}\sum_{i'=1}^{M_{h'}^{(\xi)}}\Bigl(S_{h',i'}^{(\xi)}\Bigr)^{-1}}\right]\\
  &=u^{(\xi)}(\cD)-\sum_{h=1}^H\sum_{m=0}^\infty\frac{\bigl(c_h^{(\xi)}\bigr)^m}{m!}\exp\bigl(-c_h^{(\xi)}\bigr)\\*
  &\qquad\qquad\cdot m\,\E\left[\frac{\Bigl(S_h^{(\xi)}\Bigr)^{-1}\cdot\1_{\{S_h^{(\xi)}>0\}}\cdot\1_{\{S_{h,i'}^{(\xi)}>0,\,i'\in[m-1]\}}\cdot\1_{\{S_{h',i'}^{(\xi)}>0,\,\forall h'\neq h,\,\forall i'\in[M_{h'}^{(\xi)}]\}}}{1+\Bigl(S_h^{(\xi)}\Bigr)^{-1}+\sum_{i'=1}^{m-1}\Bigl(S_{h,i'}^{(\xi)}\Bigr)^{-1}+\sum_{h'\neq h}\sum_{i'=1}^{M_{h'}^{(\xi)}}\Bigl(S_{h',i'}^{(\xi)}\Bigr)^{-1}}\right]\\
  &=u^{(\xi)}(\cD)-\sum_{h=1}^H c_h^{(\xi)}\E\left[\frac{\Bigl(S_h^{(\xi)}\Bigr)^{-1}\cdot\1_{\{S_h^{(\xi)}>0\}}\cdot\1_{\bigl\{S_{h',i'}^{(\xi)}>0,\,\forall h'\in[H],\,\forall i'\in[M_{h'}^{(\xi)}]\bigr\}}}{1+\Bigl(S_h^{(\xi)}\Bigr)^{-1}+\sum_{h'=1}^H\sum_{i'=1}^{M_{h'}^{(\xi)}}\Bigl(S_{h',i'}^{(\xi)}\Bigr)^{-1}}\right]\\
  &=u^{(\xi)}(\cD)-\sum_{h=1}^H c_h^{(\xi)}\,\E\left[\frac{X^{(\xi)}}{S_h^{(\xi)}+X^{(\xi)}}\cdot\1_{\{S_h^{(\xi)}>0,\,X^{(\xi)}>0\}}\right].
  \end{aligned}
  \end{equation}
  The last identity follows by multiplying both the numerator and the denominator by
  \[
  S_h^{(\xi)}\Biggm/\Biggl(1+\sum_{h'=1}^H\sum_{i'=1}^{M_{h'}^{(\xi)}}\Bigl(S_{h',i'}^{(\xi)}\Bigr)^{-1}\Biggr).
  \]
  Note that~\eqref{eq:X_o-xi} also holds for $\xi=1_k$, since $u^{(1_k)}=0$ and $X^{(1_k)}=0$ almost surely.

  For each $\zeta\in\cT_{k+1}$, let
  \[
  Y_j^{(\zeta)}
  \sim\cD^{(\zeta\setminus\{j\})},
  \qquad j=1,2,\ldots,k+1,
  \]
  be independent random variables.
  Independently of these variables, choose a pair
  \[
  \bigl(\tilde Y_{k+2}^{(\zeta)},Y_{k+2}^{(\zeta)}\bigr)
  \]
  such that
  \[
  \tilde Y_{k+2}^{(\zeta)}\sim\Gm_{k,c,\bm\theta}^{(\zeta\setminus\{k+2\})}(\cD),
  \quad Y_{k+2}^{(\zeta)}\sim\cD^{(\zeta\setminus\{k+2\})},
  \quad\text{and}\quad
  \tilde Y_{k+2}^{(\zeta)}\le Y_{k+2}^{(\zeta)}
  \quad\text{a.s.}
  \]
  Such a pair exists by a standard coupling argument using quantile functions under the stochastic domination $\Gm_{k,c,\bm\theta}(\cD)\le_{\st}\cD$.
  Since $\theta^{(\xi)}c_h^{(\xi)}=c\,\theta^{(\zeta_h^{(\xi)})}$,
  \begin{equation}\label{eq:EGm-YtildeY}
  \begin{aligned}
  E_{k,\bm\theta}(\cD)
  &\ge E_{k,\bm\theta}(\Gm_{k,c,\bm\theta}(\cD))\\
  &=\sum_{\xi\in\cT_k}\theta^{(\xi)}\E\bigl[X^{(\xi)}\bigr]\\
  &=\sum_{\xi\in\cT_k}\theta^{(\xi)}u^{(\xi)}(\cD)
  -c\sum_{\xi\in\cT_k}\sum_{h=1}^H\theta^{(\zeta_h^{(\xi)})}\,\E\left[\frac{X^{(\xi)}}{S_h^{(\xi)}+X^{(\xi)}}\cdot\1_{\{S_h^{(\xi)}>0,\,X^{(\xi)}>0\}}\right]\\
  &=\sum_{\xi\in\cT_k}\theta^{(\xi)}u^{(\xi)}(\cD)
  -c\sum_{\zeta\in\cT_{k+1}}\theta^{(\zeta)}\,\E\left[\frac{\tilde Y_{k+2}^{(\zeta)}}{\sum_{j=1}^{k+1}Y_j^{(\zeta)}+\tilde Y_{k+2}^{(\zeta)}}\cdot\1_{\bigl\{\sum_{j=1}^{k+1}Y_j^{(\zeta)}>0,\,\tilde Y_{k+2}^{(\zeta)}>0\bigr\}}\right]\\
  &\ge\sum_{\xi\in\cT_k}\theta^{(\xi)}u^{(\xi)}(\cD)
  -c\sum_{\zeta\in\cT_{k+1}}\theta^{(\zeta)}\,\E\left[\frac{Y_{k+2}^{(\zeta)}}{\sum_{j=1}^{k+2}Y_j^{(\zeta)}}\cdot\1_{\bigl\{\sum_{j=1}^{k+1}Y_j^{(\zeta)}>0,\,Y_{k+2}^{(\zeta)}>0\bigr\}}\right].
  \end{aligned}
  \end{equation}
  For the last inequality in~\eqref{eq:EGm-YtildeY}, we used $\widetilde Y_{k+2}^{(\zeta)}\le Y_{k+2}^{(\zeta)}$.
  Set
  \[
  I(\cD)
  :=\sum_{\zeta\in\cT_{k+1}}\theta^{(\zeta)}\E\left[\frac{Y_{k+2}^{(\zeta)}}{\sum_{j=1}^{k+2}Y_j^{(\zeta)}}\cdot\1_{\bigl\{\sum_{j=1}^{k+1}Y_j^{(\zeta)}>0,\,Y_{k+2}^{(\zeta)}>0\bigr\}}\right].
  \]
  By the cyclic definition of $\zeta\setminus\{j\}$, cyclic relabeling of $[k+2]$ cyclically permutes the variables $Y_1^{(\zeta)},\ldots,Y_{k+2}^{(\zeta)}$ in the sum over $\zeta$, while preserving $\theta^{(\zeta)}$.
  Hence, every distinguished index gives the same contribution, and averaging over $j$ yields
  \begin{align*}
  I(\cD)
  &=\frac1{k+2}\sum_{j=1}^{k+2}\sum_{\zeta\in\cT_{k+1}}\theta^{(\zeta)}\E\left[\frac{Y_j^{(\zeta)}}{\sum_{j'=1}^{k+2}Y_{j'}^{(\zeta)}}\cdot\1_{\bigl\{Y_j^{(\zeta)}>0,\,\sum_{j'\neq j}Y_{j'}^{(\zeta)}>0\bigr\}}\right]\\
  &=\frac1{k+2}\sum_{\zeta\in\cT_{k+1}}\theta^{(\zeta)}\P\Bigl(\bigl|\bigl\{j\in[k+2]\bigm|Y_j^{(\zeta)}>0\bigr\}\bigr|\ge2\Bigr).
  \end{align*}
  Indeed,
  \[
  \sum_{j=1}^{k+2}\frac{Y_j^{(\zeta)}}{\sum_{j'=1}^{k+2}Y_{j'}^{(\zeta)}}\cdot\1_{\bigl\{Y_j^{(\zeta)}>0,\,\sum_{j'\neq j}Y_{j'}^{(\zeta)}>0\bigr\}}
  \]
  equals $1$ if at least two of the variables are positive and equals $0$ otherwise.
  By independence, $I(\cD)$ equals
  \begin{align*}
  &\frac1{k+2}\sum_{\zeta\in\cT_{k+1}}\theta^{(\zeta)}\Biggl(1-\prod_{j=1}^{k+2}\Bigl(1-u^{(\zeta\setminus\{j\})}(\cD)\Bigr)-\sum_{j=1}^{k+2}u^{(\zeta\setminus\{j\})}(\cD)\prod_{j'\neq j}\Bigl(1-u^{(\zeta\setminus\{j'\})}(\cD)\Bigr)\Biggr)\\
  &=\frac1{k+2}\Biggl(1-\sum_{\zeta\in\cT_{k+1}}\theta^{(\zeta)}\prod_{j=1}^{k+2}\Bigl(1-u^{(\zeta\setminus\{j\})}(\cD)\Bigr)\Biggr)\\*
  &\qquad\qquad-\sum_{\zeta\in\cT_{k+1}}\theta^{(\zeta)}u^{(\zeta\setminus\{k+2\})}(\cD)\prod_{j=1}^{k+1}\Bigl(1-u^{(\zeta\setminus\{j\})}(\cD)\Bigr).
  \end{align*}
  Here, by the same cyclic invariance, the quantity
  \[
  \sum_{\zeta\in\cT_{k+1}}\theta^{(\zeta)}u^{(\zeta\setminus\{j\})}(\cD)\prod_{j'\neq j}\Bigl(1-u^{(\zeta\setminus\{j'\})}(\cD)\Bigr)
  \]
  is independent of $j\in[k+2]$.
  Taking $j=k+2$ and using
  \[
  \sum_{\zeta\in\cT_{k+1}}\theta^{(\zeta)}=1
  \]
  gives the last equality.
  
  Substituting this expression into~\eqref{eq:EGm-YtildeY} gives
  \[
  E_{k,\bm\theta}(\cD)\ge\Psi_{k,c,\bm\theta}(u(\cD)),
  \]
  proving~\eqref{eq:DFPE-comparison}.

  Finally, if $\Gm_{k,c,\bm\theta}(\cD)=\cD$, the coupling above may be chosen so that
  \[
  \tilde Y_{k+2}^{(\zeta)}
  =Y_{k+2}^{(\zeta)}
  \quad\text{a.s.}
  \]
  Hence, both inequalities in~\eqref{eq:EGm-YtildeY} are equalities, and therefore equality also holds in~\eqref{eq:DFPE-comparison}.
\end{proof}

We now prove Proposition~\ref{prop:E_RSM0-h} by combining Lemmas~\ref{lem:Do-Dmax} and~\ref{lem:uGm-Phiu}.
\begin{proof}[Proof of Proposition~\ref{prop:E_RSM0-h}]
  Let $u=\bigl(u^{(\xi)}\bigr)_{\xi\in\cT_k}\in[0,1]^{\cT_k}$ be a fixed point of $\Phi_{k,c,\bm\theta}$ at which the maximum in~\eqref{eq:hk_c_theta} is attained, and define $\cD(u)\in\cP([0,1])^{\cT_k}$ by
  \[
  \cD(u)^{(\xi)}
  :=\bigl(1-u^{(\xi)}\bigr)\dl_0+u^{(\xi)}\dl_1.
  \]
  By~\eqref{eq:uTheta-Phi},
  \[
  u\bigl(\Gm_{k,c,\bm\theta}(\cD(u))\bigr)
  =\Phi_{k,c,\bm\theta}(u)=u.
  \]
  Every probability measure on $[0,1]$ having mass $u^{(\xi)}$ on $(0,1]$ is stochastically dominated by $(1-u^{(\xi)})\dl_0+u^{(\xi)}\dl_1$.
  Therefore,
  \[
  \Gm_{k,c,\bm\theta}(\cD(u))\le_{\st}\cD(u).
  \]
  Since $\Gm_{k,c,\bm\theta}$ is order-preserving, we have, by the same monotonicity and compactness argument,
  \[
  \Gm_{k,c,\bm\theta}^l(\cD(u))
  \downarrow\cD_\infty
  \quad\text{as $l\to\infty$,}
  \]
  for some $\cD_\infty\in\cP([0,1])^{\cT_k}$.
  By continuity from above,
  \[
  \Gm_{k,c,\bm\theta}(\cD_\infty)=\cD_\infty.
  \]
  Moreover,~\eqref{eq:uTheta-Phi} and the identity $u=\Phi_{k,c,\bm\theta}(u)$ imply inductively that $u\bigl(\Gm_{k,c,\bm\theta}^l(\cD(u))\bigr)=u$ for every $l\in\N_0$.
  Applying Lemma~\ref{lem:uGm-Phiu} to $\Gm_{k,c,\bm\theta}^l(\cD(u))$, we obtain
  \[
  E_{k,\bm\theta}\bigl(\Gm_{k,c,\bm\theta}^l(\cD(u))\bigr)
  \ge\Psi_{k,c,\bm\theta}(u)
  \]
  for every $l\in\N_0$.
  Since $E_{k,\bm\theta}$ is continuous with respect to the product weak topology, letting $l\to\infty$ yields
  \[
  E_{k,\bm\theta}(\cD_\infty)
  \ge\Psi_{k,c,\bm\theta}(u).
  \]
  Since $\cD_\infty\le_{\st}\cD_{\max}$ and $\Psi_{k,c,\bm\theta}(u)=h_k(c,\bm\theta)$, we obtain
  \begin{equation}\label{eq:DFPE-lower-h}
  E_{k,\bm\theta}(\cD_{\max})
  \ge E_{k,\bm\theta}(\cD_\infty)
  \ge h_k(c,\bm\theta).
  \end{equation}

  Conversely, since $\cD_{\max}$ is a fixed point of $\Gm_{k,c,\bm\theta}$,~\eqref{eq:uTheta-Phi} gives
  \[
  u(\cD_{\max})=\Phi_{k,c,\bm\theta}(u(\cD_{\max})).
  \]
  The equality assertion in Lemma~\ref{lem:uGm-Phiu} therefore yields
  \[
  E_{k,\bm\theta}(\cD_{\max})
  =\Psi_{k,c,\bm\theta}(u(\cD_{\max}))
  \le h_k(c,\bm\theta).
  \]
  Together with~\eqref{eq:DFPE-lower-h}, this proves
  \[
  E_{k,\bm\theta}(\cD_{\max})=h_k(c,\bm\theta).
  \]
  
  Finally, Lemma~\ref{lem:Do-Dmax} gives
  \begin{align*}
  \E'\bigl[\mu_{(\PGW_k(c,\bm\theta\le\bm1),\sg_o)}(\{0\})\bigr]
  &=\sum_{\xi\in\cT_k}\theta^{(\xi)}\int_{[0,1]}x\,d\cD_o^{(\xi)}(x)\\
  &=E_{k,\bm\theta}(\cD_o)
  =E_{k,\bm\theta}(\cD_{\max})
  =h_k(c,\bm\theta),
  \end{align*}
  which completes the proof.
\end{proof}

\section{Convergence of Persistent Betti Numbers}
\label{sec:conv_PB}
In this section, we prove Theorem~\ref{thm:LLN_PB} by combining local weak convergence of the multiparameter complex pairs with the spectral results established in Sections~\ref{sec:RSM_CompL} and~\ref{sec:mass_zero}.
We first state the required local weak convergence and then prove Theorem~\ref{thm:LLN_PB}, deferring the proof of the local weak convergence statement to Section~\ref{ssec:LWC_MCP}.

Recall that $\nu_k(c,\bm\theta\le\bm1)\in\cP(\cS_k)$ denotes the law under $\P'$ of the $\cS_k$-valued random variable
\[
[\PGW_k(c,\bm\theta\le\bm1),\sg_o]
=[\PGW_k(c,\bm\theta)\subset\PGW_k(c),\sg_o].
\]
By Lemma~\ref{lem:esa_PGW}, this measure is concentrated on $\cS_k^{\esa}$ and may therefore also be regarded as an element of $\cP(\cS_k^{\esa})$.
We continue to use $\xrightarrow{\LW,\,\mathrm{p}}$ for local weak convergence in probability.

\subsection{Proof of Theorem~\ref{thm:LLN_PB}}
We first state the required local weak convergence of the multiparameter complex pairs.
\begin{prop}\label{prop:LWC_MCP}
  Let $k\in\N_0$ and $c>0$, and let $\cX_n=\{X_n(t)\}_{t\ge0}$ be a multiparameter complex process on $[n]$ associated with a sequence $\bm p=(p_j)_{j=0}^\infty$ of distribution functions on $[0,\infty]$.
  Assume that the parameters $0\le s_n\le t_n<\infty$ satisfy
  \begin{equation}\label{eq:criticality}
  \lim_{n\to\infty}\tbinom n{k+1}q_k(t_n)=\infty
  \quad\text{and}\quad
  \lim_{n\to\infty}nr_k(t_n)=c,
  \end{equation}
  and that the limit
  \begin{equation}\label{eq:regularity}
  \theta_j:=\lim_{n\to\infty}\frac{p_j(s_n)}{p_j(t_n)}\in[0,1]
  \end{equation}
  exists for all $0\le j\le k+1$.
  Set $\bm\theta:=(\theta_j)_{j=0}^{k+1}$.
  Then, conditional on $\{\dim X_n(t_n)\ge k\}$,
  \begin{align}
  \bigl(\{\emptyset\}\subset X_n(t_n)\bigr)
  &\xrightarrow{\LW,\,\mathrm{p}}\nu_k(c,\bm0\le\bm1)
  \label{eq:LWC_empty-t}
  \shortintertext{and}
  \bigl(X_n(s_n)\subset X_n(t_n)\bigr)
  &\xrightarrow{\LW,\,\mathrm{p}}\nu_k(c,\bm\theta\le\bm1)
  \label{eq:LWC_s-t}
  \end{align}
  as $n\to\infty$.
\end{prop}

To combine Proposition~\ref{prop:LWC_MCP} with the results of the preceding sections, we also use the following face-number estimates.
\begin{lem}[{\cite[Lemmas~20 and~23]{Kan2022}}]\label{lem:f_j}
  Let $k\in\N_0$, and let $t_n$ be a nonnegative parameter such that
  \[
  \lim_{n\to\infty}\tbinom n{k+1}q_k(t_n)=\infty.
  \]
  Then the following hold for every $p\in[1,\infty)$.
  \begin{enumerate}
  \item For each $0\le j\le k$,
  \[
  \lim_{n\to\infty}n^{j+1}q_j(t_n)=\infty
  \quad\text{and}\quad
  \frac{f_j(X_n(t_n))}{\binom n{j+1}q_j(t_n)}
  \xrightarrow{L^p}1
  \]
  as $n\to\infty$.
  In particular,
  \[
  \lim_{n\to\infty}\P(\dim X_n(t_n)\ge k)=1.
  \]
  \item If $\lim_{n\to\infty}nr_k(t_n)=c$ for some constant $c>0$, then, for every $j\ge-1$,
  \[
  \frac{f_j(X_n(t_n))}{\binom n{k+1}q_k(t_n)}
  \xrightarrow{L^p}
  \begin{cases}
    1
    &\text{if $j=k$,}\\
    c/(k+2)
    &\text{if $j=k+1$,}\\
    0
    &\text{otherwise}
  \end{cases}
  \]
  as $n\to\infty$.
  \end{enumerate}
\end{lem}

We are now in a position to complete the proof of Theorem~\ref{thm:LLN_PB}.
\begin{proof}[Proof of Theorem~\ref{thm:LLN_PB}]
  Since each $(k+1)$-simplex has exactly $k+2$ distinct $k$-faces, a standard double-counting argument yields
  \[
  (k+2)f_{k+1}(X_n(t_n))
  =\sum_{\sg\in F_k(X_n(t_n))}\deg_{X_n(t_n)}(\sg).
  \]
  Therefore, by Lemma~\ref{lem:f_j}, conditional on $\{\dim X_n(t_n)\ge k\}$,
  \begin{align*}
  \frac1{f_k(X_n(t_n))}\sum_{\sg\in F_k(X_n(t_n))}\deg_{X_n(t_n)}(\sg)
  &=(k+2)\frac{f_{k+1}(X_n(t_n))}{f_k(X_n(t_n))}\\
  &=(k+2)\frac{f_{k+1}(X_n(t_n))}{\binom n{k+1}q_k(t_n)}
  \frac{\binom n{k+1}q_k(t_n)}{f_k(X_n(t_n))}
  \xrightarrow{\mathrm p}c.
  \end{align*}
  The conditioning does not affect this convergence, since $\P(\dim X_n(t_n)\ge k)\to1$ by Lemma~\ref{lem:f_j}(1).

  For simplicity, write
  \[
  N_k(K,X_n(t_n))
  :=\nullity\bigl(\Dl_k^\up(X_n(t_n)\setminus K)\bigr)
  \]
  for $K=X_n(s_n)$ or $\{\emptyset\}$.
  By Lemma~\ref{lem:comp_L}(2),
  \begin{equation}\label{eq:bNNB}
  \b_k\bigl(X_n(s_n)\subset X_n(t_n)\bigr)
  =N_k\bigl(\{\emptyset\},X_n(t_n)\bigr)
  -N_k\bigl(X_n(s_n),X_n(t_n)\bigr)
  -\dim B^k(X_n(s_n)).
  \end{equation}
  On the event $\{\dim X_n(t_n)\ge k\}$,
  \begin{equation}\label{eq:NXX-h}
  \begin{aligned}
  &\biggl|\frac{N_k\bigl(X_n(s_n),X_n(t_n)\bigr)}{\binom n{k+1}q_k(t_n)}-h_k(c,\bm\theta)\biggr|\\
  &\quad\le\biggl|\frac{N_k\bigl(X_n(s_n),X_n(t_n)\bigr)}{\binom n{k+1}q_k(t_n)}-\frac{N_k\bigl(X_n(s_n),X_n(t_n)\bigr)}{f_k(X_n(t_n))}\biggr|+\biggl|\frac{N_k\bigl(X_n(s_n),X_n(t_n)\bigr)}{f_k(X_n(t_n))}-h_k(c,\bm\theta)\biggr|\\
  &\quad\le\biggl|\frac{f_k(X_n(t_n))}{\binom n{k+1}q_k(t_n)}-1\biggr|+\biggl|\frac{N_k\bigl(X_n(s_n),X_n(t_n)\bigr)}{f_k(X_n(t_n))}-h_k(c,\bm\theta)\biggr|.
  \end{aligned}
  \end{equation}
  The last inequality uses the elementary bound
  \[
  N_k\bigl(X_n(s_n),X_n(t_n)\bigr)
  \le f_k(X_n(t_n)\setminus X_n(s_n))
  \le f_k(X_n(t_n)).
  \]
  
  The first term on the right-hand side of~\eqref{eq:NXX-h} converges in probability to zero by Lemma~\ref{lem:f_j}(1).
  Furthermore, conditional on $\{\dim X_n(t_n)\ge k\}$, Proposition~\ref{prop:LWC_MCP}, together with Corollary~\ref{cor:conti_ARSM0}, yields
  \[
  \frac{N_k\bigl(X_n(s_n),X_n(t_n)\bigr)}{f_k(X_n(t_n))}
  \xrightarrow{\mathrm p}
  \overline\RSM_0\bigl(\nu_k(c,\bm\theta\le\bm1)\bigr).
  \]
  By Proposition~\ref{prop:E_RSM0-h}, the limit on the right-hand side equals $h_k(c,\bm\theta)$.
  Consequently, under the same conditioning, the second term on the right-hand side of~\eqref{eq:NXX-h} also converges in probability to zero.
  Since
  \[
  \lim_{n\to\infty}\P(\dim X_n(t_n)\ge k)=1,
  \]
  it follows that
  \begin{equation}\label{eq:NXXh}
  \frac{N_k\bigl(X_n(s_n),X_n(t_n)\bigr)}{\binom n{k+1}q_k(t_n)}
  \xrightarrow{\mathrm p}h_k(c,\bm\theta).
  \end{equation}
  The same argument, applied to the empty marking complex, yields
  \begin{equation}\label{eq:NeXh}
  \frac{N_k\bigl(\{\emptyset\},X_n(t_n)\bigr)}{\binom n{k+1}q_k(t_n)}
  \xrightarrow{\mathrm p}h_k(c,\bm0)=h_k(c).
  \end{equation}
  Furthermore, by Lemma~\ref{lem:f_j}(2),
  \begin{equation}\label{eq:Bff0}
  \frac{\dim B^k(X_n(s_n))}{\binom n{k+1}q_k(t_n)}
  \le\frac{f_{k-1}(X_n(s_n))}{\binom n{k+1}q_k(t_n)}
  \le\frac{f_{k-1}(X_n(t_n))}{\binom n{k+1}q_k(t_n)}
  \xrightarrow[n\to\infty]{\mathrm p}0.
  \end{equation}
  Combining~\eqref{eq:bNNB},~\eqref{eq:NXXh},~\eqref{eq:NeXh}, and~\eqref{eq:Bff0}, we conclude that
  \[
  \frac{\b_k\bigl(X_n(s_n)\subset X_n(t_n)\bigr)}{\binom n{k+1}q_k(t_n)}
  \xrightarrow{\mathrm p}h_k(c)-h_k(c,\bm\theta)
  \quad\text{as $n\to\infty$.}
  \]

  Finally, fix $1\le p<q<\infty$ and set
  \[
  Y_n
  :=\frac{\b_k\bigl(X_n(s_n)\subset X_n(t_n)\bigr)}{\binom n{k+1}q_k(t_n)}
  \quad\text{and}\quad
  y:=h_k(c)-h_k(c,\bm\theta).
  \]
  By Lemma~\ref{lem:f_j}(1) and the bound
  \[
  0\le Y_n
  \le\frac{f_k(X_n(t_n))}{\binom n{k+1}q_k(t_n)},
  \]
  we have
  \[
  \sup_{n\in\N}\E[|Y_n-y|^q]<\infty.
  \]
  Since $q>p$, the sequence $(|Y_n-y|^p)_{n=1}^\infty$ is uniformly integrable.
  The convergence $Y_n\xrightarrow{\mathrm p}y$ therefore implies $Y_n\xrightarrow{L^p}y$.
\end{proof}

The remainder of this section is devoted to the proof of Proposition~\ref{prop:LWC_MCP}.

\subsection{Local Weak Convergence of Multiparameter Complex Pairs}
\label{ssec:LWC_MCP}
In this subsection, we prove Proposition~\ref{prop:LWC_MCP}.
The proof follows the approach of~\cite[Section~4]{Kan2022}, where local weak convergence is established for multiparameter random complexes without marking complexes.
We first realize the marking complex of the multiparameter complex pair as a thinned subcomplex of the ambient complex.

Let $0\le s_n\le t_n<\infty$ be parameters as in Proposition~\ref{prop:LWC_MCP}.
Given $X_n(t_n)$, assign to each nonempty $j$-simplex $\rho$~($0\le j\le n-1$) an independent random variable
\[
\eta_n(\rho)\sim\Bernoulli\biggl(\frac{p_j(s_n)}{p_j(t_n)}\biggr),
\]
and define a subcomplex
\[
K_n(s_n,t_n):=\{\sg\in X_n(t_n)\mid\text{$\eta_n(\rho)=1$ for every nonempty face $\rho\subset\sg$}\}.
\]
The random simplicial pairs $K_n(s_n,t_n)\subset X_n(t_n)$ and $X_n(s_n)\subset X_n(t_n)$ have the same law.
Indeed, conditional on $X_n(t_n)=Y$, the indicators $\1_{\{w_\rho\le s_n\}}$, $\rho\in Y\setminus\{\emptyset\}$, are independent Bernoulli random variables with respective success probabilities $p_{\dim(\rho)}(s_n)/p_{\dim(\rho)}(t_n)$.
A simplex $\sg\in Y$ belongs to $X_n(s_n)$ if and only if all indicators corresponding to its nonempty faces equal $1$.
Thus, we may work with $K_n(s_n,t_n)$ in place of $X_n(s_n)$.

We first prove the following one-root and two-root local convergence statements.
\begin{lem}\label{lem:LC_MCP}
  Let $k\in\N_0$ and $c>0$.
  Let $s_n$ and $t_n$ be as in Proposition~\ref{prop:LWC_MCP}, and set $\bm\theta:=(\theta_j)_{j=0}^{k+1}\in[0,1]^{k+2}$.
  The following hold:
  \begin{enumerate}
  \item Let $\sg$ be a $k$-simplex in $K(n)$.
  Then, conditional on $\{\sg\in X_n(t_n)\}$, the $\cS_k$-valued random variable
  \[
  [K_n(s_n,t_n)\subset X_n(t_n),\sg]_\infty
  \]
  converges in distribution to $[\PGW_k(c,\bm\theta\le\bm1),\sg_o]$ as $n\to\infty$.
  \item Let $\sg,\sg'$ be two disjoint $k$-simplices in $K(n)$.
  Then, conditional on $\{\sg,\sg'\in X_n(t_n)\}$, the $(\cS_k\times\cS_k)$-valued random variable
  \[
  \bigl([K_n(s_n,t_n)\subset X_n(t_n),\sg]_\infty,
  [K_n(s_n,t_n)\subset X_n(t_n),\sg']_\infty\bigr)
  \]
  converges in distribution to $\bigl([\PGW_k(c,\bm\theta\le\bm1),\sg_o],
  [\PGW'_k(c,\bm\theta\le\bm1),\sg'_o]\bigr)$ as $n\to\infty$.
  Here, $[\PGW'_k(c,\bm\theta\le\bm1),\sg'_o]$ is an independent copy of $[\PGW_k(c,\bm\theta\le\bm1),\sg_o]$.
  \end{enumerate}
\end{lem}
\begin{proof}
  Write
  \[
  \P_\sg(\,\cdot\,)
  :=\P\bigl(\,\cdot\mid\sg\in X_n(t_n)\bigr)
  \quad\text{and}\quad
  \P_{\sg,\sg'}(\,\cdot\,)
  :=\P\bigl(\,\cdot\mid\sg,\sg'\in X_n(t_n)\bigr).
  \]
  We also set
  \[
  \theta_{n,j}:=\frac{p_j(s_n)}{p_j(t_n)},
  \qquad 0\le j\le k+1,
  \]
  and write $\bm\theta_n:=(\theta_{n,j})_{j=0}^{k+1}$.

  For a finite $k$-rooted simplicial pair $(L\subset Y,\tau)$ with $[L\subset Y,\tau]\in\cS_k$, define its thinning probability as follows.
  For a map $\eta\colon Y\setminus\{\emptyset\}\to\{0,1\}$, define
  \[
  L_\eta(Y)
  :=\{\sg\in Y\mid\eta(\rho)=1\text{ for every nonempty face $\rho\subset\sg$}\}.
  \]
  Let $\cA(L,Y,\tau)$ denote the set of all such maps $\eta$ for which
  \[
  (L_\eta(Y)\subset Y,\tau)\simeq(L\subset Y,\tau).
  \]
  For $\bm\theta'=(\theta'_j)_{j=0}^{k+1}\in[0,1]^{k+2}$, define
  \[
  \pi_{\bm\theta'}(L,Y,\tau)
  :=\sum_{\eta\in\cA(L,Y,\tau)}\prod_{\rho\in Y\setminus\{\emptyset\}}(\theta'_{\dim\rho})^{\eta(\rho)}
  (1-\theta'_{\dim\rho})^{1-\eta(\rho)}.
  \]
  Thus, $\pi_{\bm\theta'}(L,Y,\tau)$ is the probability that the $\bm\theta'$-thinning of $Y$ produces a rooted simplicial pair isomorphic to $(L\subset Y,\tau)$.
  In particular, it is a polynomial in $\bm\theta'$.

  \smallskip
  (1)~Fix $[\bm Y,\tau]=[L\subset Y,\tau]\in\cS_k$ and $l\in\N$.
  By Lemma~\ref{lem:Sk_basics}(2), it suffices to prove that
  \[
  \lim_{n\to\infty}\P_\sg\bigl((K_n(s_n,t_n)\subset X_n(t_n),\sg)_l\simeq(\bm Y,\tau)_l\bigr)
  =\P'\bigl((\PGW_k(c,\bm\theta\le\bm1),\sg_o)_l\simeq(\bm Y,\tau)_l\bigr).
  \]
  Let $L_l$ and $Y_l$ denote the marking complex and the ambient complex of $(L\subset Y,\tau)_l$, respectively.
  Conditional on the ambient rooted ball of $(K_n(s_n,t_n)\subset X_n(t_n),\sg)$, its marking complex is obtained by the independent $\bm\theta_n$-thinning.
  Therefore,
  \begin{equation}\label{eq:KX-LY}
  \begin{aligned}
  &\P_\sg\bigl((K_n(s_n,t_n)\subset X_n(t_n),\sg)_l\simeq(\bm Y,\tau)_l\bigr)\\
  &\quad=\pi_{\bm\theta_n}(L_l,Y_l,\tau)\,\P_\sg\bigl((X_n(t_n),\sg)_l\simeq(Y,\tau)_l\bigr).
  \end{aligned}
  \end{equation}
  The local convergence of the ambient multiparameter random complexes established in~\cite[Proposition~21]{Kan2022} gives
  \[
  \P_\sg\bigl(
  (X_n(t_n),\sg)_l\simeq(Y,\tau)_l
  \bigr)
  \longrightarrow
  \P'\bigl(
  (\PGW_k(c),\sg_o)_l\simeq(Y,\tau)_l
  \bigr).
  \]
  Moreover,~\eqref{eq:regularity} gives $\bm\theta_n\to\bm\theta$, and hence
  \[
  \pi_{\bm\theta_n}(L_l,Y_l,\tau)
  \longrightarrow\pi_{\bm\theta}(L_l,Y_l,\tau)
  \quad\text{as $n\to\infty$.}
  \]
  It follows from~\eqref{eq:KX-LY} that
  \begin{align*}
  &\lim_{n\to\infty}\P_\sg\bigl(
  (K_n(s_n,t_n)\subset X_n(t_n),\sg)_l\simeq(\bm Y,\tau)_l\bigr)\\
  &\quad=\pi_{\bm\theta}(L_l,Y_l,\tau)\,\P'\bigl((\PGW_k(c),\sg_o)_l\simeq(Y,\tau)_l\bigr)\\
  &\quad=\P'\bigl((\PGW_k(c,\bm\theta\le\bm1),\sg_o)_l\simeq(\bm Y,\tau)_l\bigr).
  \end{align*}
  The last equality follows from the definition of $\PGW_k(c,\bm\theta)$ as the $\bm\theta$-thinning of $\PGW_k(c)$.
  This proves~(1).

  \smallskip
  (2)~Equip $\cS_k\times\cS_k$ with the $\ell^\infty$-product metric induced by $d_{\loc}$.
  Then $\cS_k\times\cS_k$ is again a separable ultrametric space.
  It therefore suffices to prove that, for every fixed $[\bm Y,\tau]=[L\subset Y,\tau]\in\cS_k$, $[\bm Y',\tau']=[L'\subset Y',\tau']\in\cS_k$, and $l,l'\in\N$,
  \begin{align*}
  &\lim_{n\to\infty}\P_{\sg,\sg'}\bigl((K_n(s_n,t_n)\subset X_n(t_n),\sg)_l\simeq(\bm Y,\tau)_l,\,(K_n(s_n,t_n)\subset X_n(t_n),\sg')_{l'}\simeq(\bm Y',\tau')_{l'}\bigr)\\
  &=\P'\bigl((\PGW_k(c,\bm\theta\le\bm1),\sg_o)_l\simeq(\bm Y,\tau)_l\bigr)\cdot\P'\bigl((\PGW'_k(c,\bm\theta\le\bm1),\sg'_o)_{l'}\simeq(\bm Y',\tau')_{l'}\bigr).
  \end{align*}
  Let $L_l,Y_l,L'_{l'}$, and $Y'_{l'}$ denote the corresponding marking and ambient complexes.

  Let $D_n$ be the event that the two ambient rooted balls $(X_n(t_n),\sg)_l$ and $(X_n(t_n),\sg')_{l'}$ are vertex-disjoint.
  The breadth-first traversal argument used in the proof of~\cite[Section~4]{Kan2022} also shows that
  \[
  \P_{\sg,\sg'}(D_n^c)\longrightarrow0
  \quad\text{as $n\to\infty$.}
  \]
  Moreover, the joint local convergence of the ambient multiparameter random complexes at two disjoint $k$-simplices established in~\cite[Proposition~22]{Kan2022} gives
  \begin{equation}\label{eq:prop_22}
  \begin{aligned}
  &\P_{\sg,\sg'}\bigl((X_n(t_n),\sg)_l\simeq(Y,\tau)_l,\,(X_n(t_n),\sg')_{l'}\simeq(Y',\tau')_{l'}\bigr)\\
  &\quad\longrightarrow\P'\bigl((\PGW_k(c),\sg_o)_l\simeq(Y,\tau)_l\bigr)\cdot\P'\bigl((\PGW'_k(c),\sg'_o)_{l'}\simeq(Y',\tau')_{l'}\bigr).
  \end{aligned}
  \end{equation}
  Set
  \begin{align*}
  A_n
  &:=\bigl\{(K_n(s_n,t_n)\subset X_n(t_n),\sg)_l\simeq(\bm Y,\tau)_l,(K_n(s_n,t_n)\subset X_n(t_n),\sg')_{l'}\simeq(\bm Y',\tau')_{l'}\bigr\}
  \shortintertext{and}
  B_n
  &:=\bigl\{(X_n(t_n),\sg)_l\simeq(Y,\tau)_l,(X_n(t_n),\sg')_{l'}\simeq(Y',\tau')_{l'}\bigr\}.
  \end{align*}
  On $D_n$, the two families of Bernoulli variables used for the thinning are disjoint and hence independent.
  Therefore,
  \[
  \P_{\sg,\sg'}(A_n\cap D_n)
  =\pi_{\bm\theta_n}(L_l,Y_l,\tau)\,\pi_{\bm\theta_n}(L'_{l'},Y'_{l'},\tau')\,\P_{\sg,\sg'}(B_n\cap D_n).
  \]
  Consequently,
  \[
  \bigl|\P_{\sg,\sg'}(A_n)-\pi_{\bm\theta_n}(L_l,Y_l,\tau)\,\pi_{\bm\theta_n}(L'_{l'},Y'_{l'},\tau')\,\P_{\sg,\sg'}(B_n)\bigr|
  \le2\P_{\sg,\sg'}(D_n^c)
  \longrightarrow0.
  \]
  By~\eqref{eq:prop_22} and the convergence $\bm\theta_n\to\bm\theta$, it follows that
  \begin{align*}
  \lim_{n\to\infty}\P_{\sg,\sg'}(A_n)
  &=\pi_{\bm\theta}(L_l,Y_l,\tau)\,\P'\bigl(
  (\PGW_k(c),\sg_o)_l\simeq(Y,\tau)_l\bigr)\\
  &\qquad\cdot\pi_{\bm\theta}(L'_{l'},Y'_{l'},\tau')\,\P'\bigl((\PGW'_k(c),\sg'_o)_{l'}\simeq(Y',\tau')_{l'}\bigr)\\
  &=\P'\bigl((\PGW_k(c,\bm\theta\le\bm1),\sg_o)_l
  \simeq(\bm Y,\tau)_l\bigr)\\
  &\qquad\cdot\P'\bigl((\PGW'_k(c,\bm\theta\le\bm1),\sg'_o)_{l'}\simeq(\bm Y',\tau')_{l'}\bigr).
  \end{align*}
  This proves the joint convergence in~(2).
\end{proof}

\begin{proof}[Proof of Proposition~\ref{prop:LWC_MCP}]
  As noted above, we may replace $X_n(s_n)$ with $K_n(s_n,t_n)$ without loss of generality.
  For simplicity, set
  \[
  \bm X_n:=\bigl(K_n(s_n,t_n)\subset X_n(t_n)\bigr).
  \]
  We show that, conditional on $\{\dim X_n(t_n)\ge k\}$,
  \[
  \nu_{k,n}:=U_k\bigl(\bm X_n\bigr)
  =\frac1{f_k(X_n(t_n))}\sum_{\sg\in F_k(X_n(t_n))}\dl_{[\bm X_n,\sg]_\infty}
  \]
  converges in probability to $\nu_k(c,\bm\theta\le\bm1)$ in $\cP(\cS_k)$ as $n\to\infty$.

  The proof is the same second-moment argument as in the proof of~\cite[Theorem~13]{Kan2022}.
  Since the counting part of that argument is unaffected by the marking complex, we record only the main steps and indicate where Lemma~\ref{lem:LC_MCP} is used.
  
  Let $g\colon\cS_k\to\R$ be bounded and continuous, and define
  \[
  \tilde\nu_{k,n}(g)
  :=\frac1{\binom n{k+1}q_k(t_n)}\sum_{\sg\in F_k(K(n))}\1_{\{\sg\in X_n(t_n)\}}g\bigl([\bm X_n,\sg]_\infty\bigr).
  \]
  By the homogeneity of the model, for any $\sg\in F_k(K(n))$,
  \[
  \E\bigl[\tilde\nu_{k,n}(g)\bigr]
  =\E\bigl[g\bigl([K_n(s_n,t_n)\subset X_n(t_n),\sg]_\infty\bigr)\bigm|\sg\in X_n(t_n)\bigr],
  \]
  which converges to $\E'\bigl[g\bigl([\PGW_k(c,\bm\theta\le\bm1),\sg_o]\bigr)\bigr]$ as $n\to\infty$ by Lemma~\ref{lem:LC_MCP}(1).
  
  We next consider the second moment.
  Expanding the square of $\tilde\nu_{k,n}(g)$, we decompose the resulting double sum over ordered pairs of $k$-simplices according to whether they are disjoint or intersect.
  The contribution from disjoint pairs is
  \begin{align*}
  &\frac1{\tbinom n{k+1}^2q_k(t_n)^2}\sum_{\substack{\sg,\sg'\in F_k(K(n))\colon\\\sg\cap\sg'=\emptyset}}\E\bigl[\1_{\{\sg,\sg'\in X_n(t_n)\}}g\bigl([\bm X_n,\sg]_\infty\bigr)g\bigl([\bm X_n,\sg']_\infty\bigr)\bigr]\\
  &=\frac{\binom{n-k-1}{k+1}}{\binom n{k+1}}\E\bigl[g\bigl([\bm X_n,\sg]_\infty\bigr)g\bigl([\bm X_n,\sg']_\infty\bigr)\bigm|\sg,\sg'\in X_n(t_n)\bigr],
  \end{align*}
  where $\sg$ and $\sg'$ are any two disjoint $k$-simplices.
  Hence, Lemma~\ref{lem:LC_MCP}(2) shows that this contribution converges to $\E'\bigl[g\bigl([\PGW_k(c,\bm\theta\le\bm1),\sg_o]\bigr)\bigr]^2$ as $n\to\infty$.

  If two $k$-simplices have exactly $j$ common vertices, then the number of such ordered pairs is
  \[
  \binom n{k+1}\binom{k+1}j\binom{n-k-1}{k+1-j},
  \]
  and their joint probability of belonging to $X_n(t_n)$ is
  \[
  \frac{q_k(t_n)^2}{q_{j-1}(t_n)}.
  \]
  Consequently, the absolute value of the contribution from intersecting pairs is at most
  \[
  \frac{\|g\|_\infty^2}{\tbinom n{k+1}}\sum_{j=1}^{k+1}\binom{k+1}j\binom{n-k-1}{k+1-j}\frac1{q_{j-1}(t_n)},
  \]
  which converges to zero under~\eqref{eq:criticality} by Lemma~\ref{lem:f_j}(1).
  It follows that
  \[
  \E\bigl[\tilde\nu_{k,n}(g)^2\bigr]
  \longrightarrow\E'\bigl[g\bigl([\PGW_k(c,\bm\theta\le\bm1),\sg_o]\bigr)\bigr]^2.
  \]
  Together with the convergence of the first moment, this yields
  \[
  \tilde\nu_{k,n}(g)
  \xrightarrow[n\to\infty]{\mathrm p}\E'\bigl[g\bigl([\PGW_k(c,\bm\theta\le\bm1),\sg_o]\bigr)\bigr].
  \]

  On the event $\{\dim X_n(t_n)\ge k\}$, we have
  \[
  \int_{\cS_k}g\,d\nu_{k,n}
  =\frac{\tilde\nu_{k,n}(g)}{f_k(X_n(t_n))/\bigl(\tbinom n{k+1}q_k(t_n)\bigr)}.
  \]
  By Lemma~\ref{lem:f_j}(1), it follows that, conditional on $\{\dim X_n(t_n)\ge k\}$,
  \[
  \int_{\cS_k}g\,d\nu_{k,n}
  \xrightarrow{\mathrm p}\E'\bigl[g\bigl([\PGW_k(c,\bm\theta\le\bm1),\sg_o]\bigr)\bigr].
  \]
  Since $\cS_k$ is a complete separable metric space, a standard countable convergence-determining argument yields~\eqref{eq:LWC_s-t}.

  The convergence~\eqref{eq:LWC_empty-t} follows from the same argument with the marking complex identically empty, using~\cite[Propositions~21 and~22]{Kan2022} in place of Lemma~\ref{lem:LC_MCP}.
\end{proof}

\section{Convergence of Persistence Diagrams}
\label{sec:conv_PD}
In this section, we prove Theorems~\ref{thm:LLN_PD} and~\ref{thm:LLN_PD_test} and discuss Corollary~\ref{cor:LLN_OT}.
We first establish vague convergence and identify the limiting measure through its quadrant masses.
We then use tail estimates to prove convergence of the total masses and obtain weak convergence.
A truncation argument yields convergence for unbounded test functions.
Finally, we combine the convergence of persistence diagrams and their total persistence to obtain convergence in the optimal transport distance.
Throughout this section, fix $k\in\N_0$ and use reduced homology.

\subsection{Vague Convergence of Persistence Diagrams}
Our first step is the following vague convergence result.
\begin{prop}\label{prop:vague_PD}
  Assume that Conditions~\textup{\ref{cond:C1}}--\textup{\ref{cond:C3}} hold.
  Then there exists a unique deterministic Radon measure $\sfD_k\in\cM(\Dl)$ satisfying
  \begin{equation}\label{eq:vague_PD}
  \sfD_k\bigl([0,x]\times(y,\infty]\bigr)=\hat\b_k(x,y),
  \qquad 0<x<y<\infty.
  \end{equation}
  Moreover,
  \[
  \sfD_{k,n}\xrightarrow{\mathrm p}\sfD_k
  \quad\text{and}\quad
  \E[\sfD_{k,n}]\longrightarrow\sfD_k
  \quad\text{as $n\to\infty$}
  \]
  in $\cM(\Dl)$ equipped with the vague topology.
\end{prop}

We first record some consequences of Conditions~\textup{\ref{cond:C2}} and~\textup{\ref{cond:C3}}.
\begin{lem}\label{lem:theta_Theta}
  Under Conditions~\textup{\ref{cond:C2}} and~\textup{\ref{cond:C3}}, the following hold.
  \begin{enumerate}
    \item For every $0\le u<r_k(t_0)$,
    \[
    r_k\bigl(r_k^\rightarrow(u)\bigr)=u.
    \]
    \item Each $\theta_{k,j}$ is positive and continuous on $\Dl\cap(0,\infty)^2$.
    \item The function $\Theta_k$ is well-defined, positive, continuous, and nondecreasing on $(0,\infty)$, with $\Theta_k(1)=1$.
  \end{enumerate}
\end{lem}
The proof is given in Section~S3.1 of the Supplementary Material~\cite{HKsupp}.
Under Condition~\ref{cond:C2}, the limit for $j=k+1$ in Condition~\ref{cond:C3} follows from those for $0\le j\le k$ by the product formula for $r_k$.
We include it in Condition~\ref{cond:C3} to keep the formulation uniform across dimensions.

We also need the following continuity property of the limiting constant.
\begin{lem}\label{lem:hk_conti}
  Let $k\in\N_0$.
  For any $c,c'\ge0$ and $\bm\theta=(\theta_j)_{j=0}^{k+1},\bm\theta'=(\theta'_j)_{j=0}^{k+1}\in[0,1]^{k+2}$,
  \begin{equation}\label{eq:hk_conti}
  \bigl|h_k(c,\bm\theta)-h_k(c',\bm\theta')\bigr|\le\frac{|c-c'|}{k+2}+\sum_{j=0}^k\binom{k+1}{j+1}|\theta_j-\theta_j'|.
  \end{equation}
  In particular, $(c,\bm\theta)\mapsto h_k(c,\bm\theta)$ is continuous.
\end{lem}
The estimate follows from a common coupling and the fact that adding or deleting one row or column changes nullity by at most one.
The proof is given in Section~S3.2 of the Supplementary Material~\cite{HKsupp}.

We use the following criterion for passing from convergence of persistent Betti numbers to vague convergence of persistence diagrams.
We state it in a form adapted to the present setting, including the identification of the limiting quadrant masses when their limits are continuous.

\begin{lem}[cf.~{\cite[Proposition~3.4]{HST2018}; see also~\cite[Theorem~18]{KHMT2024}}]\label{lem:lifting}
  Let $\{\sfD_n\}_{n\in\N}$ be a sequence of $\cM(\Dl)$-valued random variables, and write
  \[
  Q(x,y):=[0,x]\times(y,\infty],
  \qquad 0<x<y<\infty.
  \]
  Suppose that $\E[\sfD_n]\in\cM(\Dl)$ for every $n\in\N$ and that the limit
  \[
  c_{x,y}:=\lim_{n\to\infty}\E[\sfD_n(Q(x,y))]
  \]
  exists in $[0,\infty)$ for every $0<x<y<\infty$.
  Then there exists a unique deterministic Radon measure $\sfD\in\cM(\Dl)$ such that $\E[\sfD_n]$ converges vaguely to $\sfD$ as $n\to\infty$.

  Suppose further that, for every $0<x<y<\infty$,
  \[
  \sfD_n(Q(x,y))\longrightarrow c_{x,y}
  \quad\text{as $n\to\infty$}
  \]
  almost surely.
  Then $\sfD_n$ converges vaguely to $\sfD$ almost surely.
  The same assertion holds with convergence in probability in place of almost sure convergence.

  Moreover, if $(x,y)\mapsto c_{x,y}$ is continuous on $\{(x,y)\mid0<x<y<\infty\}$, then
  \[
  \sfD(Q(x,y))=c_{x,y}
  \]
  for every $0<x<y<\infty$.
\end{lem}
The proof is given in Section~S3.3 of the Supplementary Material~\cite{HKsupp}.

We now prove the vague convergence result by applying Lemma~\ref{lem:lifting}.
\begin{proof}[Proof of Proposition~\ref{prop:vague_PD}]
  Set $Q(x,y):=[0,x]\times(y,\infty]$ for $0<x<y<\infty$.
  The quadrant identity~\eqref{eq:PD-PB} and Theorem~\ref{thm:LLN_PB} give
  \begin{equation}\label{eq:Pi_kn-Q}
  \sfD_{k,n}(Q(x,y))
  =\frac{\b_k\bigl(\tilde X_n(x)\subset\tilde X_n(y)\bigr)}{\binom n{k+1}q_k\bigl(r_k^\rightarrow(1/n)\bigr)}
  \xrightarrow{L^1}\hat\b_k(x,y).
  \end{equation}
  Indeed, set $s_n:=r_k^\rightarrow(x/n)$ and $t_n:=r_k^\rightarrow(y/n)$.
  By Lemma~\ref{lem:theta_Theta}(1), we have $nr_k(t_n)=y$ for all sufficiently large $n$.
  Moreover, $q_k(t_n)/q_k(r_k^\rightarrow(1/n))\to\Theta_k(y)>0$ by Lemma~\ref{lem:theta_Theta}(3).
  Thus, Condition~\ref{cond:C1} implies $\binom n{k+1}q_k(t_n)\to\infty$, and Condition~\ref{cond:C3} supplies the required thinning limits.
  Hence, Theorem~\ref{thm:LLN_PB} applies with $c=y$ and $\bm\theta=\bm\theta_k(x,y)$.
  Together with the convergence of the above ratio, this yields~\eqref{eq:Pi_kn-Q}.
  Since the mean measures $\E[\sfD_{k,n}]$ are finite and $\hat\b_k$ is continuous by Lemmas~\ref{lem:theta_Theta} and~\ref{lem:hk_conti}, Lemma~\ref{lem:lifting} applied to~\eqref{eq:Pi_kn-Q} yields the claimed vague convergence and the quadrant identity~\eqref{eq:vague_PD}.
  The uniqueness of the measure satisfying this identity follows from Remark~\ref{rmk:unique}.
\end{proof}

\subsection{Total Masses and Weak Convergence}
The following estimates control the expected mass at large death times and near the diagonal.
They are used here to prove convergence of the total masses and later to handle unbounded test functions.
To avoid interrupting the convergence argument, we defer their proof to Appendix~\ref{app:PD_tails}.

\begin{lem}\label{lem:tails}
  Assume that Conditions~\textup{\ref{cond:C1}}--\textup{\ref{cond:C3}} hold.
  There exist $n_0\in\N$ and constants $C_m>0$~$(m\in\N)$ and $C>0$ such that
  \begin{equation}\label{eq:d_tail}
  \sup_{n\ge n_0}\E\bigl[\sfD_{k,n}(\{(x,y)\in\Dl\mid y>R\})\bigr]\le C_mR^{-m}
  \end{equation}
  for every $m\in\N$ and $R\ge1$, and
  \begin{equation}\label{eq:diagonal_tail}
  \sup_{n\ge n_0}\E\bigl[\sfD_{k,n}(\{(x,y)\in\Dl_\fin\mid y\le R,\,y-x\le h\})\bigr]\le Ch(1+R)^{k+1}
  \end{equation}
  for every $R\ge1$ and $0<h\le1$.
\end{lem}

We now use Lemma~\ref{lem:tails} to prove convergence of the total masses, which, together with vague convergence, yields weak convergence.
\begin{prop}\label{prop:total_mass}
  Assume that Conditions~\textup{\ref{cond:C1}}--\textup{\ref{cond:C3}} hold.
  Then $\sfD_k(\Dl)<\infty$ and
  \[
  \sfD_{k,n}(\Dl)\xrightarrow{L^1}\sfD_k(\Dl)
  \quad\text{as $n\to\infty$.}
  \]
\end{prop}
\begin{proof}
    For $R\ge1$ and $0<h\le1$, the set
  \[
  K_{R,h}:=\{(x,y)\in\Dl\mid y\le R,\,y-x\ge h\}
  \]
  is compact in $\Dl$.
  Lemma~\ref{lem:tails} gives
  \[
  \sup_{n\ge n_0}\E[\sfD_{k,n}(\Dl\setminus K_{R,h})]
  \le C_1R^{-1}+Ch(1+R)^{k+1}.
  \]
  Taking $R$ large and then $h$ small proves uniform tightness of the mean measures.
  Together with their vague convergence from Proposition~\ref{prop:vague_PD}, this implies that $\sfD_k(\Dl)<\infty$ and
  \begin{equation}\label{eq:E_total_mass}
  \lim_{n\to\infty}\E[\sfD_{k,n}(\Dl)]=\sfD_k(\Dl).
  \end{equation}

  For any $\eps>0$, choose $\varphi\in C_c(\Dl)$ with $0\le\varphi\le1$ such that
  \[
  \sfD_k(\Dl)-\int_\Dl\varphi\,d\sfD_k<\eps.
  \]
  Since $\sfD_{k,n}(\Dl)\ge\int_\Dl\varphi\,d\sfD_{k,n}$, we have
  \[
  \sfD_k(\Dl)-\sfD_{k,n}(\Dl)
  \le\eps+\left|\int_\Dl\varphi\,d\sfD_{k,n}-\int_\Dl\varphi\,d\sfD_k\right|.
  \]
  Vague convergence in probability and the arbitrariness of $\eps$ therefore imply that the positive part
  \[
  \bigl(\sfD_k(\Dl)-\sfD_{k,n}(\Dl)\bigr)_+\xrightarrow{\mathrm p}0.
  \]
  These variables are bounded by the finite constant $\sfD_k(\Dl)$, so their expectations also converge to zero.
  Combining this with~\eqref{eq:E_total_mass}, we obtain
  \[
  \E[|\sfD_{k,n}(\Dl)-\sfD_k(\Dl)|]=\E[\sfD_{k,n}(\Dl)]-\sfD_k(\Dl)+2\E[(\sfD_k(\Dl)-\sfD_{k,n}(\Dl))_+]\longrightarrow0,
  \]
  which completes the proof.
\end{proof}

We can now conclude the proof of the weak convergence theorem.
\begin{proof}[Proof of Theorem~\ref{thm:LLN_PD}]
  Existence and uniqueness follow from Proposition~\ref{prop:vague_PD}, and finiteness follows from Proposition~\ref{prop:total_mass}.
  Every subsequence has a further subsequence along which both vague convergence and convergence of the total masses hold almost surely.
  Vague convergence together with convergence of total masses implies weak convergence along this further subsequence.
  The subsequence characterization of convergence in probability therefore gives weak convergence in probability.

  The mean measures converge vaguely by Proposition~\ref{prop:vague_PD}, and their total masses converge by Proposition~\ref{prop:total_mass}.
  Hence $\E[\sfD_{k,n}]$ converges weakly to $\sfD_k$.
\end{proof}

\subsection{Convergence for Unbounded Test Functions}
\label{ssec:LLN_test}
We combine the tail estimates in Lemma~\ref{lem:tails} with a truncation argument, as in~\cite[Section~4.3]{Me2025}.
The death-time estimate controls polynomial growth in the death coordinate, while the estimate near the diagonal controls the singularity allowed in~\eqref{eq:PD_test_growth}.

\begin{proof}[Proof of Theorem~\ref{thm:LLN_PD_test}]
  By Theorem~\ref{thm:LLN_PD}, the mean measures $\E[\sfD_{k,n}]$ converge weakly to $\sfD_k$.
  Applying the Portmanteau theorem to the open set $\{(x,y)\in\Dl\mid y>R\}$ and using~\eqref{eq:d_tail}, we obtain
  \[
  \sfD_k(\{(x,y)\in\Dl\mid y>R\})\le C_1R^{-1},
  \qquad R\ge1.
  \]
  Letting $R\to\infty$ gives $\sfD_k(\Dl\setminus\Dl_\fin)=0$.

  Choose $\eta\in(\gm,1)$ and set
  \[
  g(x,y):=(1+y)^{\a+1}\bigl\{1+(y-x)^{-\eta}\bigr\},
  \qquad (x,y)\in\Dl_\fin.
  \]
  We claim that
  \begin{equation}\label{eq:g_Pi_kn_Pi_k}
  \sup_{n\ge n_0}\E\left[\int_{\Dl_\fin}g\,d\sfD_{k,n}\right]<\infty
  \quad\text{and}\quad
  \int_{\Dl_\fin}g\,d\sfD_k<\infty.
  \end{equation}
  We first derive the conclusion assuming these bounds.
  Below, $C>0$ denotes a constant independent of $n$, $R$, and $h$ that may change from line to line.

  Since~\eqref{eq:PD_test_growth} implies $|f|\le Cg$, the bounds in~\eqref{eq:g_Pi_kn_Pi_k} ensure that $\int_{\Dl_\fin}|f|\,d\sfD_k<\infty$ and that $\int_{\Dl_\fin}|f|\,d\sfD_{k,n}$ is an integrable random variable for every $n\ge n_0$.
  Fix $R\ge1$ and $0<h\le1$, and choose a continuous function $\chi_{R,h}\colon\Dl\to[0,1]$ equal to $1$ on $\{y\le R,\,y-x\ge h\}$ and supported in $\{y\le2R,\,y-x\ge h/2\}$.
  Define $f_{R,h}:=f\chi_{R,h}$ on $\Dl_\fin$ and extend it by zero to $\Dl$.
  Then $f_{R,h}\in C_c(\Dl)$, so its integrals converge in probability by Theorem~\ref{thm:LLN_PD}.
  These integrals are uniformly integrable, since their absolute values are bounded by $\|f_{R,h}\|_\infty\sfD_{k,n}(\Dl)$ and the total masses converge in $L^1$ by Proposition~\ref{prop:total_mass}.
  Hence,
  \[
  \int_\Dl f_{R,h}\,d\sfD_{k,n}\xrightarrow{L^1}\int_\Dl f_{R,h}\,d\sfD_k.
  \]
  Moreover,~\eqref{eq:PD_test_growth} and the choice of $g$ give
  \[
  |f-f_{R,h}|
  \le |f|\bigl(\1_{\{y>R\}}+\1_{\{y-x<h\}}\bigr)
  \le C\bigl(R^{-1}+h^{\eta-\gm}\bigr)g.
  \]
  Here, on $\{y>R\}$ we use $(1+y)^{-1}\le R^{-1}$, while on $\{y-x<h\}$ we use $(y-x)^{-\gm}\le h^{\eta-\gm}(y-x)^{-\eta}$, together with $h\le1$ and $\eta>\gm$.
  The triangle inequality gives
  \begin{align*}
  &\E\left[\left|\int_{\Dl_\fin}f\,d\sfD_{k,n}-\int_{\Dl_\fin}f\,d\sfD_k\right|\right]\\
  &\quad\le\E\left[\int_{\Dl_\fin}|f-f_{R,h}|\,d\sfD_{k,n}\right]
  +\E\left[\left|\int_\Dl f_{R,h}\,d\sfD_{k,n}-\int_\Dl f_{R,h}\,d\sfD_k\right|\right]\\
  &\qquad+\int_{\Dl_\fin}|f-f_{R,h}|\,d\sfD_k.
  \end{align*}
  The middle term tends to zero by the preceding $L^1$ convergence.
  The first and last terms satisfy
  \begin{align*}
  &\E\left[\int_{\Dl_\fin}|f-f_{R,h}|\,d\sfD_{k,n}\right]
  +\int_{\Dl_\fin}|f-f_{R,h}|\,d\sfD_k\\
  &\quad\le C\bigl(R^{-1}+h^{\eta-\gm}\bigr)
  \left(\E\left[\int_{\Dl_\fin}g\,d\sfD_{k,n}\right]+\int_{\Dl_\fin}g\,d\sfD_k\right)\\
  &\quad\le C\bigl(R^{-1}+h^{\eta-\gm}\bigr),
  \end{align*}
  where the last inequality follows from~\eqref{eq:g_Pi_kn_Pi_k}.
  Consequently,
  \[
  \limsup_{n\to\infty}\E\left[\left|\int_{\Dl_\fin}f\,d\sfD_{k,n}-\int_{\Dl_\fin}f\,d\sfD_k\right|\right]
  \le C\bigl(R^{-1}+h^{\eta-\gm}\bigr).
  \]
  Letting $R\to\infty$ and $h\downarrow0$ proves~\eqref{eq:LLN_PD_test}, assuming~\eqref{eq:g_Pi_kn_Pi_k}.

  It remains to prove~\eqref{eq:g_Pi_kn_Pi_k}.
  We combine the death-tail and diagonal-tail bounds by taking the minimum of the two estimates on each dyadic death-time block.
  Choose $m\in\N$ such that $m(1-\eta)>\a+1+\eta(k+1)$, and partition $\Dl_\fin$ into
  \[
  \Dl_j:=\{(x,y)\in\Dl\mid2^j-1\le y<2^{j+1}-1\},
  \qquad j\in\N_0.
  \]
  For $j\ge2$, apply~\eqref{eq:d_tail} using $\Dl_j\subset\{y>2^{j-1}\}$.
  For $j=0,1$, use $\E[\sfD_{k,n}(\Dl_j)]\le\E[\sfD_{k,n}(\Dl)]$, whose right-hand side is uniformly bounded by Proposition~\ref{prop:total_mass}.
  Together with~\eqref{eq:diagonal_tail} applied with $R=2^{j+1}-1$, these estimates give, after increasing $C_m$,
  \[
  \E[\sfD_{k,n}(\Dl_j)]\le C_m2^{-jm}
  \quad\text{and}\quad
  \E[\sfD_{k,n}(\Dl_j\cap\{y-x\le h\})]\le C_m\min\{2^{j(k+1)}h,2^{-jm}\}
  \]
  for all $j\in\N_0$, $n\ge n_0$, and $0<h\le1$.

  Using $z^{-\eta}=\eta\int_0^\infty h^{-\eta-1}\1_{\{h\ge z\}}\,dh$ for $z>0$, Tonelli's theorem gives
  \begin{align*}
  &\E\left[\int_{\Dl_j}\bigl\{1+(y-x)^{-\eta}\bigr\}\,d\sfD_{k,n}(x,y)\right]\\
  &\quad=\E[\sfD_{k,n}(\Dl_j)]
  +\eta\int_0^\infty h^{-\eta-1}\E[\sfD_{k,n}(\Dl_j\cap\{y-x\le h\})]\,dh\\
  &\quad\le2C_m2^{-jm}
  +C_m\eta\int_0^1h^{-\eta-1}\min\{2^{j(k+1)}h,2^{-jm}\}\,dh.
  \end{align*}
  To bound the integral over $[1,\infty)$, we use $\E[\sfD_{k,n}(\Dl_j\cap\{y-x\le h\})]\le\E[\sfD_{k,n}(\Dl_j)]$ and $\eta\int_1^\infty h^{-\eta-1}\,dh=1$.
  Together with the initial term $\E[\sfD_{k,n}(\Dl_j)]$, this contributes at most $2C_m2^{-jm}$.

  To evaluate the remaining integral, set $h_j:=2^{-j(m+k+1)}$, so that $2^{j(k+1)}h_j=2^{-jm}$.
  Splitting the integral at $h_j$, we obtain
  \begin{align*}
  &\eta\int_0^1h^{-\eta-1}\min\{2^{j(k+1)}h,2^{-jm}\}\,dh\\
  &\quad=\eta2^{j(k+1)}\int_0^{h_j}h^{-\eta}\,dh
  +\eta2^{-jm}\int_{h_j}^1h^{-\eta-1}\,dh\\
  &\quad=\frac{\eta}{1-\eta}2^{j(k+1)}h_j^{1-\eta}
  +2^{-jm}(h_j^{-\eta}-1)\\
  &\quad=\frac1{1-\eta}2^{j\{\eta(k+1)-m(1-\eta)\}}-2^{-jm}\\
  &\quad\le\frac1{1-\eta}2^{j\{\eta(k+1)-m(1-\eta)\}}.
  \end{align*}
  Since $2^{-jm}\le2^{j\{\eta(k+1)-m(1-\eta)\}}$, it follows that
  \[
  \E\left[\int_{\Dl_j}\bigl\{1+(y-x)^{-\eta}\bigr\}\,d\sfD_{k,n}(x,y)\right]
  \le\left(2+\frac1{1-\eta}\right)C_m2^{j\{\eta(k+1)-m(1-\eta)\}}.
  \]
  Multiplying by the bound $(1+y)^{\a+1}\le2^{(j+1)(\a+1)}$ on $\Dl_j$ and summing over $j$ yields
  \[
  \sup_{n\ge n_0}\E\left[\int_{\Dl_\fin}g\,d\sfD_{k,n}\right]
  \le C\sum_{j=0}^\infty2^{j\{\a+1+\eta(k+1)-m(1-\eta)\}}
  <\infty
  \]
  by the choice of $m$.

  Finally, extending $g$ by zero to $\Dl\setminus\Dl_\fin$ gives a nonnegative lower semicontinuous function.
  The weak convergence of the mean measures and the Portmanteau theorem therefore imply
  \[
  \int_{\Dl_\fin}g\,d\sfD_k
  \le\liminf_{n\to\infty}\E\left[\int_{\Dl_\fin}g\,d\sfD_{k,n}\right]
  <\infty.
  \]
  This proves~\eqref{eq:g_Pi_kn_Pi_k} and completes the proof.
\end{proof}

We now combine the convergence of persistence diagrams with the convergence of total persistence, obtained from Theorem~\ref{thm:LLN_PD_test} in Corollary~\ref{cor:LLN_Pers}, to prove convergence in the optimal transport distance.
The proof is given in Section~S3.4 of the Supplementary Material~\cite{HKsupp}.

\begin{appendix}
\section{Tail Estimates for Persistence Diagrams}
\label{app:PD_tails}
Throughout this appendix, fix $k\in\N_0$.
We first establish a deterministic estimate for use in the proof of Lemma~\ref{lem:tails}.

For a finite right-continuous filtration $\cX=\{X(t)\}_{t\ge0}$, we define the \emph{appearance time} of a simplex $\sg$ by
\[
\kappa^\cX(\sg):=\inf\{t\ge0\mid\sg\in X(t)\},
\]
with $\inf\emptyset:=\infty$.
For $0\le s<t<\infty$, set
\[
\Dl_{s,t}:=\{(b,d)\in\Dl\mid s<b<d\le t\}.
\]
The following lemma bounds the number of birth--death pairs in $\Dl_{s,t}$ by the number of incident pairs of simplices appearing at distinct times in $(s,t]$.
\begin{lem}\label{lem:bars_in_st}
  Let $\cX=\{X(t)\}_{t\ge0}$ be a finite right-continuous filtration.
  Then, for every $0\le s<t<\infty$,
  \begin{equation}\label{eq:bars_in_st}
  \Dgm_k^\cX(\Dl_{s,t})
  \le\bigl|\{(\sg,\tau)\mid\sg\subset\tau,\,\dim\sg=k,\,\dim\tau=k+1,\,s<\kappa^\cX(\sg)<\kappa^\cX(\tau)\le t\}\bigr|.
  \end{equation}
\end{lem}
\begin{proof}
  If no simplex appears in $(s,t]$, the claim is immediate.
  Let $A$ be the submatrix of the boundary matrix of $X(t)$ whose rows and columns correspond to the $k$-simplices and $(k+1)$-simplices appearing in $(s,t]$, respectively, ordered by appearance time.
  Since $A$ represents $\del_{k+1}^{X(t)\setminus X(s)}$, applying~\eqref{eq:PB} and~\eqref{eq:CK_BX} to $X(s)\subset X(t)$ gives
  \begin{align*}
  \Dgm_k^\cX([0,s]\times(s,t])
  &=\b_k^\cX(s,s)-\b_k^\cX(s,t)\\
  &=\dim B_k(X(t))-\dim B_k(X(s))-\rank A.
  \end{align*}

  Let $u_1<u_2<\cdots<u_m$ be the appearance times in $(s,t]$.
  For each such time $u$, write $X(u-):=\bigcup_{v<u}X(v)$ and let $A_{u,u}$ be the block corresponding to rows and columns appearing at $u$.
  Since $A_{u,u}$ represents $\del_{k+1}^{X(u)\setminus X(u-)}$, the same identities applied to $X(u-)\subset X(u)$ yield
  \begin{align*}
  \Dgm_k^\cX([0,u)\times\{u\})
  &=\b_k(X(u-))-\b_k\bigl(X(u-)\subset X(u)\bigr)\\
  &=\dim B_k(X(u))-\dim B_k(X(u-))-\rank A_{u,u}.
  \end{align*}
  Summing this identity over the appearance times and subtracting the preceding expression for $\Dgm_k^\cX([0,s]\times(s,t])$, we obtain
  \[
  \Dgm_k^\cX(\Dl_{s,t})
  =\rank A-\sum_{i=1}^m\rank A_{u_i,u_i}.
  \]

  Since every face appears no later than the simplex containing it, $A$ has the block form
  \[
  A=
  \begin{pmatrix}
    A_{u_1,u_1} & * & \cdots & *\\
    0 & A_{u_2,u_2} & \cdots & *\\
    \vdots & \ddots & \ddots & \vdots\\
    0 & \cdots & 0 & A_{u_m,u_m}
  \end{pmatrix}.
  \]
  Set $D:=\operatorname{diag}(A_{u_1,u_1},\ldots,A_{u_m,u_m})$.
  The subadditivity of rank gives
  \[
  \rank A-\sum_{i=1}^m\rank A_{u_i,u_i}
  =\rank A-\rank D
  \le\rank(A-D).
  \]
  The last rank is clearly at most the number of nonzero entries of $A-D$, which correspond precisely to the incident pairs counted on the right-hand side of~\eqref{eq:bars_in_st}.
\end{proof}

We now apply this estimate to the time-changed multiparameter complex process.
\begin{proof}[Proof of Lemma~\ref{lem:tails}]
  Fix $\dl>0$ with $2\dl<r_k(t_0)$.
  By Condition~\textup{\ref{cond:C1}}, choose $n_0$ sufficiently large that $n\dl\ge8$ and the normalization $\tbinom n{k+1}q_k(r_k^\rightarrow(1/n))$ is at least $1$ for every $n\ge n_0$.
  All constants in the estimates below are independent of $n\ge n_0$.
  
  \emph{Preliminary estimates.}
  For $0<x\le2n\dl$, set
  \[
  \Theta_{k,n}(x):=\frac{q_k\bigl(r_k^\rightarrow(x/n)\bigr)}{q_k\bigl(r_k^\rightarrow(1/n)\bigr)}.
  \]
  For $x\le1$, monotonicity gives $\Theta_{k,n}(x)\le1$.

  For $x\ge1$, each ratio $p_j(r_k^\rightarrow(x/n))/p_j(r_k^\rightarrow(1/n))$ is at least $1$.
  The product formulas~\eqref{eq:qkt-rkt} and the inequalities $\binom{k+1}{j+1}\le(k+1)\binom{k+1}j$ therefore give
  \[
  \Theta_{k,n}(x)
  \le\prod_{j=0}^{k+1}
  \left(\frac{p_j(r_k^\rightarrow(x/n))}{p_j(r_k^\rightarrow(1/n))}\right)^{(k+1)\binom{k+1}j}
  =\left(\frac{r_k(r_k^\rightarrow(x/n))}{r_k(r_k^\rightarrow(1/n))}\right)^{k+1}
  =x^{k+1},
  \]
  where the last equality follows from Lemma~\ref{lem:theta_Theta}(1).
  Hence
  \begin{equation}\label{eq:Theta_1vx}
  \Theta_{k,n}(x)\le(1\vee x)^{k+1},
  \qquad 0<x\le2n\dl.
  \end{equation}

  By~\cite[Theorem~3.6]{HK2019}, for every $\ell\in\N$ there exists $C_{k,\ell}>0$, depending only on $k$ and $\ell$, such that
  \begin{equation}\label{eq:HK_bound}
  \E[\b_k(X_n(t))]
  \le\binom n{k+1}q_k(t)\min\bigl\{1,C_{k,\ell}(nr_k(t))^{-\ell}\bigr\},
  \qquad t\ge0.
  \end{equation}
  Here, we use the convention $0^{-\ell}:=\infty$.
  Together with~\eqref{eq:Theta_1vx} and Lemma~\ref{lem:theta_Theta}(1), this yields
  \begin{equation}\label{eq:Betti_bound}
  \frac{\E[\b_k(\tilde X_n(x))]}{\binom n{k+1}q_k\bigl(r_k^\rightarrow(1/n)\bigr)}
  \le C_{k,\ell} x^{k+1-\ell},
  \qquad 1\le x\le2n\dl.
  \end{equation}

  For a fixed incident pair $\sg\subset\tau$ of dimensions $k$ and $k+1$, define
  \[
  U_\sg:=\max_{\emptyset\ne\rho\subset\sg}w_\rho,
  \qquad
  W_{\sg,\tau}:=\max_{\substack{\emptyset\ne\rho\subset\tau\colon\\\rho\not\subset\sg}}w_\rho.
  \]
  Here, $w_\rho$ is the random weight appearing in~\eqref{eq:def_MCP}.
  These variables are independent, with distribution functions $q_k$ and $r_k$, respectively, and the appearance times of $\sg$ and $\tau$ in $\cX_n$ are $U_\sg$ and $\max\{U_\sg,W_{\sg,\tau}\}$.
  For $0<x<y\le2n\dl$, put $s_n=r_k^\rightarrow(x/n)$ and $t_n=r_k^\rightarrow(y/n)$.
  Distinct appearance times in $(x,y]$ in $\tilde\cX_n$ require $s_n<U_\sg<W_{\sg,\tau}\le t_n$, an event of probability at most
  \[
  \bigl(q_k(t_n)-q_k(s_n)\bigr)\bigl(r_k(t_n)-r_k(s_n)\bigr).
  \]
  Since there are $(k+2)\binom n{k+2}$ incident pairs and $r_k(t_n)-r_k(s_n)=(y-x)/n$, Lemma~\ref{lem:bars_in_st} gives
  \begin{equation}\label{eq:bars_in_xy}
  \E[\sfD_{k,n}(\Dl_{x,y})]
  \le(y-x)\bigl(\Theta_{k,n}(y)-\Theta_{k,n}(x)\bigr),
  \qquad 0<x<y\le2n\dl,
  \end{equation}
  where we used $(k+2)\binom n{k+2}/(n\binom n{k+1})\le1$.
    
  \smallskip
  \emph{Proof of~\eqref{eq:d_tail}.}
  Fix $m\in\N$ and take $\ell=2k+2m+3$.
  We claim that there exists a constant $C_m>0$, independent of $n\ge n_0$, such that
  \begin{equation}\label{eq:b_tail}
  \E[\sfD_{k,n}(\{(x,y)\in\Dl\mid x>R\})]
  \le C_mR^{-m}
  \quad\text{for every $1\le R\le n\dl$.}
  \end{equation}
  We first deduce~\eqref{eq:d_tail} from this estimate.
  For $1\le R\le n\dl$, the quadrant identity~\eqref{eq:PD-PB} and~\eqref{eq:Betti_bound} give
  \begin{align*}
  \E[\sfD_{k,n}(\{(x,y)\in\Dl\mid y>R\})]
  &=\E[\sfD_{k,n}(\{(x,y)\in\Dl\mid x>R\})]
  +\frac{\E[\b_k(\tilde X_n(R))]}{\binom n{k+1}q_k\bigl(r_k^\rightarrow(1/n)\bigr)}\\
  &\le C_mR^{-m}+C_{k,\ell}R^{k+1-\ell}.
  \end{align*}
  Since $k+1-\ell=-k-2m-2\le-m$ and $R\ge1$, this proves~\eqref{eq:d_tail} for $1\le R\le n\dl$ after adjusting $C_m$.
  For $n\dl<R<n$, monotonicity and the estimate just obtained at $R=n\dl$ yield
  \[
  \E[\sfD_{k,n}(\{(x,y)\in\Dl\mid y>R\})]
  \le C_m(n\dl)^{-m}
  \le C_m\dl^{-m}R^{-m}.
  \]
  Since $\dl$ is fixed, the factor $\dl^{-m}$ can be absorbed into $C_m$.
  For $R\ge n$, the mass is zero because $\tilde X_n(n)=K(n)$ and we use reduced homology.
  Thus,~\eqref{eq:d_tail} follows from~\eqref{eq:b_tail}.

  It remains to prove~\eqref{eq:b_tail}.
  Fix $1\le R\le n\dl$.
  We control points born in $(R,n\dl]$ and after $n\dl$ separately.
  
  We first control points born in $(R,n\dl]$.
  For any $1\le L<n\dl$, partition $[L,\min\{2L,n\dl\}]$ into $N=\lceil L^{k+m+2}\rceil$ equal intervals with endpoints $x_0<\cdots<x_N$.
  Each interval has length at most $L^{-k-m-1}$.
  A birth--death pair with birth time in $(x_{i-1},x_i]$ either lies in $\Dl_{x_{i-1},x_i}$ or is still alive at $x_i$ and hence counted by $\b_k(\tilde X_n(x_i))$.
  Thus,~\eqref{eq:Theta_1vx}--\eqref{eq:bars_in_xy} give
  \begin{align*}
  &\E[\sfD_{k,n}(\{(x,y)\in\Dl\mid L<x\le\min\{2L,n\dl\}\})]\\
  &\quad\le\sum_{i=1}^N\Bigl\{(x_i-x_{i-1})\bigl(\Theta_{k,n}(x_i)-\Theta_{k,n}(x_{i-1})\bigr)+C_{k,\ell}x_i^{k+1-\ell}\Bigr\}\\
  &\quad\le L^{-k-m-1}\bigl(\Theta_{k,n}(x_N)-\Theta_{k,n}(x_0)\bigr)+C_{k,\ell}NL^{k+1-\ell}\\
  &\quad\le(2^{k+1}+2C_{k,\ell})L^{-m}.
  \end{align*}
  Here, we used $x_N\le2L$, $N\le2L^{k+m+2}$, and $\ell=2k+2m+3$.
  For $1\le R\le n\dl$, the intervals $(2^aR,\min\{2^{a+1}R,n\dl\}]$, indexed by $a\in\N_0$ satisfying $2^aR<n\dl$, form a disjoint partition of $(R,n\dl]$.
  Applying the preceding estimate with $L=2^aR$ and summing over these indices, we obtain
  \begin{equation}\label{eq:b_tail_early}
  \begin{aligned}
  \E[\sfD_{k,n}(\{(x,y)\in\Dl\mid R<x\le n\dl\})]
  &\le(2^{k+1}+2C_{k,\ell})R^{-m}\sum_{a=0}^{\infty}2^{-am}\\
  &\le2(2^{k+1}+2C_{k,\ell})R^{-m}.
  \end{aligned}
  \end{equation}
  This controls points born in $(R,n\dl]$.
  
  We next control points born after $n\dl$ by comparing them with those in the original filtration:
  \begin{equation}\label{eq:Pi_kn-Dgm}
  \sfD_{k,n}(\{(x,y)\in\Dl\mid x>n\dl\})
  \le\frac{\Dgm_k^{\cX_n}(\{(b,d)\in\Dl\mid b>r_k^\rightarrow(\dl)\})}{\binom n{k+1}q_k\bigl(r_k^\rightarrow(1/n)\bigr)}.
  \end{equation}
  Indeed, applying the time change to the interval decomposition replaces each interval $[b,d)$ by
  \[
  \{x\ge0\mid b\le r_k^\rightarrow(x/n)<d\}.
  \]
  Since $r_k^\rightarrow$ is nondecreasing, whenever this set is nonempty, its left endpoint, and hence the birth coordinate of the resulting point, is
  \[
  \inf\{x\ge0\mid b\le r_k^\rightarrow(x/n)\}.
  \]
  If $b\le r_k^\rightarrow(\dl)$, this coordinate is at most $n\dl$.
  Thus, only original pairs born after $r_k^\rightarrow(\dl)$ can contribute points born after $n\dl$, proving the inequality.

  We estimate the numerator of~\eqref{eq:Pi_kn-Dgm} using the independent simplex weights $\{w_\rho\}$ defining the original filtration $\cX_n$.
  Every original birth time greater than $r_k^\rightarrow(\dl)$ is the appearance time of a $k$-simplex and hence equals a finite weight $w_\rho$ of some simplex $\rho$ of dimension at most $k$.
  All classes with positive lifetime born at that time are alive in $X_n(w_\rho)$, so their number is at most $\b_k(X_n(w_\rho))$.
  Consequently,
  \[
  \E[\Dgm_k^{\cX_n}(\{(b,d)\in\Dl\mid b>r_k^\rightarrow(\dl)\})]
  \le\sum_{j=0}^k\ \sum_{\rho\in F_j(K(n))}\E\bigl[\1_{\{r_k^\rightarrow(\dl)<w_\rho<\infty\}}\b_k(X_n(w_\rho))\bigr].
  \]
  Fix $0\le j\le k$ and a $j$-simplex $\rho\in K(n)$.
  Conditioning first on all weights other than $w_\rho$ and then integrating with respect to its distribution, we obtain
  \begin{align*}
  \E\bigl[\1_{\{r_k^\rightarrow(\dl)<w_\rho<\infty\}}\b_k(X_n(w_\rho))\bigr]
  &=\int_{(r_k^\rightarrow(\dl),\infty)}\frac{\E[\b_k(X_n(t))\1_{\{w_\rho\le t\}}]}{p_j(t)}\,dp_j(t)\\
  &\le\int_{(r_k^\rightarrow(\dl),\infty)}\frac{\E[\b_k(X_n(t))]}{p_j(t)}\,dp_j(t)\\
  &\le C_{k,\ell}n^{k+1-\ell}\int_{(r_k^\rightarrow(\dl),\infty)}\frac{q_k(t)}{p_j(t)r_k(t)^\ell}\,dp_j(t)\\
  &\le\frac{C_{k,\ell}}{p_j(r_k^\rightarrow(\dl))\dl^\ell}\,n^{k+1-\ell}.
  \end{align*}
  Here, $dp_j$ denotes the Lebesgue--Stieltjes probability measure associated with $p_j$.
  For the first equality, note that $X_n(t)$ depends on $w_\rho$ only through $\1_{\{w_\rho\le t\}}$, which equals $1$ at $t=w_\rho$.
  The penultimate inequality follows from~\eqref{eq:HK_bound} and $\binom n{k+1}\le n^{k+1}$, while the last uses $q_k(t)\le1$, $r_k(t)\ge\dl$, and $p_j(t)\ge p_j(r_k^\rightarrow(\dl))>0$ on the integration domain.

  Combining the above estimates yields
  \[
  \E[\Dgm_k^{\cX_n}(\{(b,d)\in\Dl\mid b>r_k^\rightarrow(\dl)\})]
  \le C_mn^{2k+2-\ell}
  \]
  for some constant $C_m>0$ independent of $n$ and $R$.
  Since $\ell=2k+2m+3$ and $\binom n{k+1}q_k(r_k^\rightarrow(1/n))\ge1$ by the choice of $n_0$, it follows from~\eqref{eq:Pi_kn-Dgm} that
  \[
  \E[\sfD_{k,n}(\{(x,y)\in\Dl\mid x>n\dl\})]
  \le C_mn^{-2m-1}
  \le C_mR^{-m},
  \]
  where we used $R\le n\dl<n$.
  Together with~\eqref{eq:b_tail_early}, this proves~\eqref{eq:b_tail} after increasing $C_m$ if necessary.

  \smallskip
  \emph{Proof of~\eqref{eq:diagonal_tail}.}
  Fix $R\ge1$ and $0<h\le1$.
  Among the points satisfying $y\le R$ and $y-x\le h$, we consider those with $y\le n\dl$ and those with $y>n\dl$ separately.
  The latter case occurs only when $R>n\dl$.
  
  We first consider the former case and set $S:=R\wedge n\dl$.
  Every point with $y\le S$ and $y-x\le h$ either has $y\le2h$ or belongs to $\Dl_{ih,(i+2)h}$ for some $i\in\N$ with $ih<S$.
  To control the first set of points, observe that the number of birth--death pairs with death time at most $\eps$ is bounded by the number of $(k+1)$-simplices in $\tilde X_n(\eps)$.
  Thus, for $0<\eps\le2$,
  \begin{equation}\label{eq:small_death}
  \begin{aligned}
  \E[\sfD_{k,n}(\{(x,y)\in\Dl\mid y\le\eps\})]
  &\le\frac{\E[f_{k+1}(\tilde X_n(\eps))]}{\binom n{k+1}q_k(r_k^\rightarrow(1/n))}\\
  &=\frac{\binom n{k+2}q_{k+1}(r_k^\rightarrow(\eps/n))}{\binom n{k+1}q_k(r_k^\rightarrow(1/n))}\\
  &=\frac{n-k-1}{n(k+2)}\,\eps\Theta_{k,n}(\eps)
  \le\frac{\eps}{k+2}(1\vee\eps)^{k+1}.
  \end{aligned}
  \end{equation}
  Here, we used $q_{k+1}=q_kr_k$, $r_k(r_k^\rightarrow(\eps/n))=\eps/n$, and~\eqref{eq:Theta_1vx}.
  Since $S+2h\le n\dl+2\le2n\dl$, applying~\eqref{eq:small_death} with $\eps=2h$ and~\eqref{eq:bars_in_xy} to the remaining sets of points gives
  \begin{equation}\label{eq:below_S}
  \begin{aligned}
  \E[\sfD_{k,n}(\{(x,y)\in\Dl\mid y\le S,\,y-x\le h\})]
  &\le Ch+2h\sum_{\substack{i\in\N\colon\\ih<S}}
  \bigl(\Theta_{k,n}((i+2)h)-\Theta_{k,n}(ih)\bigr)\\
  &\le Ch+4h\Theta_{k,n}(S+2h)\\
  &\le Ch(1+R)^{k+1}.
  \end{aligned}
  \end{equation}
  Here, $C>0$ denotes a constant independent of $n$, $R$, and $h$.
  The second inequality follows by telescoping, and the last follows from~\eqref{eq:Theta_1vx} and $S+2h\le R+2$.
  
  We next consider the latter case: points satisfying $n\dl<y\le R$ and $y-x\le h$.
  Cover these points by the sets $\Dl_{ih,(i+2)h}$ with $i\in\N$ and $(i+2)h>n\dl$, and apply Lemma~\ref{lem:bars_in_st}.
  Each incident pair is counted at most twice.
  Every counted pair has distinct appearance times at most $2h$ apart, with the earlier time greater than $n\dl-2h$.
  We claim that, for every incident pair $\sg\subset\tau$ with $\dim\sg=k$ and $\dim\tau=k+1$,
  \begin{equation}\label{eq:prob_incident}
  \P\left(
  0<\kappa^{\tilde\cX_n}(\tau)-\kappa^{\tilde\cX_n}(\sg)\le2h,\,
  \kappa^{\tilde\cX_n}(\sg)>n\dl-2h
  \right)
  \le C\frac hn,
  \end{equation}
  with $C>0$ independent of $n$, $h$, and the pair.

  Once~\eqref{eq:prob_incident} is proved, summing over the $(k+2)\binom n{k+2}$ incident pairs, each counted at most twice, gives
  \[
  \E[\sfD_{k,n}(\{(x,y)\in\Dl\mid n\dl<y\le R,\,y-x\le h\})]
  \le Chn^{k+1}
  \le ChR^{k+1}.
  \]
  Here, we used $\binom n{k+1}q_k(r_k^\rightarrow(1/n))\ge1$ by the choice of $n_0$ and $R>n\dl$, increasing $C$ if necessary.
  Together with~\eqref{eq:below_S}, this establishes~\eqref{eq:diagonal_tail}.

  It remains to prove~\eqref{eq:prob_incident}.
  We first express the appearance times in terms of the original weights, allowing possible jumps of $r_k$.
  For $t\in(0,\infty)$, write $r_k(t-):=\lim_{s\uparrow t}r_k(s)$, and set $r_k(0-):=0$ and $r_k(\infty-):=\lim_{t\to\infty}r_k(t)$.
  Since
  \[
  t\le r_k^\rightarrow(x/n)
  \quad\Longleftrightarrow\quad
  r_k(t-)\le x/n
  \]
  for every $t\in[0,\infty]$ and $x\ge0$, we obtain
  \begin{equation}\label{eq:relate_kappa}
  \kappa^{\tilde\cX_n}(\sg)=nr_k(\kappa^{\cX_n}(\sg)-)
  \end{equation}
  for every $\sg\in K(n)$.
  The left limit is essential here because $r_k$ need not be continuous away from the origin.
  
  Fix an incident pair $\sg\subset\tau$, and recall that $U_\sg$ and $W_{\sg,\tau}$ are independent, with distribution functions $q_k$ and $r_k$, respectively.
  On the event in~\eqref{eq:prob_incident}, the distinct appearance times imply $U_\sg<W_{\sg,\tau}$.
  By~\eqref{eq:relate_kappa}, this event therefore implies
  \[
  0<r_k(W_{\sg,\tau}-)-r_k(U_\sg-)\le\frac{2h}n,
  \qquad
  r_k(U_\sg-)>\dl-\frac{2h}n.
  \]
  In particular,
  \[
  r_k(W_{\sg,\tau}-)-\frac{2h}n
  >r_k(U_\sg-)-\frac{2h}n
  >\dl-\frac{4h}n
  \ge\frac\dl2,
  \]
  since $n\dl\ge8$ and $h\le1$.

  We estimate the probability in~\eqref{eq:prob_incident} by comparing the Lebesgue--Stieltjes measures associated with $q_k$ and $r_k$.
  For $r_k^\rightarrow(\dl/4)<s<t\le\infty$, we have $r_k(t)\ge r_k(s)\ge\dl/4$, so each ratio $p_j(s)/p_j(t)$ lies in $(0,1]$.
  Using again the product formulas~\eqref{eq:qkt-rkt} and $\binom{k+1}{j+1}\le(k+1)\binom{k+1}j$, we obtain
  \[
  \frac{q_k(s)}{q_k(t)}
  \ge\prod_{j=0}^{k+1}\left(\frac{p_j(s)}{p_j(t)}\right)^{(k+1)\binom{k+1}j}
  =\left(\frac{r_k(s)}{r_k(t)}\right)^{k+1}.
  \]
  Consequently,
  \begin{align*}
  q_k(t)-q_k(s)
  =q_k(t)\left(1-\frac{q_k(s)}{q_k(t)}\right)
  &\le1-\left(\frac{r_k(s)}{r_k(t)}\right)^{k+1}\\
  &\le(k+1)\frac{r_k(t)-r_k(s)}{r_k(t)}
  \le\frac{4(k+1)}\dl\bigl(r_k(t)-r_k(s)\bigr).
  \end{align*}
  Here, we used $q_k(t)\le1$ and $1-u^{k+1}\le(k+1)(1-u)$ for $u\in[0,1]$.
  Therefore, the corresponding Lebesgue--Stieltjes measures satisfy
  \begin{equation}\label{eq:LS_qk-rk}
  dq_k\le\frac{4(k+1)}\dl\,dr_k
  \quad\text{on $(r_k^\rightarrow(\dl/4),\infty]$.}
  \end{equation}

  We now fix a value $t$ of $W_{\sg,\tau}$ satisfying
  \[
  r_k(t-)-\frac{2h}n\ge\frac\dl2.
  \]
  When $W_{\sg,\tau}=t$, the event in~\eqref{eq:prob_incident} requires
  \[
  U_\sg<t,
  \qquad
  r_k(t-)-\frac{2h}n\le r_k(U_\sg-)<r_k(t-).
  \]
  In particular, $U_\sg>r_k^\rightarrow(\dl/4)$, since otherwise monotonicity would give
  \[
  r_k(U_\sg-)\le r_k(r_k^\rightarrow(\dl/4))=\dl/4.
  \]
  Set
  \[
  I_t:=\left\{s\in[0,t)\relmiddle|r_k(t-)-\frac{2h}n\le r_k(s-)<r_k(t-)\right\}.
  \]
  Thus, applying~\eqref{eq:LS_qk-rk} on $I_t$ yields
  \begin{equation}\label{eq:prob_given_W}
  \begin{aligned}
  \P\biggl(U_\sg<t,\,r_k(t-)-\frac{2h}n\le r_k(U_\sg-)<r_k(t-)\biggr)
  &=\int_{I_t}dq_k(s)\\
  &\le\frac{4(k+1)}\dl\int_{I_t}dr_k(s)\\
  &\le\frac{8(k+1)}\dl\frac hn.
  \end{aligned}
  \end{equation}
  To justify the last inequality, note that $I_t$ is an interval, since $s\mapsto r_k(s-)$ is nondecreasing.
  If $I_t\ne\emptyset$, then
  \[
  \int_{I_t}dr_k(s)
  =\sup_{\substack{s_1,s_2\in I_t\colon\\s_1\le s_2}}
  \bigl(r_k(s_2)-r_k(s_1-)\bigr)
  \le r_k(t-)-\left(r_k(t-)-\frac{2h}n\right)
  =\frac{2h}n.
  \]
  The equality follows by exhausting $I_t$ by closed intervals contained in it, and the inequality uses $r_k(s_1-)\ge r_k(t-)-2h/n$ and $r_k(s_2)\le r_k(t-)$.
  The bound is immediate if $I_t=\emptyset$.

  Values of $W_{\sg,\tau}$ with $r_k(W_{\sg,\tau}-)-2h/n<\dl/2$ cannot contribute to the event in~\eqref{eq:prob_incident}.
  By independence, integrating~\eqref{eq:prob_given_W} with respect to the distribution of $W_{\sg,\tau}$ therefore proves~\eqref{eq:prob_incident}.
\end{proof}

\end{appendix}

\begin{funding}
MH is partially supported by JSPS KAKENHI Grant Number 25K07056.
SK is partially supported by EPSRC Research Grant EP/Y008642/1.
\end{funding}

\begin{supplement}
\stitle{Supplement to ``Limit Theorems for Persistent Homology of Multiparameter Complex Processes''}
\sdescription{Essential self-adjointness, simplifications of the limiting constant, and auxiliary proofs for convergence of persistence diagrams.}
\end{supplement}

\bibliographystyle{spr-ims-number}
\bibliography{references}

@article {AL2007,
    AUTHOR = {Aldous, David and Lyons, Russell},
     TITLE = {Processes on unimodular random networks},
   JOURNAL = {Electron. J. Probab.},
  FJOURNAL = {Electronic Journal of Probability},
    VOLUME = {12},
      YEAR = {2007},
     PAGES = {1454--1508},
      NOTE = {No. 54},
      ISSN = {1083-6489},
   MRCLASS = {60C05 (05C80 60G50)},
  MRNUMBER = {2354165},
MRREVIEWER = {Jean-Fran\c cois\ Delmas},
       DOI = {10.1214/EJP.v12-463},
       URL = {https://doi.org/10.1214/EJP.v12-463},
}

@incollection {AS2004,
    AUTHOR = {Aldous, David and Steele, J. Michael},
     TITLE = {The objective method: probabilistic combinatorial optimization
              and local weak convergence},
 BOOKTITLE = {Probability on discrete structures},
    SERIES = {Encyclopaedia Math. Sci.},
    VOLUME = {110},
     PAGES = {1--72},
 PUBLISHER = {Springer, Berlin},
      YEAR = {2004},
      ISBN = {3-540-00845-4},
   MRCLASS = {60C05 (05C80 60F05)},
  MRNUMBER = {2023650},
MRREVIEWER = {Ralph\ Neininger},
       DOI = {10.1007/978-3-662-09444-0\_1},
}

@article {ALLM2013,
    AUTHOR = {Aronshtam, Lior and Linial, Nathan and {\polishL}uczak, Tomasz and
              Meshulam, Roy},
     TITLE = {Collapsibility and vanishing of top homology in random
              simplicial complexes},
   JOURNAL = {Discrete Comput. Geom.},
  FJOURNAL = {Discrete \& Computational Geometry. An International Journal
              of Mathematics and Computer Science},
    VOLUME = {49},
      YEAR = {2013},
    NUMBER = {2},
     PAGES = {317--334},
      ISSN = {0179-5376,1432-0444},
   MRCLASS = {55U10 (05C80 60D05)},
  MRNUMBER = {3017914},
MRREVIEWER = {Michael\ S.\ Farber},
       DOI = {10.1007/s00454-012-9483-8},
       URL = {https://doi.org/10.1007/s00454-012-9483-8},
}

@article {BS2001,
    AUTHOR = {Benjamini, Itai and Schramm, Oded},
     TITLE = {Recurrence of distributional limits of finite planar graphs},
   JOURNAL = {Electron. J. Probab.},
  FJOURNAL = {Electronic Journal of Probability},
    VOLUME = {6},
      YEAR = {2001},
      NOTE = {No. 23, 13 pp.},
      ISSN = {1083-6489},
   MRCLASS = {82B41 (05C80 52C26 60G50)},
  MRNUMBER = {1873300},
MRREVIEWER = {Olle\ H\"aggstr\"om},
       DOI = {10.1214/EJP.v6-96},
       URL = {https://doi.org/10.1214/EJP.v6-96},
}

@article {BL2010,
    AUTHOR = {Bordenave, Charles and Lelarge, Marc},
     TITLE = {Resolvent of large random graphs},
   JOURNAL = {Random Structures Algorithms},
  FJOURNAL = {Random Structures \& Algorithms},
    VOLUME = {37},
      YEAR = {2010},
    NUMBER = {3},
     PAGES = {332--352},
      ISSN = {1042-9832,1098-2418},
   MRCLASS = {60F99 (05C80 60B20 60J80)},
  MRNUMBER = {2724665},
       DOI = {10.1002/rsa.20313},
       URL = {https://doi.org/10.1002/rsa.20313},
}

@article {BLS2011,
    AUTHOR = {Bordenave, Charles and Lelarge, Marc and Salez, Justin},
     TITLE = {The rank of diluted random graphs},
   JOURNAL = {Ann. Probab.},
  FJOURNAL = {The Annals of Probability},
    VOLUME = {39},
      YEAR = {2011},
    NUMBER = {3},
     PAGES = {1097--1121},
      ISSN = {0091-1798,2168-894X},
   MRCLASS = {05C80 (05C50 15B52 47A10 60B20)},
  MRNUMBER = {2789584},
MRREVIEWER = {Marianna\ Bolla},
       DOI = {10.1214/10-AOP567},
       URL = {https://doi.org/10.1214/10-AOP567},
}

@article {CEH2007,
    AUTHOR = {Cohen-Steiner, David and Edelsbrunner, Herbert and Harer,
              John},
     TITLE = {Stability of persistence diagrams},
   JOURNAL = {Discrete Comput. Geom.},
  FJOURNAL = {Discrete \& Computational Geometry. An International Journal
              of Mathematics and Computer Science},
    VOLUME = {37},
      YEAR = {2007},
    NUMBER = {1},
     PAGES = {103--120},
      ISSN = {0179-5376,1432-0444},
   MRCLASS = {68U05 (55N05)},
  MRNUMBER = {2279866},
       DOI = {10.1007/s00454-006-1276-5},
       URL = {https://doi.org/10.1007/s00454-006-1276-5},
}

@article {CF2016a,
    AUTHOR = {Costa, A. and Farber, M.},
     TITLE = {Large random simplicial complexes, {I}},
   JOURNAL = {J. Topol. Anal.},
  FJOURNAL = {Journal of Topology and Analysis},
    VOLUME = {8},
      YEAR = {2016},
    NUMBER = {3},
     PAGES = {399--429},
      ISSN = {1793-5253,1793-7167},
   MRCLASS = {55U10 (05C80 57M20 60D05)},
  MRNUMBER = {3509567},
MRREVIEWER = {Dirk\ Sch\"utz},
       DOI = {10.1142/S179352531650014X},
       URL = {https://doi.org/10.1142/S179352531650014X},
}

@incollection {CF2016b,
    AUTHOR = {Costa, Armindo and Farber, Michael},
     TITLE = {Random simplicial complexes},
 BOOKTITLE = {Configuration spaces},
    SERIES = {Springer INdAM Ser.},
    VOLUME = {14},
     PAGES = {129--153},
 PUBLISHER = {Springer, [Cham]},
      YEAR = {2016},
      ISBN = {978-3-319-31579-9; 978-3-319-31580-5},
   MRCLASS = {57Q05 (60D05)},
  MRNUMBER = {3615731},
MRREVIEWER = {D.\ Yogeshwaran},
}

@article {DL2021,
    AUTHOR = {Divol, Vincent and Lacombe, Th\'{e}o},
     TITLE = {Understanding the topology and the geometry of the space of
              persistence diagrams via optimal partial transport},
   JOURNAL = {J. Appl. Comput. Topol.},
  FJOURNAL = {Journal of Applied and Computational Topology},
    VOLUME = {5},
      YEAR = {2021},
    NUMBER = {1},
     PAGES = {1--53},
      ISSN = {2367-1726,2367-1734},
   MRCLASS = {62R40 (49Q22 55N31)},
  MRNUMBER = {4224153},
       DOI = {10.1007/s41468-020-00061-z},
       URL = {https://doi.org/10.1007/s41468-020-00061-z},
}

@article {ELZ2002,
    AUTHOR = {Edelsbrunner, Herbert and Letscher, David and Zomorodian,
              Afra},
     TITLE = {Topological persistence and simplification},
      NOTE = {Discrete and computational geometry and graph drawing
              (Columbia, SC, 2001)},
   JOURNAL = {Discrete Comput. Geom.},
  FJOURNAL = {Discrete \& Computational Geometry. An International Journal
              of Mathematics and Computer Science},
    VOLUME = {28},
      YEAR = {2002},
    NUMBER = {4},
     PAGES = {511--533},
      ISSN = {0179-5376,1432-0444},
   MRCLASS = {52B55 (65D18)},
  MRNUMBER = {1949898},
MRREVIEWER = {H.\ W.\ Guggenheimer},
       DOI = {10.1007/s00454-002-2885-2},
       URL = {https://doi.org/10.1007/s00454-002-2885-2},
}

@incollection {El2010,
    AUTHOR = {Elek, G\'abor},
     TITLE = {Betti numbers are testable},
 BOOKTITLE = {Fete of combinatorics and computer science},
    SERIES = {Bolyai Soc. Math. Stud.},
    VOLUME = {20},
     PAGES = {139--149},
 PUBLISHER = {J\'anos Bolyai Math. Soc., Budapest},
      YEAR = {2010},
      ISBN = {978-963-9453-12-8; 978-3-642-13579-8},
   MRCLASS = {05C80 (05E45)},
  MRNUMBER = {2797963},
       DOI = {10.1007/978-3-642-13580-4\_6},
}

@article {FG2010,
    AUTHOR = {Figalli, Alessio and Gigli, Nicola},
     TITLE = {A new transportation distance between non-negative measures,
              with applications to gradients flows with {D}irichlet boundary
              conditions},
   JOURNAL = {J. Math. Pures Appl. (9)},
  FJOURNAL = {Journal de Math\'{e}matiques Pures et Appliqu\'{e}es.
              Neuvi\`eme S\'{e}rie},
    VOLUME = {94},
      YEAR = {2010},
    NUMBER = {2},
     PAGES = {107--130},
      ISSN = {0021-7824,1776-3371},
   MRCLASS = {49Q20 (35K10)},
  MRNUMBER = {2665414},
       DOI = {10.1016/j.matpur.2009.11.005},
       URL = {https://doi.org/10.1016/j.matpur.2009.11.005},
}

@article {FMT2023,
    AUTHOR = {Fraiman, Nicolas and Mukherjee, Sayan and Thoppe, Gugan},
     TITLE = {The shadow knows: empirical distributions of minimum spanning
              acycles and persistence diagrams of random complexes},
   JOURNAL = {Discrete Anal.},
  FJOURNAL = {Discrete Analysis},
      YEAR = {2023},
      NOTE = {Paper No. 2, 18 pp.},
      ISSN = {2397-3129},
   MRCLASS = {60C05 (05C80 60G57)},
  MRNUMBER = {MR4583993},
MRREVIEWER = {Bogdan\ S.\ Chlebus},
       DOI = {10.19086/da.73323},
       URL = {https://doi.org/10.19086/da.73323},
}

@article {HK2019,
    AUTHOR = {Hino, Masanori and Kanazawa, Shu},
     TITLE = {Asymptotic behavior of lifetime sums for random simplicial
              complex processes},
   JOURNAL = {J. Math. Soc. Japan},
  FJOURNAL = {Journal of the Mathematical Society of Japan},
    VOLUME = {71},
      YEAR = {2019},
    NUMBER = {3},
     PAGES = {765--804},
      ISSN = {0025-5645,1881-1167},
   MRCLASS = {60D05 (05C80 05E45 55U10 60C05)},
  MRNUMBER = {3984242},
       DOI = {10.2969/jmsj/79777977},
       URL = {https://doi.org/10.2969/jmsj/79777977},
}

@article {HS2017,
    AUTHOR = {Hiraoka, Yasuaki and Shirai, Tomoyuki},
     TITLE = {Minimum spanning acycle and lifetime of persistent homology in
              the {L}inial-{M}eshulam process},
   JOURNAL = {Random Structures Algorithms},
  FJOURNAL = {Random Structures \& Algorithms},
    VOLUME = {51},
      YEAR = {2017},
    NUMBER = {2},
     PAGES = {315--340},
      ISSN = {1042-9832,1098-2418},
   MRCLASS = {55N35 (05C80)},
  MRNUMBER = {3683365},
MRREVIEWER = {Michael\ Robinson},
       DOI = {10.1002/rsa.20718},
       URL = {https://doi.org/10.1002/rsa.20718},
}

@article {HST2018,
    AUTHOR = {Hiraoka, Yasuaki and Shirai, Tomoyuki and Trinh, Khanh Duy},
     TITLE = {Limit theorems for persistence diagrams},
   JOURNAL = {Ann. Appl. Probab.},
  FJOURNAL = {The Annals of Applied Probability},
    VOLUME = {28},
      YEAR = {2018},
    NUMBER = {5},
     PAGES = {2740--2780},
      ISSN = {1050-5164,2168-8737},
   MRCLASS = {60G55 (55N20 60B10)},
  MRNUMBER = {3847972},
       DOI = {10.1214/17-AAP1371},
       URL = {https://doi.org/10.1214/17-AAP1371},
}

@article {Kah2009,
    AUTHOR = {Kahle, Matthew},
     TITLE = {Topology of random clique complexes},
   JOURNAL = {Discrete Math.},
  FJOURNAL = {Discrete Mathematics},
    VOLUME = {309},
      YEAR = {2009},
    NUMBER = {6},
     PAGES = {1658--1671},
      ISSN = {0012-365X,1872-681X},
   MRCLASS = {05C80 (05C10 58E05)},
  MRNUMBER = {2510573},
MRREVIEWER = {Klas\ Markstr\"{o}m},
       DOI = {10.1016/j.disc.2008.02.037},
       URL = {https://doi.org/10.1016/j.disc.2008.02.037},
}

@article {Kah2014,
    AUTHOR = {Kahle, Matthew},
     TITLE = {Sharp vanishing thresholds for cohomology of random flag
              complexes},
   JOURNAL = {Ann. of Math. (2)},
  FJOURNAL = {Annals of Mathematics. Second Series},
    VOLUME = {179},
      YEAR = {2014},
    NUMBER = {3},
     PAGES = {1085--1107},
      ISSN = {0003-486X,1939-8980},
   MRCLASS = {05C80 (05C10 05E45 20F65 60D05)},
  MRNUMBER = {3171759},
MRREVIEWER = {Dirk\ Sch\"{u}tz},
       DOI = {10.4007/annals.2014.179.3.5},
       URL = {https://doi.org/10.4007/annals.2014.179.3.5},
}

@article {KHMT2024,
    AUTHOR = {Kanazawa, Shu and Hiraoka, Yasuaki and Miyanaga, Jun and
              Tsunoda, Kenkichi},
     TITLE = {Large deviation principle for persistence diagrams of random
              cubical filtrations},
   JOURNAL = {J. Appl. Comput. Topol.},
  FJOURNAL = {Journal of Applied and Computational Topology},
    VOLUME = {8},
      YEAR = {2024},
    NUMBER = {6},
     PAGES = {1649--1700},
      ISSN = {2367-1726,2367-1734},
   MRCLASS = {55N31 (60D05 60F10)},
  MRNUMBER = {4814516},
MRREVIEWER = {Christian\ Hirsch},
       DOI = {10.1007/s41468-023-00161-6},
       URL = {https://doi.org/10.1007/s41468-023-00161-6},
}

@article {Kan2022,
    AUTHOR = {Kanazawa, Shu},
     TITLE = {Law of large numbers for {B}etti numbers of homogeneous and
              spatially independent random simplicial complexes},
   JOURNAL = {Random Structures Algorithms},
  FJOURNAL = {Random Structures \& Algorithms},
    VOLUME = {60},
      YEAR = {2022},
    NUMBER = {1},
     PAGES = {68--105},
      ISSN = {1042-9832,1098-2418},
   MRCLASS = {05E45 (05C80 55U10 60C05 60F25)},
  MRNUMBER = {4340474},
       DOI = {10.1002/rsa.21015},
       URL = {https://doi.org/10.1002/rsa.21015},
}

@article {KT2025,
    AUTHOR = {Kanazawa, Shu and Trinh, Khanh Duy},
     TITLE = {Central limit theorem for linear eigenvalue statistics of the
              adjacency matrices of random simplicial complexes},
   JOURNAL = {Random Structures Algorithms},
  FJOURNAL = {Random Structures \& Algorithms},
    VOLUME = {66},
      YEAR = {2025},
    NUMBER = {1},
      NOTE = {Paper No. e21257, 42 pp.},
      ISSN = {1042-9832,1098-2418},
   MRCLASS = {05C50 (05C80 60B20 60C05 60F05)},
  MRNUMBER = {4827405},
MRREVIEWER = {Yilun\ Shang},
       DOI = {10.1002/rsa.21257},
       URL = {https://doi.org/10.1002/rsa.21257},
}

@article {LM2006,
    AUTHOR = {Linial, Nathan and Meshulam, Roy},
     TITLE = {Homological connectivity of random 2-complexes},
   JOURNAL = {Combinatorica},
  FJOURNAL = {Combinatorica. An International Journal on Combinatorics and
              the Theory of Computing},
    VOLUME = {26},
      YEAR = {2006},
    NUMBER = {4},
     PAGES = {475--487},
      ISSN = {0209-9683},
   MRCLASS = {55M99 (05C80 55U35)},
  MRNUMBER = {2260850},
MRREVIEWER = {Igor Rivin},
       DOI = {10.1007/s00493-006-0027-9},
       URL = {https://doi.org/10.1007/s00493-006-0027-9},
}

@article {LP2016a,
    AUTHOR = {Linial, Nathan and Peled, Yuval},
     TITLE = {On the phase transition in random simplicial complexes},
   JOURNAL = {Ann. of Math. (2)},
  FJOURNAL = {Annals of Mathematics. Second Series},
    VOLUME = {184},
      YEAR = {2016},
    NUMBER = {3},
     PAGES = {745--773},
      ISSN = {0003-486X,1939-8980},
   MRCLASS = {57M15 (05C80 05C82 05D40 60C05)},
  MRNUMBER = {3549622},
MRREVIEWER = {Michael\ S.\ Farber},
       DOI = {10.4007/annals.2016.184.3.3},
       URL = {https://doi.org/10.4007/annals.2016.184.3.3},
}

@incollection {LP2016b,
    AUTHOR = {Linial, Nathan and Peled, Yuval},
     TITLE = {Random simplicial complexes: around the phase transition},
 BOOKTITLE = {A journey through discrete mathematics},
     PAGES = {543--570},
 PUBLISHER = {Springer, Cham},
      YEAR = {2017},
      ISBN = {978-3-319-44478-9; 978-3-319-44479-6},
   MRCLASS = {05C80 (05D40 60C05 60D05)},
  MRNUMBER = {3726613},
MRREVIEWER = {Yilun\ Shang},
}

@article {MW2009,
    AUTHOR = {Meshulam, R. and Wallach, N.},
     TITLE = {Homological connectivity of random {$k$}-dimensional
              complexes},
   JOURNAL = {Random Structures Algorithms},
  FJOURNAL = {Random Structures \& Algorithms},
    VOLUME = {34},
      YEAR = {2009},
    NUMBER = {3},
     PAGES = {408--417},
      ISSN = {1042-9832,1098-2418},
   MRCLASS = {60B15 (05C80 60E10 60E15)},
  MRNUMBER = {2504405},
       DOI = {10.1002/rsa.20238},
       URL = {https://doi.org/10.1002/rsa.20238},
}

@misc{Me2025,
  author        = {M\'esz\'aros, Andr\'as},
  title         = {The persistent homology of the {L}inial--{M}eshulam process},
  year          = {2025},
  eprint        = {2511.18338},
  archivePrefix = {arXiv},
  primaryClass  = {math.PR},
  note   = {arXiv:2511.18338}
}

@book {RS1972,
    AUTHOR = {Reed, Michael and Simon, Barry},
     TITLE = {Methods of modern mathematical physics. {I}. {F}unctional
              analysis},
 PUBLISHER = {Academic Press, New York-London},
      YEAR = {1972},
     PAGES = {xvii+325},
   MRCLASS = {47-02 (81.47)},
  MRNUMBER = {493419},
MRREVIEWER = {P.\ R.\ Chernoff},
}

@article {STY2020,
    AUTHOR = {Skraba, Primoz and Thoppe, Gugan and Yogeshwaran, D.},
     TITLE = {Randomly weighted {$d$}-complexes: minimal spanning acycles
              and persistence diagrams},
   JOURNAL = {Electron. J. Combin.},
  FJOURNAL = {Electronic Journal of Combinatorics},
    VOLUME = {27},
      YEAR = {2020},
    NUMBER = {2},
      NOTE = {Paper No. 2.11, 47 pp.},
      ISSN = {1077-8926},
   MRCLASS = {60C05 (05C80 05E45 60B99 60G70)},
  MRNUMBER = {4245066},
       DOI = {10.37236/8679},
       URL = {https://doi.org/10.37236/8679},
}

@article {ZC2005,
    AUTHOR = {Zomorodian, Afra and Carlsson, Gunnar},
     TITLE = {Computing persistent homology},
   JOURNAL = {Discrete Comput. Geom.},
  FJOURNAL = {Discrete \& Computational Geometry. An International Journal
              of Mathematics and Computer Science},
    VOLUME = {33},
      YEAR = {2005},
    NUMBER = {2},
     PAGES = {249--274},
      ISSN = {0179-5376,1432-0444},
   MRCLASS = {55N99 (55-04 55U10)},
  MRNUMBER = {2121296},
MRREVIEWER = {Donald\ M.\ Davis},
       DOI = {10.1007/s00454-004-1146-y},
       URL = {https://doi.org/10.1007/s00454-004-1146-y},
}

@misc{HKsupp,
    AUTHOR = {Hino, Masanori and Kanazawa, Shu},
     TITLE = {Supplement to ``{Limit Theorems for Persistent Homology of Multiparameter Complex Processes}''},
      YEAR = {2026},
}

\end{document}


\begin{frontmatter}
\title{Supplement to ``Limit Theorems for Persistent Homology of Multiparameter Complex Processes''}
\runtitle{Supplement: Persistent Homology of Multiparameter Complex Processes}

\begin{aug}
\author[A]{\fnms{Masanori}~\snm{Hino}\ead[label=e1]{hino@math.kyoto-u.ac.jp}\orcid{0009-0004-8835-7115}}
\author[B]{\fnms{Shu}~\snm{Kanazawa}\ead[label=e2]{s.kanazawa@qmul.ac.uk}\orcid{0000-0002-3368-1629}}
\address[A]{Department of Mathematics, Graduate School of Science, Kyoto University\printead[presep={,\ }]{e1}}
\address[B]{School of Mathematical Sciences, Queen Mary University of London\printead[presep={,\ }]{e2}}
\end{aug}
\end{frontmatter}

This supplement contains the essential self-adjointness arguments in Section~S1, simplifications of the limiting constant in Section~S2, and auxiliary proofs for convergence of persistence diagrams in Section~S3.
Notation and references to Sections~1--6 and Appendix~A, and to the results and equations therein, are as in the main paper.

\renewcommand{\thesection}{S\arabic{section}}

\section{Essential Self-adjointness}
\label{sec:esa}
In this section, we prove Lemmas~3.2 and~4.2.

\subsection{Proof of Lemma~3.2}
\label{ssec:esa_Borel}
\begin{proof}[Proof of Lemma~3.2]
  Let
  \[
  \cS'_k
  :=\bigl\{[K\subset X,\sg]\in\cS_k\bigm|\sg\in K\bigr\}.
  \]
  Then $\cS'_k$ is both open and closed.
  Indeed, whether an element $\a=[K\subset X,\sg]\in\cS_k$ belongs to $\cS'_k$ is determined by $\a_0=[K\subset X,\sg]_0$.
  Since $d_{\loc}(\a,\a')<1$ implies $\a_0=\a'_0$, both $\cS'_k$ and its complement are open in $\cS_k$.
  
  Set
  \[
  \Q(i):=\{q+ir\mid q,r\in\Q\}.
  \]
  For $\a=[K\subset X,\sg]\in\cS_k\setminus\cS'_k$ and $l\in\N_0$, let $X_l(\a)$ denote the ambient complex of $(K\subset X,\sg)_l$, and set
  \[
  F_{k,l}(\a):=F_k(X_l(\a))\cap F_k(X\setminus K)
  \]
  and
  \[
  C^k_{l,\Q(i)}(\a)
  :=\Span_{\Q(i)}\bigl\{e_{\sg'}\bigm|\sg'\in F_{k,l}(\a)\bigr\}.
  \]
  
  For $s\in\{-1,1\}$, $l,m\in\N_0$ with $l\le m$, and $n\in\N$, define $E_{l,m,n}^s$ as the set of all
  $\a=[K\subset X,\sg]\in\cS_k\setminus\cS'_k$ such that, for every $\sg'\in F_{k,l}(\a)$, there exists $\phi\in C^k_{m,\Q(i)}(\a)$ satisfying
  \[
  \bigl\|e_{\sg'}-(\Dl_k^\up(X\setminus K)+siI)\phi\bigr\|_k
  <\frac1n.
  \]

  We first verify that $E_{l,m,n}^s$ is Borel in $\cS_k$.
  Indeed, by the explicit expression~(2.6) and Lemma~2.5(1), the defining condition of $E_{l,m,n}^s$ is determined by $[K\subset X,\sg]_{m+1}$.
  Hence, $E_{l,m,n}^s$ is a union of sets of the form
  \[
  \bigl\{\a'\in\cS_k\bigm|\a'_{m+1}=[\bm Y,\tau]_{m+1}\bigr\},
  \qquad[\bm Y,\tau]\in\cS_k.
  \]
  Since there are only countably many isomorphism classes of finite rooted simplicial pairs, this union is countable.
  Each such set is a closed ball of radius $2^{-(m+1)}$ centered at $[\bm Y,\tau]$.
  Therefore, $E_{l,m,n}^s$ is Borel.

  Recall that a densely defined symmetric operator $A$ on a Hilbert space $\cH$ is essentially self-adjoint if and only if
  \[
  \overline{\Ran(A+iI)}
  =\overline{\Ran(A-iI)}
  =\cH.
  \]
  Since $(e_{\sg'})_{\sg'\in F_k(X\setminus K)}$ is an orthonormal basis of $C^k_{(2)}(X\setminus K)$, this is equivalent to requiring that every $e_{\sg'}$ can be approximated by cochains of the form $(\Dl_k^\up(X\setminus K)+siI)\phi$, with $\phi\in C_c^k(X\setminus K)$, for both $s\in\{-1,1\}$.

  Every finitely supported cochain is supported on some rooted ball.
  Moreover, $C^k_{m,\Q(i)}(\a)$ is dense in the finite-dimensional space of cochains supported on $F_{k,m}(\a)$, and therefore cochains belonging to these sets suffice for the above approximations.
  Consequently,
  \[
  \cS_k^{\esa}
  =\cS'_k\sqcup\bigcap_{s\in\{-1,1\}}\bigcap_{l\in\N_0}\bigcap_{n\in\N}\bigcup_{m\ge l}E_{l,m,n}^s.
  \]
  The right-hand side is Borel, proving that $\cS_k^{\esa}$ is a Borel subset of $\cS_k$.
\end{proof}

\subsection{Proof of Lemma~4.2}
\label{ssec:esa_PGW}
Let $k\in\N_0$, and let $K\subset X$ be a simplicial pair such that $\deg_X(\sg)<\infty$ for every $\sg\in F_k(X)$.
For a finite subset $S\subset F_k(X)\setminus F_k(K)$, define its \emph{boundary} by
\[
\del S
:=\bigl\{\{\sg,\sg'\}\bigm|\sg\in S,\,\sg'\in\bigl(F_k(X)\setminus F_k(K)\bigr)\setminus S,\,\sg\cup\sg'\in F_{k+1}(X)\bigr\}.
\]
Its \emph{maximum boundary degree} is defined by
\[
d_{\max}(\del S)
:=\sup_{\rho\in F_k(X)\setminus F_k(K)}\bigl|\bigl\{e\in\del S\bigm|\rho\in e\bigr\}\bigr|.
\]

The following deterministic criterion for simplicial pairs is adapted from the essential self-adjointness arguments in~\cite[Proposition~3]{BLS2011} and~\cite[Appendix~A]{LP2016a}.
\begin{lem}\label{lem:suffi_esa}
  Let $k\in\N_0$, and let $K\subset X$ be a simplicial pair such that $\deg_X(\sg)<\infty$ for every $\sg\in F_k(X)$.
  Suppose that there exists an increasing sequence $(S_l)_{l\in\N}$ of finite subsets of $F_k(X)\setminus F_k(K)$ such that
  \[
  \bigcup_{l\in\N}S_l=F_k(X)\setminus F_k(K)
  \quad\text{and}\quad
  \sup_{l\in\N}d_{\max}\bigl(\del S_l\bigr)<\infty.
  \]
  Then $\Dl_k^\up(X\setminus K)$ is essentially self-adjoint.
\end{lem}

\begin{proof}
  Set $L:=\Dl_k^\up(X\setminus K)$, and denote by $L^*$ its adjoint.
  By the standard criterion for essential self-adjointness~\cite[Theorem~VIII.3]{RS1972}, it suffices to show that
  \begin{equation}\label{eq:deficiency}
  \Ker(L^*-si I)=\{0\},
  \qquad s\in\{-1,1\}.
  \end{equation}
  Fix $s\in\{-1,1\}$ and suppose that
  \[
  \phi\in\Dom(L^*),
  \qquad
  L^*\phi=si\phi.
  \]
  Fix an orientation of each $k$-simplex in $X$, and let $(e_\sg)_{\sg\in F_k(X)\setminus F_k(K)}$ be the resulting canonical orthonormal basis of $C_{(2)}^k(X\setminus K)$.

  For $\sg,\sg'\in F_k(X)\setminus F_k(K)$, define the \emph{flow from $\sg$ to $\sg'$} by
  \[
  J(\sg,\sg')
  :=-s\Im\Bigl(\phi(\sg)\la Le_\sg,e_{\sg'}\ra_k\overline{\phi(\sg')}\Bigr)
  =-s\Im\bigl\la L(\phi(\sg)e_\sg),\phi(\sg')e_{\sg'}\bigr\ra_k.
  \]
  By Lemma~2.5(1) and~(2.6), for every $\sg,\sg'\in F_k(X)\setminus F_k(K)$,
  \[
  \la Le_\sg,e_{\sg'}\ra_k
  =\la Le_{\sg'},e_\sg\ra_k\in\R.
  \]
  Hence, $J(\sg,\sg')=-J(\sg',\sg)$.
  In particular, $J(\sg,\sg)=0$.
  Moreover, for every $\sg\in F_k(X)\setminus F_k(K)$,
  \begin{align*}
  \sum_{\substack{\sg'\in F_k(X)\setminus F_k(K)\colon\\\sg'\ne\sg}}J(\sg,\sg')
  &=-s\Im\left(\sum_{\sg'\in F_k(X)\setminus F_k(K)}\bigl\la L\bigl(\phi(\sg)e_\sg\bigr),\phi(\sg')e_{\sg'}\bigr\ra_k\right)\\
  &=-s\Im\bigl\la L(\phi(\sg)e_\sg),\phi\bigr\ra_k\\
  &=-s\Im\bigl\la\phi(\sg)e_\sg,L^*\phi\bigr\ra_k\\
  &=-s\Im\bigl\la\phi(\sg)e_\sg,si\phi\bigr\ra_k
  =|\phi(\sg)|^2.
  \end{align*}
  Summing the preceding flow identity over $\sg\in S_l$ and using the antisymmetry of $J$, we obtain
  \begin{align*}
  \sum_{\sg\in S_l}|\phi(\sg)|^2
  &=\sum_{\substack{\sg,\sg'\in S_l\\\sg\ne\sg'}}J(\sg,\sg')+\sum_{\substack{\sg\in S_l\\\sg'\in(F_k(X)\setminus F_k(K))\setminus S_l}}J(\sg,\sg')\\
  &=\sum_{\substack{\sg\in S_l\\\sg'\in(F_k(X)\setminus F_k(K))\setminus S_l}}J(\sg,\sg')\\
  &=\sum_{\substack{\{\sg,\sg'\}\in\del S_l\\\sg\in S_l}}J(\sg,\sg').
  \end{align*}
  Here, the second equality follows from the antisymmetry $J(\sg,\sg')=-J(\sg',\sg)$, and the last equality follows because $J(\sg,\sg')\neq0$ only when $\sg\cup\sg'\in F_{k+1}(X)$.
  
  For each $l\in\N$, define the inner and outer vertex boundaries of $S_l$ by
  \[
  \del^-S_l
  :=\bigl\{\sg\in S_l\bigm|
  \text{$\{\sg,\sg'\}\in\del S_l$ for some $\sg'$}\bigr\}
  \]
  and
  \[
  \del^+S_l
  :=\bigl\{\sg'\notin S_l\bigm|\text{$\{\sg,\sg'\}\in\del S_l$ for some $\sg$}\bigr\}.
  \]
  Since $|\la Le_{\sg'},e_\sg\ra_k|=1$ on $\del S_l$, the Cauchy--Schwarz inequality and the assumption $D:=\sup_{l\in\N}d_{\max}\bigl(\del S_l\bigr)<\infty$ yield
  \begin{align*}
  \sum_{\sg\in S_l}|\phi(\sg)|^2
  &\le\sum_{\substack{\{\sg,\sg'\}\in\del S_l\\\sg\in S_l}}|\phi(\sg)|\,|\phi(\sg')|\\
  &\le D\left(
    \sum_{\sg\in\del^-S_l}|\phi(\sg)|^2
  \right)^{1/2}
  \left(
    \sum_{\sg'\in\del^+S_l}|\phi(\sg')|^2
  \right)^{1/2}.
  \end{align*}

  Since $\deg_X(\sg)<\infty$ for every $\sg\in F_k(X)$, each $k$-simplex has only finitely many adjacent $k$-simplices.
  Together with the fact that $(S_l)_{l\in\N}$ exhausts $F_k(X)\setminus F_k(K)$, this implies that every finite subset $A\subset F_k(X)\setminus F_k(K)$ is disjoint from both $\del^-S_l$ and $\del^+S_l$ for all sufficiently large $l$.
  Indeed, once $S_l$ contains $A$ and all its adjacent $k$-simplices in $F_k(X)\setminus F_k(K)$, no element of $A$ lies on either boundary.
  Since $\phi\in C_{(2)}^k(X\setminus K)$, it follows that
  \[
  \sum_{\sg\in\del^-S_l}|\phi(\sg)|^2\longrightarrow0
  \quad\text{and}\quad
  \sum_{\sg\in\del^+S_l}|\phi(\sg)|^2\longrightarrow0.
  \]
  On the other hand,
  \[
  \sum_{\sg\in S_l}|\phi(\sg)|^2\longrightarrow\|\phi\|_k^2.
  \]
  Consequently, $\|\phi\|_k=0$, which proves~\eqref{eq:deficiency}.
\end{proof}

\begin{proof}[Proof of Lemma~4.2]
  Set
  \[
  T:=\PGW_k(c)
  \quad\text{and}\quad
  K:=\PGW_k(c,\bm\theta).
  \]
  We verify almost surely the deterministic condition in Lemma~\ref{lem:suffi_esa}, following the argument in~\cite[Appendix~A]{LP2016a}.

  Let $M$ be a $\Poisson(c)$ random variable.
  Choose $a\in\N$ sufficiently large that
  \[
  (k+1)\E\bigl[M\1_{\{M>a\}}\bigr]<1.
  \]
  For each $\sg\in F_k(T)$, let $m_\sg$ denote the number of $(k+1)$-simplices generated by $\sg$, and call $\sg$ \emph{exceptional} if $m_\sg>a$.
  Conditional on $\sg$ being exceptional, the expected number of exceptional child $k$-simplices of $\sg$ is
  \[
  \E[(k+1)M\mid M>a]\,\P(M>a)
  =(k+1)\E\bigl[M\1_{\{M>a\}}\bigr]<1.
  \]
  Thus, the Galton--Watson process obtained by successively following exceptional child $k$-simplices is subcritical.
  Consequently, the exceptional descendant cluster starting from any fixed $k$-simplex is almost surely finite.
  Since $T$ is almost surely countable, this holds simultaneously for every $k$-simplex.

  Starting from the root, at each stage we expand every boundary $k$-simplex and, whenever this produces an exceptional child $k$-simplex, adjoin its entire exceptional descendant cluster.
  Repeating this procedure yields an increasing sequence $(T_l)_{l\in\N}$ of finite $k$-rooted subtrees such that $T_l\uparrow T$ and every boundary $k$-simplex of $T_l$ is nonexceptional.
  Set
  \[
  S_l:=F_k(T_l)\setminus F_k(K),
  \qquad l\in\N.
  \]
  Then $(S_l)_{l\in\N}$ is an increasing sequence of finite subsets exhausting $F_k(T)\setminus F_k(K)$, and the rooted tree structure gives
  \[
  d_{\max}(\del S_l)\le(k+1)a
  \]
  for every $l\in\N$.
  Therefore, Lemma~\ref{lem:suffi_esa} shows that $\Dl_k^\up(T\setminus K)$ is essentially self-adjoint.
\end{proof}

\section{Simplifications of the Limiting Constant}
\label{app:simpler_hk}
We give two simplifications of the finite-dimensional optimization defining $h_k(c,\bm\theta)$.
The first reduces it to a scalar problem under additional conditions on $c$ or $\bm\theta$.
The second shows that the least fixed point attains the maximum when $c$ exceeds a threshold depending only on $k$.
We use the notation introduced in Section~4.1.

\paragraph*{A scalar reduction}
Under the conditions below, all but at most one coordinate of a fixed point of $\Phi_{k,c,\bm\theta}$ are determined explicitly.
The remaining fixed-point equation and the value of $\Psi_{k,c,\bm\theta}$ at its solutions then reduce the optimization defining $h_k(c,\bm\theta)$ to a scalar problem.
\begin{prop}\label{prop:hk_scalar}
  Let $k\in\N_0$, $c\ge0$, and $\bm\theta=(\theta_j)_{j=0}^{k+1}\in[0,1]^{k+2}$.
  Suppose that
  \[
  c\le1
  \quad\text{or}\quad
  \theta_j=1\text{ for every $0\le j<k$.}
  \]
  When $k=0$, the second condition is vacuous, so no restriction on $c$ or $\bm\theta$ is required.
  Set
  \[
  Q:=\prod_{j=0}^{k-1}\theta_j^{\binom{k+1}{j+1}},
  \quad
  \lm:=c\prod_{j=0}^{k-1}\theta_j^{\binom{k+1}j},
  \quad\text{and}\quad
  \eta:=1-\theta_k,
  \]
  where empty products are understood to be $1$, and define
  \[
  \psi(v):=\eta v(1-\log v)+\frac{\lm}{k+2}(1-\eta v)^{k+2},
  \qquad v\in(0,1].
  \]
  Let $v_{\min}$ and $v_{\max}$ denote the smallest and largest solutions in $(0,1]$ of
  \begin{equation}\label{eq:scalar_fixed}
  v=\exp\bigl(-\lm(1-\eta v)^{k+1}\bigr).
  \end{equation}
  Then
  \begin{equation}\label{eq:hk_scalar}
  h_k(c,\bm\theta)
  =1-Q-\frac c{k+2}+Q\max\{\psi(v_{\min}),\psi(v_{\max})\}.
  \end{equation}
  Moreover, if $c\le1$, then $\Phi_{k,c,\bm\theta}$ has a unique fixed point and~\eqref{eq:scalar_fixed} has a unique solution.
  With $v^*:=v_{\min}=v_{\max}$,
  \begin{equation}\label{eq:hk_small_c}
  h_k(c)-h_k(c,\bm\theta)=Q\{1-\psi(v^*)\}.
  \end{equation}
\end{prop}
\begin{proof}
  Let $\xi'\in\cT_k$ be the type that equals $1$ on every nonempty proper face of $[k+1]$ and equals $0$ on $[k+1]$ itself.
  Let $u=\bigl(u^{(\xi)}\bigr)_{\xi\in\cT_k}\in[0,1]^{\cT_k}$ be a fixed point of $\Phi_{k,c,\bm\theta}$.
  The first step is to show that all its coordinates except $u^{(\xi')}$ are determined by
  \begin{equation}\label{eq:simple_fix}
  u^{(1_k)}=0
  \quad\text{and}\quad
  u^{(\xi)}=1\text{ for every }\xi\notin\{1_k,\xi'\}.
  \end{equation}
  The first equality follows directly from~(4.1).
  The same definition gives $u^{(\xi)}=1$ whenever $\theta^{(\xi)}=0$ and $\xi\ne1_k$.

  Suppose first that $c\le1$.
  It remains to show that $u^{(\xi)}=1$ for every $\xi\notin\{1_k,\xi'\}$ with $\theta^{(\xi)}>0$.
  These are precisely the types with $\theta^{(\xi)}>0$ having a nonempty proper face labeled $0$.
  If there are no such types, there is nothing to prove.
  Otherwise, set
  \[
  m:=\min\bigl\{u^{(\xi)}\bigm|\theta^{(\xi)}>0,\ \xi\notin\{1_k,\xi'\}\bigr\}.
  \]
  Since every coordinate lies in $[0,1]$, the desired conclusion follows once $m=1$ is proved.

  Choose $\xi$ with $u^{(\xi)}=m$ and a nonempty proper face $\rho\subsetneq[k+1]$ with $\xi(\rho)=0$.
  We bound each product appearing in~(4.1) for this coordinate.
  Terms with $\theta^{(\zeta)}=0$ do not contribute to the fixed-point equation.
  It therefore suffices to consider $\zeta\in\cT_{k+1}$ satisfying $\zeta\setminus\{k+2\}=\xi$ and $\theta^{(\zeta)}>0$.
  Choose $j\in[k+1]\setminus\rho$.
  The other $k$-face $[k+2]\setminus\{j\}$ also contains $\rho$ as a proper face labeled $0$.
  Its type satisfies $\theta^{(\zeta\setminus\{j\})}>0$, so its coordinate is among those defining $m$.
  Hence,
  \[
  \prod_{i=1}^{k+1}\bigl(1-u^{(\zeta\setminus\{i\})}\bigr)
  \le1-u^{(\zeta\setminus\{j\})}
  \le1-m.
  \]
  Since
  \[
  \sum_{\substack{\zeta\in\cT_{k+1}\colon\\\zeta\setminus\{k+2\}=\xi}}
  \frac{\theta^{(\zeta)}}{\theta^{(\xi)}}=1,
  \]
  the fixed-point equation gives
  \[
  m=u^{(\xi)}
  =\Phi_{k,c,\bm\theta}^{(\xi)}(u)
  \ge\exp\bigl(-c(1-m)\bigr).
  \]
  If $m<1$, then $c\le1$ implies
  \[
  \exp\bigl(-c(1-m)\bigr)\ge\exp\bigl(-(1-m)\bigr)>m,
  \]
  a contradiction.
  Thus $m=1$, proving that every coordinate other than $u^{(1_k)}$ and $u^{(\xi')}$ equals $1$.
  Since $u^{(1_k)}=0$, only $u^{(\xi')}$ can remain undetermined.
  
    If instead $\theta_j=1$ for every $0\le j<k$, then $\theta^{(\xi)}=0$ for every $\xi\notin\{1_k,\xi'\}$, so~\eqref{eq:simple_fix} follows directly from~(4.1).
  Thus, under either assumption, only the coordinate $u^{(\xi')}$ remains to be determined.
  We first determine this coordinate and then evaluate $\Psi_{k,c,\bm\theta}$ at the resulting fixed points.

  Set $v:=u^{(\xi')}$.
  If $\theta^{(\xi')}=Q\eta=0$, then~(4.1) gives $v=1$.
  In this case, all coordinates of the fixed point are determined.
  Suppose now that $\theta^{(\xi')}=Q\eta>0$.
  We derive the scalar fixed-point equation for $v$.
  Let $\zeta\in\cT_{k+1}$ satisfy $\zeta\setminus\{k+2\}=\xi'$ and $\theta^{(\zeta)}>0$.
  A nonzero product in the fixed-point equation requires every face containing the new vertex $k+2$ and having dimension less than $k$ to be labeled $1$.
  Indeed, any such face labeled $0$ is contained in a $k$-face whose type is neither $1_k$ nor $\xi'$.
  By~\eqref{eq:simple_fix}, its coordinate is $1$, making the product zero.

  For each $0\le j<k$, there are $\binom{k+1}j$ $j$-faces containing the new vertex.
  Requiring their labels to equal $1$ contributes
  \[
  \prod_{j=0}^{k-1}\theta_j^{\binom{k+1}j}
  \]
  to the weight $\theta^{(\zeta)}/\theta^{(\xi')}$.
  With these labels fixed, each of the other $k$-faces has type $1_k$ or $\xi'$ according as its own label is $1$ or $0$.
  Summing over these two labels gives
  \[
  \theta_k(1-0)+(1-\theta_k)(1-v)=1-\eta v.
  \]
  There are $k+1$ such $k$-faces, and the weights are products of the contributions from individual labels.
  Hence,
  \[
  c\sum_{\substack{\zeta\in\cT_{k+1}\colon\\\zeta\setminus\{k+2\}=\xi'}}
  \frac{\theta^{(\zeta)}}{\theta^{(\xi')}}
  \prod_{i=1}^{k+1}\bigl(1-u^{(\zeta\setminus\{i\})}\bigr)
  =\lm(1-\eta v)^{k+1}.
  \]
  Here, summing over the label of $[k+2]$ itself contributes $\theta_{k+1}+(1-\theta_{k+1})=1$, since this label does not affect the restricted types.
  The fixed-point equation for $\xi'$ therefore becomes
  \[
  v=\exp\bigl(-\lm(1-\eta v)^{k+1}\bigr),
  \]
  which is~\eqref{eq:scalar_fixed}.
  Conversely, every solution of~\eqref{eq:scalar_fixed} gives a fixed point by setting $u^{(\xi')}=v$ and assigning the other coordinates according to~\eqref{eq:simple_fix}.

  We next evaluate the three sums in~(4.2).
  By the definition of $Q$, the types having a nonempty proper face labeled $0$ have total weight $1-Q$.
  Their corresponding coordinates equal $1$ by~\eqref{eq:simple_fix}.
  Since $u^{(1_k)}=0$ and $\theta^{(\xi')}=Q\eta$,
  \[
  \sum_{\xi\in\cT_k}\theta^{(\xi)}u^{(\xi)}
  =1-Q+Q(1-\eta)\cdot0+Q\eta u^{(\xi')}
  =1-Q+Q\eta v.
  \]

  In each of the remaining sums, a nonzero contribution requires every nonempty face of $[k+2]$ of dimension less than $k$ to be labeled $1$.
  Indeed, any such face labeled $0$ is contained in a $k$-face other than $[k+1]$, whose coordinate equals $1$ by~\eqref{eq:simple_fix}, making the product zero.
  The contribution from these labels, multiplied by $c$, is
  \[
  c\prod_{j=0}^{k-1}\theta_j^{\binom{k+2}{j+1}}=Q\lm,
  \]
  by $\binom{k+2}{j+1}=\binom{k+1}{j+1}+\binom{k+1}j$.
  With these labels fixed, summing over the label of a $k$-face replaces its factor $u$ by $\theta_k\cdot0+(1-\theta_k)v=\eta v$, and its factor $1-u$ by $1-\eta v$, as above.
  The second sum contains one factor $u$ and $k+1$ factors $1-u$, whereas the third contains $k+2$ factors $1-u$.
  Consequently,
  \begin{align*}
  c\sum_{\zeta\in\cT_{k+1}}\theta^{(\zeta)}u^{(\zeta\setminus\{k+2\})}\prod_{i=1}^{k+1}\bigl(1-u^{(\zeta\setminus\{i\})}\bigr)
  &=Q\lm\eta v(1-\eta v)^{k+1},\\
  c\sum_{\zeta\in\cT_{k+1}}\theta^{(\zeta)}\prod_{i=1}^{k+2}\bigl(1-u^{(\zeta\setminus\{i\})}\bigr)
  &=Q\lm(1-\eta v)^{k+2}.
  \end{align*}
  Substituting these identities into~(4.2) gives
  \[
  \Psi_{k,c,\bm\theta}(u)
  =1-Q+Q\eta v+Q\lm\eta v(1-\eta v)^{k+1}-\frac c{k+2}+\frac{Q\lm}{k+2}(1-\eta v)^{k+2}.
  \]
  If $Q\eta=0$, this expression reduces to
  \[
  h_k(c,\bm\theta)=1-Q-\frac c{k+2}+\frac{Q\lm}{k+2}.
  \]
  If $Q=0$, the product $Q\psi$ vanishes; if $\eta=0$, then $\psi(v)=\lm/(k+2)$ for every $v$.  Thus, $Q\psi$ is constant with value $Q\lm/(k+2)$.
  Thus,~\eqref{eq:hk_scalar} holds in this case.
  If $Q\eta>0$, using $-\log v=\lm(1-\eta v)^{k+1}$ gives
  \[
  \Psi_{k,c,\bm\theta}(u)=1-Q-\frac c{k+2}+Q\psi(v).
  \]
  Thus, maximizing $\Psi_{k,c,\bm\theta}$ over the fixed points of $\Phi_{k,c,\bm\theta}$ reduces to maximizing $\psi$ over the solutions of~\eqref{eq:scalar_fixed}.
  
  To obtain~\eqref{eq:hk_scalar}, it remains to show that this maximum is attained at the smallest or largest solution.
  Set
  \[
  J(v):=\log v+\lm(1-\eta v)^{k+1},
  \qquad v\in(0,1].
  \]
  Its zeros are precisely the solutions of~\eqref{eq:scalar_fixed}, and
  \[
  \psi'(v)=-\eta J(v),
  \qquad
  J'(v)=\frac{1-\lm(k+1)\eta v(1-\eta v)^k}{v}.
  \]
  If $\lm=0$, then $J(v)=\log v$, while if $\eta=0$, then $J(v)=\log v+\lm$.
  In either case, $J$ has a unique zero, so the assertion follows.

  Suppose that $\lm\eta>0$.
  The equation $J'(v)=0$ is equivalent to
  \begin{equation}\label{eq:dJ_dv_0}
  v(1-\eta v)^k=\frac1{\lm(k+1)\eta}.
  \end{equation}
  For $k\ge1$, the derivative of the left-hand side is
  \[
  (1-\eta v)^{k-1}\bigl(1-(k+1)\eta v\bigr).
  \]
  Since $1-\eta v>0$ on $(0,1)$, this derivative is positive before $1/((k+1)\eta)$ and negative after it.
  Thus, the left-hand side of~\eqref{eq:dJ_dv_0} is strictly increasing and then possibly strictly decreasing on $(0,1)$, and it takes any prescribed value at most twice.
  For $k=0$, the left-hand side is simply $v$.
  Consequently, $J'$ has at most two zeros in $(0,1)$, so Rolle's theorem implies that $J$ has at most three distinct zeros in $(0,1]$.

  If $J$ has at most two zeros, every zero is either $v_{\min}$ or $v_{\max}$, and there is nothing further to prove.
  Suppose that $J$ has three distinct zeros $v_1<v_2<v_3$.
  By Rolle's theorem, $J'$ has a zero in each of $(v_1,v_2)$ and $(v_2,v_3)$ and no other zeros in $(0,1)$.
  In particular, $J'(v_1)$ and $J'(v_2)$ are nonzero, so $J$ changes sign at both $v_1$ and $v_2$.
  Since $\lim_{v\downarrow0}J(v)=-\infty$, it follows that
  \[
  J(v)>0\text{ on }(v_1,v_2),
  \qquad
  J(v)<0\text{ on }(v_2,v_3).
  \]
  The identity $\psi'=-\eta J$ therefore shows that $\psi$ decreases on $(v_1,v_2)$ and increases on $(v_2,v_3)$.
  Hence, its value at $v_2$ is smaller than its value at either outer zero.
  The maximum is attained at $v_{\min}$ or $v_{\max}$, proving~\eqref{eq:hk_scalar}.

  Finally, suppose that $c\le1$.
  The binomial expansion gives
  \[
  (k+1)z(1-z)^k\le1-(1-z)^{k+1}<1,
  \qquad 0\le z<1.
  \]
  Since $\lm\le c\le1$, taking $z=\eta v$ in the expression for $J'$ yields $J'(v)>0$ for $0<v<1$.
  Together with $\lim_{v\downarrow0}J(v)=-\infty$ and $J(1)=\lm(1-\eta)^{k+1}\ge0$, this shows that~\eqref{eq:scalar_fixed} has a unique solution $v^*$.
  All coordinates other than $u^{(\xi')}$ are determined by~\eqref{eq:simple_fix}.
  The remaining coordinate equals $1$ if $Q\eta=0$ and $v^*$ if $Q\eta>0$.
  Hence, $\Phi_{k,c,\bm\theta}$ also has a unique fixed point.
  When $\bm\theta=\bm0$, we have $\eta=1$ and $v^*=1$, so $\psi(v^*)=1$.
  Applying~\eqref{eq:hk_scalar} and using $h_k(c,\bm0)=h_k(c)$ gives
  \[
  h_k(c)=1-\frac c{k+2}.
  \]
  Subtracting~\eqref{eq:hk_scalar} from this identity yields~\eqref{eq:hk_small_c}.
\end{proof}

\begin{rmk}\label{rmk:flag_scalar}
  For the random flag complex process in degree one, we have $\bm\theta_1(x,y)=(1,\sqrt{x/y},1)$.
  Proposition~\ref{prop:hk_scalar} applies with $Q=1$, $\lm=y$, and $\eta=1-\sqrt{x/y}$.
  At any solution in $(0,1]$ of $v=\exp(-y(1-\eta v)^2)$, the identity $-\log v=y(1-\eta v)^2$ gives
  \begin{align*}
  -\frac y3+\psi(v)
  &=-\frac y3+\eta v+y\eta v(1-\eta v)^2+\frac y3(1-\eta v)^3\\
  &=\eta v-y\eta^2v^2+\frac{2y}{3}\eta^3v^3.
  \end{align*}
  Taking the maximum over the smallest and largest solutions proves~(1.15).
\end{rmk}

\paragraph*{Evaluation at the least fixed point}
We next give conditions under which the maximum defining $h_k(c,\bm\theta)$ is attained at the least fixed point of $\Phi_{k,c,\bm\theta}$.
Since this map is coordinatewise nondecreasing and continuous, its iterates starting from $\bm0$ converge to its least fixed point in coordinatewise order.
Write
\[
u_{\min}=\bigl(u_{\min}^{(\xi)}\bigr)_{\xi\in\cT_k}
:=\lim_{\ell\to\infty}\Phi_{k,c,\bm\theta}^{\,\ell}(\bm0).
\]
For $k\ge1$, let $a_k\in(0,1)$ be the unique solution of
\[
(k+2)(1-a_k)+\bigl(1+(k+1)a_k\bigr)\log a_k=0,
\]
and set

\[
c_1^*:=1,
\qquad
c_{k+1}^*:=\frac{-\log a_k}{(1-a_k)^{k+1}},
\qquad k\ge1.
\]
The proof of the following proposition adapts the scalar comparison of Linial and Peled~\cite[Lemma~4.5 and Appendix~B]{LP2016a} to the present system of fixed-point equations.

\begin{prop}\label{prop:hk_least}
  Let $k\in\N_0$, $c\ge0$, and $\bm\theta\in[0,1]^{k+2}$.
  If $c\le1$ or $c\ge c_{k+1}^*$, then
  \begin{equation}\label{eq:hk_least}
  h_k(c,\bm\theta)=\Psi_{k,c,\bm\theta}(u_{\min}).
  \end{equation}
  In particular, this identity holds for every $c\ge0$ when $k=0$.
\end{prop}
\begin{proof}
  The case $c\le1$ follows from the uniqueness assertion in Proposition~\ref{prop:hk_scalar}.
  For the remaining cases, we compare $\Psi_{k,c,\bm\theta}(u_{\min})$ with $\Psi_{k,c,\bm\theta}(u)$ for an arbitrary fixed point $u$.

  We first rewrite the objective function in a form suitable for this comparison.
  At any fixed point $u$, we have $u^{(1_k)}=0$ and $u^{(\xi)}>0$ for every $\xi\ne1_k$.
  Taking logarithms in~(4.1) and multiplying by $\theta^{(\xi)}$ gives
  \begin{equation}\label{eq:theta_logu}
  c\sum_{\substack{\zeta\in\cT_{k+1}\colon\\\zeta\setminus\{k+2\}=\xi}}\theta^{(\zeta)}\prod_{j=1}^{k+1}\bigl(1-u^{(\zeta\setminus\{j\})}\bigr)
  =-\theta^{(\xi)}\log u^{(\xi)},
  \qquad \xi\ne1_k.
  \end{equation}
  This identity also holds when $\theta^{(\xi)}=0$, since every extension of $\xi$ then has weight zero.
  Multiplying~\eqref{eq:theta_logu} by $u^{(\xi)}$ and summing over $\xi\ne1_k$ yields
  \[
  c\sum_{\zeta\in\cT_{k+1}}\theta^{(\zeta)}u^{(\zeta\setminus\{k+2\})}\prod_{j=1}^{k+1}\bigl(1-u^{(\zeta\setminus\{j\})}\bigr)
  =-\sum_{\substack{\xi\in\cT_k\colon\\\xi\ne1_k}}\theta^{(\xi)}u^{(\xi)}\log u^{(\xi)}.
  \]
  Here, terms whose root type is $1_k$ vanish because $u^{(1_k)}=0$.

  For the sum inside the parentheses in~(4.2), first sum separately over the types $\zeta$ with each fixed restriction $\zeta\setminus\{k+2\}=\xi$.
  Then
  \begin{align*}
  &c\sum_{\zeta\in\cT_{k+1}}\theta^{(\zeta)}\prod_{j=1}^{k+2}\bigl(1-u^{(\zeta\setminus\{j\})}\bigr)\\
  &\quad=c\sum_{\xi\in\cT_k}\sum_{\substack{\zeta\in\cT_{k+1}\colon\\\zeta\setminus\{k+2\}=\xi}}\theta^{(\zeta)}\prod_{j=1}^{k+2}\bigl(1-u^{(\zeta\setminus\{j\})}\bigr)\\
  &\quad=\sum_{\xi\in\cT_k}(1-u^{(\xi)})
  \left(c\sum_{\substack{\zeta\in\cT_{k+1}\colon\\\zeta\setminus\{k+2\}=\xi}}\theta^{(\zeta)}\prod_{j=1}^{k+1}\bigl(1-u^{(\zeta\setminus\{j\})}\bigr)\right)\\
  &\quad=-\sum_{\substack{\xi\in\cT_k\colon\\\xi\ne1_k}}\theta^{(\xi)}(1-u^{(\xi)})\log u^{(\xi)}
  +c\sum_{\substack{\zeta\in\cT_{k+1}\colon\\\zeta\setminus\{k+2\}=1_k}}\theta^{(\zeta)}\prod_{j=1}^{k+1}\bigl(1-u^{(\zeta\setminus\{j\})}\bigr).
  \end{align*}
  The second equality extracts the factor with $j=k+2$, which equals $1-u^{(\xi)}$.
  The last equality uses~\eqref{eq:theta_logu} for $\xi\ne1_k$ and $u^{(1_k)}=0$ for the remaining term.
  Substituting these identities into~(4.2), we obtain
  \begin{equation}\label{eq:Psi_compare}
  \begin{aligned}
  \Psi_{k,c,\bm\theta}(u)
  &=\sum_{\substack{\xi\in\cT_k\colon\\\xi\ne1_k}}
  \theta^{(\xi)}\left\{u^{(\xi)}-u^{(\xi)}\log u^{(\xi)}-\frac{1-u^{(\xi)}}{k+2}\log u^{(\xi)}\right\}-\frac c{k+2}\\*
  &\quad+\frac c{k+2}\sum_{\substack{\zeta\in\cT_{k+1}\colon\\\zeta\setminus\{k+2\}=1_k}}\theta^{(\zeta)}\prod_{j=1}^{k+1}\bigl(1-u^{(\zeta\setminus\{j\})}\bigr)\\
  &=\sum_{\substack{\xi\in\cT_k\colon\\\xi\ne1_k}}\theta^{(\xi)}g_k(u^{(\xi)})-\frac c{k+2}+\frac c{k+2}\sum_{\substack{\zeta\in\cT_{k+1}\colon\\\zeta\setminus\{k+2\}=1_k}}\theta^{(\zeta)}\prod_{j=1}^{k+1}\bigl(1-u^{(\zeta\setminus\{j\})}\bigr),
  \end{aligned}
  \end{equation}
  where
  \[
  g_k(v):=v-v\log v-\frac{1-v}{k+2}\log v
  =v-\frac{1+(k+1)v}{k+2}\log v,
  \qquad v\in(0,1].
  \]
  The last sum in~\eqref{eq:Psi_compare} is nonincreasing in each coordinate and is therefore largest at $u_{\min}$ among all fixed points.
  It remains to compare the values of $g_k$ in the first sum.

  Direct differentiation gives
  \[
  g_k'(v)=\frac{-(k+1)v\log v-(1-v)}{(k+2)v},
  \qquad
  g_k''(v)=\frac{1-(k+1)v}{(k+2)v^2}.
  \]
  For $k=0$, we have $g_0''(v)>0$ on $(0,1)$ and $g_0'(1)=0$, so $g_0$ is decreasing on $(0,1]$.
  Since $u_{\min}\le u$,~\eqref{eq:Psi_compare} gives
  \[
  \Psi_{0,c,\bm\theta}(u_{\min})\ge\Psi_{0,c,\bm\theta}(u)
  \]
  for every fixed point $u$, proving the assertion for $k=0$.

  Suppose that $k\ge1$.
  The expression for $g_k''$ shows that $g_k'$ is strictly increasing on $(0,1/(k+1))$ and strictly decreasing on $(1/(k+1),1)$.
  Since $\lim_{v\downarrow0}g_k'(v)=-\infty$ and $g_k'(1)=0$, the derivative changes sign exactly once in $(0,1)$, from negative to positive.
  Together with $g_k(1)=1$ and $\lim_{v\downarrow0}g_k(v)=\infty$, this shows that $g_k(v)=1$ has a unique solution in $(0,1)$.
  This solution is precisely $a_k$, since the equation $g_k(v)=1$ is equivalent to its defining equation.
  Moreover, $g_k$ is decreasing on $(0,a_k]$ and satisfies $g_k(v)\le1$ on $[a_k,1]$.
  Consequently,
  \begin{equation}\label{eq:gk_compare}
  g_k(s)\ge g_k(t),
  \qquad 0<s\le a_k,\quad s\le t\le1.
  \end{equation}
  Since $u_{\min}\le u$, this inequality gives $g_k(u_{\min}^{(\xi)})\ge g_k(u^{(\xi)})$ whenever $u_{\min}^{(\xi)}\le a_k$.
  In view of~\eqref{eq:Psi_compare}, it therefore suffices to prove that, when $c\ge c_{k+1}^*$,
  \begin{equation}\label{eq:umin_ak_bound}
  u_{\min}^{(\xi)}\le a_k
  \quad\text{for $\theta^{(\xi)}>0$ and $\xi\ne1_k$.}
  \end{equation}

  We establish this bound by showing that iteration from $\bm0$ preserves the corresponding coordinate bounds.
  Let $w\in[0,1]^{\cT_k}$ satisfy $w^{(\xi)}\le a_k$ whenever $\theta^{(\xi)}>0$.
  If $\theta^{(\zeta)}>0$, then $\theta^{(\zeta\setminus\{j\})}>0$ for every $j$, and hence
  \[
  \prod_{j=1}^{k+1}\bigl(1-w^{(\zeta\setminus\{j\})}\bigr)\ge(1-a_k)^{k+1}.
  \]
  For $\theta^{(\xi)}>0$, the identity
  \[
  \sum_{\substack{\zeta\in\cT_{k+1}\colon\\\zeta\setminus\{k+2\}=\xi}}\frac{\theta^{(\zeta)}}{\theta^{(\xi)}}=1
  \]
  therefore gives
  \[
  \Phi_{k,c,\bm\theta}^{(\xi)}(w)
  \le\exp\bigl(-c(1-a_k)^{k+1}\bigr)
  \le a_k,
  \qquad \xi\ne1_k.
  \]
  The last inequality follows from $c\ge c_{k+1}^*$ and the definition of $c_{k+1}^*$.
  Also, $\Phi_{k,c,\bm\theta}^{(1_k)}(w)=0$.
  Thus, $\Phi_{k,c,\bm\theta}$ preserves these bounds.
  Iterating from $\bm0$ and passing to the limit proves~\eqref{eq:umin_ak_bound}.

  Finally, for any fixed point $u$,~\eqref{eq:gk_compare} and~\eqref{eq:umin_ak_bound} imply
  \[
  g_k(u_{\min}^{(\xi)})\ge g_k(u^{(\xi)}),
  \qquad \theta^{(\xi)}>0,\quad \xi\ne1_k.
  \]
  Terms with $\theta^{(\xi)}=0$ do not contribute to the first sum in~\eqref{eq:Psi_compare}, and the last sum is nonincreasing in each coordinate.
  Hence,
  \[
  \Psi_{k,c,\bm\theta}(u_{\min})\ge\Psi_{k,c,\bm\theta}(u).
  \]
  Taking the maximum over all fixed points $u$ proves~\eqref{eq:hk_least}.
\end{proof}

\begin{rmk}
  In degree $0$, the condition on $\theta_j$ for $j<k$ in Proposition~\ref{prop:hk_scalar} is vacuous, and $Q=1$, $\lm=c$, and $\eta=1-\theta_0$.
  That proposition gives
  \[
  h_0(c,\bm\theta)=-\frac c2+\max\{\psi(v_{\min}),\psi(v_{\max})\},
  \]
  where
  \[
  \psi(v)=(1-\theta_0)v(1-\log v)+\frac c2\{1-(1-\theta_0)v\}^2
  \]
  and $v_{\min}$ and $v_{\max}$ are the smallest and largest solutions in $(0,1]$ of
  \[
  v=\exp\bigl(-c\{1-(1-\theta_0)v\}\bigr).
  \]
  If $\theta_0<1$, the smallest scalar solution corresponds to the least fixed point of $\Phi_{0,c,\bm\theta}$.
  Proposition~\ref{prop:hk_least} with $k=0$ therefore shows that the maximum is attained at $v_{\min}$.
  Using $-\log v=c\{1-(1-\theta_0)v\}$, we obtain
  \begin{align*}
  h_0(c,\bm\theta)
  &=-\frac c2+(1-\theta_0)v_{\min}(1-\log v_{\min})+\frac c2\{1-(1-\theta_0)v_{\min}\}^2\\
  &=(1-\theta_0)v_{\min}-\frac c2(1-\theta_0)^2v_{\min}^2.
  \end{align*}
  The same formula holds when $\theta_0=1$, since both sides are zero.
\end{rmk}

\section{Auxiliary Proofs for Convergence of Persistence Diagrams}
\label{supp:auxiliary_PD}

\subsection{Proof of Lemma~6.2}
\begin{proof}
  (1)~This identity follows immediately from Condition~\textup{(C2)}.
  
  (2)~For $0<x<y<\infty$ and all sufficiently large $n$, the product formula~(1.5) for $r_k$ gives
  \[
  \prod_{j=0}^{k+1}\left(\frac{p_j(r_k^\rightarrow(x/n))}{p_j(r_k^\rightarrow(y/n))}\right)^{\binom{k+1}j}=\frac xy.
  \]
  Every factor is positive and at most one, and therefore
  \begin{equation}\label{eq:Dgm_profile_bound}
  \left(\frac xy\right)^{1/\binom{k+1}j}\le\frac{p_j(r_k^\rightarrow(x/n))}{p_j(r_k^\rightarrow(y/n))}\le1.
  \end{equation}
  Taking $n\to\infty$ proves positivity of $\theta_{k,j}$.
  The same product identity shows that the limit for $j=k+1$ exists whenever those for $0\le j\le k$ do, with
  \[
  \theta_{k,k+1}(x,y)
  =\frac{x/y}{\prod_{j=0}^k\theta_{k,j}(x,y)^{\binom{k+1}j}}.
  \]

  To prove continuity, let $0<x<y<\infty$ and $0<x'<y'<\infty$.
  By~\eqref{eq:Dgm_profile_bound}~(applied to the pairs $(x,x')$ and $(y,y')$ with the smaller argument first) and the triangle inequality,
  \begin{align*}
  \left|\log\frac{p_j(r_k^\rightarrow(x/n))}{p_j(r_k^\rightarrow(y/n))}-\log\frac{p_j(r_k^\rightarrow(x'/n))}{p_j(r_k^\rightarrow(y'/n))}\right|
  &=\left|\log\frac{p_j(r_k^\rightarrow(x/n))}{p_j(r_k^\rightarrow(x'/n))}-\log\frac{p_j(r_k^\rightarrow(y/n))}{p_j(r_k^\rightarrow(y'/n))}\right|\\
  &\le\frac{|\log(x/x')|+|\log(y/y')|}{\binom{k+1}j}
  \end{align*}
  for all sufficiently large $n$.
  Passing to the limit proves continuity of $\theta_{k,j}$.
  
  (3)~The product formula~(1.5) for $q_k$ gives
  \[
  \Theta_k(y)=
  \begin{cases}
    \displaystyle\prod_{j=0}^k\theta_{k,j}(y,1)^{\binom{k+1}{j+1}},&0<y<1,\\[2mm]
    1,&y=1,\\[2mm]
    \displaystyle\prod_{j=0}^k\theta_{k,j}(1,y)^{-\binom{k+1}{j+1}},&y>1.
  \end{cases}
  \]
  Existence and positivity follow immediately.
  Continuity away from $1$ follows from that of $\theta_{k,j}$, and continuity at $1$ follows from~\eqref{eq:Dgm_profile_bound}.
  Monotonicity follows from that of $q_k$ and $r_k^\rightarrow$.
\end{proof}

\subsection{Proof of Lemma~6.3}
\begin{proof}
  We first compare the nullities under a common coupling and then identify the limits of their normalized expectations.
  For all sufficiently large $n$, let $(v_\tau)$ be independent uniform random variables on $[0,1]$, indexed by the $(k+1)$-simplices of $K(n)$.
  Let $Y_n(c)$ consist of the complete $k$-skeleton on $[n]$ together with those $(k+1)$-simplices $\tau$ satisfying $v_\tau\le c/n$.
  Thus, $Y_n(c)$ is a $(k+1)$-Linial--Meshulam complex with parameter $p=c/n$.
  Independently of $(v_\tau)$, let $(u_\rho)$ be independent uniform random variables on $[0,1]$, indexed by the nonempty simplices of dimension at most $k$, and set
  \[
  K_n(\bm\theta)
  :=\{\sg\in K(n)\mid\dim\sg\le k,\,u_\rho\le\theta_{\dim\rho}\text{ for every nonempty face $\rho\subset\sg$}\}.
  \]
  Write
  \[
  N_n(c,\bm\theta)
  :=\nullity\bigl(\Dl_k^\up(Y_n(c)\setminus K_n(\bm\theta))\bigr).
  \]
  We construct $Y_n(c')$ and $K_n(\bm\theta')$ using the same variables $(v_\tau)$ and $(u_\rho)$.

  The quantity $N_n(c,\bm\theta)$ is the nullity of the coboundary matrix of $Y_n(c)$ after deleting the columns indexed by the $k$-simplices of $K_n(\bm\theta)$.
  Adding or deleting a row or a column changes this nullity by at most one.
  Under the above coupling, the expected number of differing rows is $\binom n{k+2}|c-c'|/n$.
  For a $k$-simplex $\sg$, its membership in $K_n(\bm\theta)$ and $K_n(\bm\theta')$ can differ only if the labels of at least one nonempty face of $\sg$ differ.
  Since the two labels of a $j$-face differ with probability $|\theta_j-\theta_j'|$, the union bound gives
  \[
  \P\bigl(\sg\in K_n(\bm\theta)\mathbin\triangle K_n(\bm\theta')\bigr)
  \le\sum_{\emptyset\ne\rho\subset\sg}\P\bigl(\1_{\{u_\rho\le\theta_{\dim\rho}\}}\neq\1_{\{u_\rho\le\theta'_{\dim\rho}\}}\bigr)\\
  =\sum_{j=0}^k\binom{k+1}{j+1}|\theta_j-\theta_j'|.
  \]
  Counting the differing rows and columns therefore yields
  \begin{equation}\label{eq:row-column}
  \frac{\E[|N_n(c,\bm\theta)-N_n(c',\bm\theta')|]}{\binom n{k+1}}
  \le\frac{n-k-1}{n(k+2)}|c-c'|
  +\sum_{j=0}^k\binom{k+1}{j+1}|\theta_j-\theta_j'|.
  \end{equation}

  It remains to identify the limits of the normalized expected nullities.
  For $c>0$, the pair $K_n(\bm\theta)\subset Y_n(c)$ has the same law as the pair observed at times $0$ and $1+c/n$ in a multiparameter complex process whose independent simplex weights in~(1.4) have distribution $\Bernoulli(1-\theta_{\dim\rho})$ when $\dim\rho\le k$ and $\Unif([1,2])$ when $\dim\rho=k+1$, and are equal to $\infty$ otherwise.
  For this process, $p_j(0)=\theta_j$ and $p_j(1+c/n)=1$ for $0\le j\le k$, while $p_{k+1}(0)=0$ and $p_{k+1}(1+c/n)=c/n$.
  Therefore, the corresponding limiting parameter vector is $(\theta_0,\theta_1,\ldots,\theta_k,0)$.
  Applying the nullity convergence~(5.7) established in the proof of Theorem~1.1 to this pair, and using Remark~4.1, we obtain
  \[
  \frac{N_n(c,\bm\theta)}{\binom n{k+1}}
  \xrightarrow{\mathrm p}h_k(c,\bm\theta).
  \]
  Since $0\le N_n(c,\bm\theta)\le\binom n{k+1}$, this convergence also holds in $L^1$.
  Consequently,
  \begin{equation}\label{eq:N_n-hk}
  \lim_{n\to\infty}\frac{\E[N_n(c,\bm\theta)]}{\binom n{k+1}}
  =h_k(c,\bm\theta).
  \end{equation}
  The same conclusion holds for $c=0$, since $Y_n(0)$ has no $(k+1)$-simplices almost surely and hence
  \[
  \frac{\E[N_n(0,\bm\theta)]}{\binom n{k+1}}
  =\frac{\E[f_k(K_n(\bm1)\setminus K_n(\bm\theta))]}{\binom n{k+1}}
  =1-\prod_{j=0}^k\theta_j^{\binom{k+1}{j+1}}
  =h_k(0,\bm\theta).
  \]
  The last equality follows directly from~(4.1)--(4.2) with $c=0$.

  Finally, using
  \[
  \bigl|\E[N_n(c,\bm\theta)]-\E[N_n(c',\bm\theta')]\bigr|
  \le\E[|N_n(c,\bm\theta)-N_n(c',\bm\theta')|]
  \]
  and letting $n\to\infty$ in~\eqref{eq:row-column}, we obtain~(6.2) from~\eqref{eq:N_n-hk}.
\end{proof}

\subsection{Proof of Lemma~6.4}
\begin{proof}
  Finite intersections of the quadrants $Q(x,y)$ are again quadrants, and inclusion--exclusion expresses the masses of their set differences in terms of quadrant masses.
  Thus, the convergence of the mean measures and the almost sure assertion follow from~\cite[Proposition~3.4 and Corollary~A.3]{HST2018}.

  For the assertion in probability, choose a countable family of $\sfD$-continuity sets of the form $Q(x,y)$, $0<x<y<\infty$, such that convergence of their masses to the corresponding $\sfD$-masses implies vague convergence to $\sfD$~(a convergence-determining family for $\sfD$); see~\cite[Appendix~A]{HST2018}.
  Vague convergence of the mean measures gives $c_{x,y}=\sfD(Q(x,y))$ for every quadrant in this family.
  The assumed convergence in probability allows us to extract from any subsequence a further subsequence along which all these quadrant masses converge almost surely simultaneously.
  By the choice of the family, $\sfD_n$ converges vaguely to $\sfD$ almost surely along this further subsequence.
  The subsequence characterization of convergence in probability proves the assertion.

  Finally, suppose that $(x,y)\mapsto c_{x,y}$ is continuous.
  Fix $0<x<y<\infty$ and $0<\eps<\min\{x,(y-x)/2\}$.
  Since
  \[
  \overline{Q(x-\eps,y+\eps)}
  \subset Q(x,y)^\circ
  \subset\overline{Q(x,y)}
  \subset Q(x+\eps,y-\eps)^\circ,
  \]
  the Portmanteau inequalities for vague convergence, applied to the mean measures, give
  \[
  c_{x-\eps,y+\eps}
  \le\sfD(Q(x,y)^\circ)
  \le\sfD(\overline{Q(x,y)})
  \le c_{x+\eps,y-\eps}.
  \]
  Letting $\eps\downarrow0$ and using continuity yields $\sfD(Q(x,y))=c_{x,y}$.
\end{proof}

\subsection{Proof of Corollary~1.7}
\begin{proof}
  Fix $p\in[1,\infty)$.
  Since all norms on $\R^2$ are equivalent, the corresponding optimal transport distances are bounded by constant multiples of one another.
  It therefore suffices to use the Euclidean norm in the definition of $\OT_p$, which we do throughout this proof.
  All death times in $\tilde\cX_n$ are finite, and Theorem~1.4 gives $\sfD_k(\Dl\setminus\Dl_\fin)=0$.
  We therefore regard $\sfD_{k,n}$ and $\sfD_k$ as measures on $\Dl_\fin$.
  Every function in $C_c(\Dl_\fin)$ extends by zero to a function in $C_c(\Dl)$, so Theorem~1.3 implies vague convergence on $\Dl_\fin$, both in probability and for the mean measures.
  Moreover, Corollary~1.6 gives
  \begin{equation}\label{eq:Pers_p-L1}
  \Pers_p(\sfD_{k,n})\xrightarrow{L^1}\Pers_p(\sfD_k)<\infty.
  \end{equation}

  The Euclidean distance from $(x,y)\in\Dl_\fin$ to the diagonal equals $c(y-x)$, where $c:=1/\sqrt2$.
  The $p$th total persistence in the sense of~\cite{DL2021} is therefore $c^p\Pers_p$, and~\eqref{eq:Pers_p-L1} gives its convergence as well.
  By~\cite[Theorem~3.7]{DL2021}, convergence in $\OT_p$ is equivalent to vague convergence together with convergence of the $p$th total persistence.
  From any subsequence, we may extract a further subsequence along which both of these convergences hold almost surely, and the cited characterization gives convergence in $\OT_p$ almost surely along it.
  The subsequence characterization of convergence in probability therefore yields
  \[
  \OT_p(\sfD_{k,n},\sfD_k)
  \xrightarrow{\mathrm p}0.
  \]

  To upgrade this to convergence in $L^p$, consider the admissible transport plan $\gm$ that sends all mass of $\sfD_{k,n}$ to the diagonal and supplies all mass of $\sfD_k$ from the diagonal, using orthogonal projection in both cases.
  Explicitly,
  \[
  \gm:=(z\mapsto(z,P_\del z))_\#\sfD_{k,n}+(z\mapsto(P_\del z,z))_\#\sfD_k.
  \]
  Here, $P_\del(x,y):=((x+y)/2,(x+y)/2)$ denotes the orthogonal projection onto the diagonal.
  Its cost gives
  \[
  \OT_p(\sfD_{k,n},\sfD_k)^p
  \le c^p\bigl\{\Pers_p(\sfD_{k,n})+\Pers_p(\sfD_k)\bigr\}.
  \]
  By~\eqref{eq:Pers_p-L1}, the right-hand side converges in $L^1$ and is therefore uniformly integrable.
  Hence, so are the random variables $\OT_p(\sfD_{k,n},\sfD_k)^p$, which converge to zero in probability, and consequently
  \[
  \E\bigl[\OT_p(\sfD_{k,n},\sfD_k)^p\bigr]\longrightarrow0.
  \]

  Finally, integration against the mean measure commutes with taking expectations for nonnegative integrands.
  Together with~\eqref{eq:Pers_p-L1}, this yields
  \[
  \Pers_p\bigl(\E[\sfD_{k,n}]\bigr)
  =\E\bigl[\Pers_p(\sfD_{k,n})\bigr]
  \longrightarrow\Pers_p(\sfD_k).
  \]
  Together with vague convergence of the mean measures on $\Dl_\fin$, another application of~\cite[Theorem~3.7]{DL2021} proves the second assertion.
\end{proof}

\bibliographystyle{spr-ims-number}
\bibliography{references}